\documentclass[twoside,11pt]{article}

\usepackage{graphicx, subfigure}
\usepackage[top=1in,bottom=1in,left=1in,right=1in]{geometry}
\usepackage[sort&compress]{natbib} 
\usepackage{amsfonts, amsmath, amssymb, amsthm, constants, bbm}
\usepackage{comment}
\usepackage{multirow}
\usepackage{booktabs}

\usepackage{cases}
\usepackage{color}
\definecolor{darkred}{RGB}{100,0,0}
\definecolor{darkgreen}{RGB}{0,100,0}
\definecolor{darkblue}{RGB}{0,0,150}

\usepackage{hyperref}
\hypersetup{colorlinks=true, linkcolor=darkred, citecolor=darkgreen, urlcolor=darkblue}
\usepackage{url}

\usepackage{cleveref}[2012/02/15]
\crefformat{footnote}{#2\footnotemark[#1]#3}

\newtheorem{prp}{Proposition}
\newtheorem{lem}{Lemma}

\theoremstyle{definition}
\newtheorem*{rem}{Remark}

\def\beq{\begin{equation}}
\def\eeq{\end{equation}}
\def\beqn{\begin{eqnarray*}}
\def\eeqn{\end{eqnarray*}}
\def\bitem{\begin{itemize}}
\def\eitem{\end{itemize}}
\def\benum{\begin{enumerate}}
\def\eenum{\end{enumerate}}
\def\bmult{\begin{multline*}}
\def\emult{\end{multline*}}
\def\bcenter{\begin{center}}
\def\ecenter{\end{center}}

\newcommand{\prpref}[1]{Proposition~\ref{prp:#1}}

\newcommand{\lemref}[1]{Lemma~\ref{lem:#1}}
\newcommand{\secref}[1]{Section~\ref{sec:#1}}
\newcommand{\figref}[1]{Figure~\ref{fig:#1}}

\DeclareMathOperator*{\argmax}{arg\, max}

\DeclareMathOperator{\trace}{trace}

\DeclareMathOperator{\diag}{diag}

\DeclareMathOperator{\acos}{acos}

\def\cA{\mathcal{A}}
\def\cB{\mathcal{B}}

\def\cN{\mathcal{N}}

\def\cS{\mathcal{S}}
\def\cT{\mathcal{T}}
\def\cU{\mathcal{U}}
\def\cV{\mathcal{V}}

\def\bA{\mathbf{A}}

\def\bI{\mathbf{I}}
\def\bJ{\mathbf{J}}

\def\bW{\mathbf{W}}
\def\bX{\mathbf{X}}

\def\bZ{\mathbf{Z}}

\newcommand\bSigma{\boldsymbol{\Sigma}}

\def\bbB{\mathbb{B}}

\def\bbH{\mathbb{H}}

\def\bbR{\mathbb{R}}

\newcommand{\E}{\operatorname{\mathbb{E}}}
\renewcommand{\P}{\operatorname{\mathbb{P}}}
\newcommand{\Var}{\operatorname{Var}}
\newcommand{\Cov}{\operatorname{Cov}}

\newcommand{\pr}[1]{\mathbb{P}\left(#1\right)}

\newcommand{\var}[1]{\operatorname{Var}\left(#1\right)}
\newcommand{\cov}[1]{\operatorname{Cov}\left(#1\right)}

\def\iid{\stackrel{\rm iid}{\sim}}
\def\Bin{\text{Bin}}

\newcommand{\<}{\langle}
\renewcommand{\>}{\rangle}

\def\eps{\varepsilon}

\DeclareMathOperator{\sign}{sign}

\usepackage{dsfont}

\newcommand{\IND}[1]{\mathbbm{1}_{\{ #1 \}}}

\definecolor{purple}{rgb}{0.4,.1,.9}

\def\mudif{\Delta\mu}

\begin{document}
\thispagestyle{empty}

\title{\bf Detection and Feature Selection \\ in Sparse Mixture Models}

\author{\Large
Nicolas Verzelen\footnote{INRA, UMR 729 MISTEA, F-34060 Montpellier, FRANCE} \
and
Ery Arias-Castro\footnote{Department of Mathematics, University of California, San Diego, USA}
}
\date{}

\maketitle

\noindent {\bf Abstract.}
We consider Gaussian mixture models in high dimensions, focusing on the twin tasks of detection and feature selection.
Under sparsity assumptions on the difference in means, we derive minimax rates for the problems of testing and of variable selection. We find these rates to depend crucially on the knowledge of the covariance matrices and on whether the mixture is symmetric or not. We establish the performance of various procedures, including the top sparse eigenvalue of the sample covariance matrix (popular in the context of Sparse PCA), as well as new tests inspired by the normality tests of \cite{malkovich1973tests}.


\section{Introduction} \label{sec:intro}
Variable (aka feature) selection is a fundamental aspect of regression analysis and classification, particularly in high-dimensional settings where the number of variables exceeds the number of observations.  
The corresponding literature is vast, from the early proposals based on penalizing the number of variables (i.e., the $\ell_0$ norm) \citep{MR0423716, mallows1973some, schwarz78}, to the more recent variants based on convex relaxations (e.g., the $\ell_1$ norm) \citep{Ti1996, CT07, zhu04, Chen_Donoho_Saunders98} and a wide array of alternative approaches, including non-convex relaxations \citep{MR2065194}, greedy methods \citep{Mallat_Zhang93, Tropp04} and methods based on multiple testing \citep{ji2010ups, MR2520682, MR2546395, MR2546396}.
We refer the reader to \citep{MR2319879} and \citep[Chapters~3, 7, 18]{ESL} for additional pointers.

In contrast, variable selection in the context of clustering is at a comparatively infant stage of development, even though clustering is routinely used in high-dimensional settings.   
Also, according to \cite{ESL}:
\begin{quote}
``Specifying an appropriate dissimilarity measure is far more
important in obtaining success with clustering than choice of clustering
algorithm."
\end{quote}
And, of course, choosing a dissimilarity measure is intimately related to weighting the variables, or combinations of variables, according to their importance in clustering the observations.
The literature on variable selection for clustering is indeed much more recent, scarce and ad hoc.   
\cite{chang83} concludes empirically that performing principal component analysis as a preprocessing step to clustering a Gaussian mixture is not necessarily useful.
\cite{raftery06} and \cite{MR2870505} propose a model selection approach, while penalized methods are suggested in \citep{pan07, xie08, wang08, friedman04, witten2010framework}.

We focus here on the emblematic setting of a mixture of two Gaussians in high-dimensions.  Working under the crucial assumption that the difference in means is sparse, we study the cousin problems of mixture detection (i.e., testing whether the difference in means is zero or not) and variable selection (i.e., estimating the support of the difference in means), both when the covariance matrix is known and when it is unknown.  We obtain minimax lower bounds and propose a number of methods which are able to match these bounds.

\subsection{Detection problem}
The first problem that we consider is that of {\em detection} of mixing, specifically, we test the null hypothesis that there is only one component, versus the alternative hypothesis that there are two components, in a sample assumed to come from a Gaussian mixture model.
We assume throughout that the group covariance matrices are identical, and we consider the case where it is known and the case where it is unknown. 
Formally, in the case where it is unknown, we observe $X_1, \dots, X_n \in \bbR^p$ and consider the general  testing problem
\beq \label{h0}
\begin{array}{c}
H_0 : X_1, \dots, X_n \iid \cN(\mu, \bSigma)\ , \quad  \\[.03in]
\text{for some  $\mu \in \bbR^p$ and some $\bSigma \in \bbR^{p \times p}$ psd\ ;} 
\end{array}
\eeq  
versus
\beq\label{h1-0}
\begin{array}{c}
H_1: \dots, X_n \iid \nu \cN(\mu_0, \bSigma) + (1-\nu) \cN(\mu_1, \bSigma)\ , \\[.03in]
\text{for some $\bSigma \in \bbR^{p \times p}$ psd, some $\mu_0 \ne \mu_1 \in \bbR^p$ and some $\nu\in (0,1)$.}
\end{array}
\eeq
(As usual, `psd' stands for `positive semidefinite'.) We are specifically interested in settings where the difference in means is sparse:
\beq \label{mudif}
\mudif := \mu_1 - \mu_0 \text{ is $s$-sparse.}
\eeq
where $1 \le s \le p$ and $\nu$ belongs to $(0,1)$.
(We say that a vector is $s$-sparse if it has at most $s$ nonzero entries.)
In the sequel, we denote $\theta=( \nu,\mu_0,\mu_1, \bSigma)$ the set of parameters with the convention that under the null hypothesis $\mu_1=\mu_0$, so that $\mudif = 0$, and $\nu\in (0,1)$ is arbitrary. We then write $\P_{\theta}$ for the probability distribution of $X_1,\ldots, X_n$.

We note that the model \eqref{h0}-\eqref{h1-0} can be written as 
\beq \label{model} 
\begin{array}{c}
X_i = \mu_0 + (1- \eta_i) \mudif + \bSigma^{1/2} Z_i\ , \\
\text{where $\eta_1, \dots, \eta_n \iid {\rm Bern}(\nu)$ and independent of $Z_1, \dots, Z_n \iid \cN(0, \bI)$\ ,} 
\end{array}
\eeq
where $\bI$ denotes the identity matrix (here in dimension $p$), and the hypothesis testing problem then reads
\beq\label{eq:hypothesis}
H_0: \mudif = 0 \quad \text{versus} \quad H^{\nu}_{1,s}: \mudif \ne 0 \text{ is $s$-sparse.}  
\eeq

For simplicity of exposition:
\bitem \setlength{\itemsep}{0in}
\item We assume that the sparsity $s$ is known.  This is a rather mild assumption (at least in theory) as discussed in \secref{discussion}.  
\item  We assume the parameter $\nu$ is unknown and bounded away from 0 and 1.  When $\nu$ approaches 0 or 1, the problem becomes that of testing for contamination.  Although the two settings are intimately related, treating both would burden the presentation.
\eitem  
  
We consider the testing problem \eqref{h0} vs \eqref{h1-0} in a high-dimensional large-sample context where all the parameters ($p$, $s$, $\Delta\mu$, $\bSigma$)  may depend on $n$.  Unless specified otherwise, all the limits are taken when the sample size increases to infinity, $n \to \infty$.
We adopt a minimax perspective, which consists of quantifying the performance of tests in the worst case sense.

As various testing problems are studied in this manuscript,  the notion of minimax detection rates is first introduced in an abstract way. 
Consider  $H_0: \theta \in \Omega^n_0$ versus $H_1: \theta \in \Omega^n_1$ based on a sample from a distribution belonging to some family $\{\P_\theta : \theta \in \Omega\}$ and define a non-negative function $R$ that satisfies $R(\theta)=0$ for all $\theta\in \Omega_0^n$ and $R(\theta)>0$ for all $\theta\in \Omega_1^n$. Henceforth, $R(\cdot)$ is called the signal-to-noise ratio. In our Gaussian mixture framework, think of $R(\theta)$ as some (pseudo-)norm of $\Delta\mu$. 
Given some number $r_n>0$, define $\Omega_1^n(R,r_n):=\{\theta\in \Omega_1^n:\ R(\theta)\geq r_n\}$, the set of parameters in the alternative that are $r_n$-separated from the null hypothesis. 
Then the worst-case risk of a test $\phi$ for testing $\theta\in \Omega_0^n$ versus $\theta\in \Omega_1^n(R,r_n)$  is the sum of its probabilities of type I and type II errors, maximized over the null set $\Omega^n_0$ and alternative distributions $\mathbb{P}_{\theta}$ whose signal-to-noise ratio is larger than $r_n$, or in formula
\[
\gamma(\phi; \Omega^n_0, \Omega^n_1(R,r_n)) := \sup_{\theta \in \Omega^n_0} \P_\theta(\phi = 1) + \sup_{\theta \in \Omega^n_1(R,r_n)} \P_\theta(\phi = 0)\ .
\]
The rationale behind the introduction of $r_n$ is that in testing problems such as \eqref{h0}-\eqref{h1-0} some distributions in the alternative are arbitrarily close to the null hypothesis so that $\gamma(\phi; \Omega^n_0, \Omega^n_1; R,0)=1$ for any test $\phi$. This is why the probability type II error is maximized over alternatives that are sufficiently separated from the null distribution, which is here quantified as $R(\theta)\geq r_n$. Then the minimax risk for this testing problem is defined as 
\[
\gamma^*(\Omega^n_0, \Omega^n_1(R, r_n)) := \inf_\phi \gamma(\phi;  \Omega^n_0, \Omega^n_1(R, r_n))\ ,
\]
where the infimum is over all possible tests for $H_0$ versus $H_1$. Formally speaking, we consider a sequence of hypotheses indexed by the sample size $n$ and, correspondingly, consider sequences of tests, also indexed by $n$.  Understood as such, $\liminf_{n \to \infty} \gamma^*[\Omega^n_0, \Omega^n_1(R, r_n)] = 1$ is equivalent to saying that, in the large-sample limit, no test does better than random guessing.  When a sequence of tests $\phi_n$ satisfies $\gamma(\phi_n; \Omega^n_0, \Omega^n_1( R, r_n))\to 0$, it is said to be asymptotically powerful.
A real sequence $r_n^*$ is said to be a {\it minimax separation rate} of $H_0$ versus $H_1$ if for any sequence $r_n$ satisfying $r_n\ll r_n^*$, $\gamma^*[\Omega^n_0, \Omega^n_{1}( R, r_n)]\to 1$, while for any sequence $r_n$ satisfying $r_n\gg r_n^*$, $\gamma^*[\Omega^n_0, \Omega^n_{1}(R,r_n)]\to 0$. 
As we shall see in concrete situations, the minimax separation rate $r_n^*$ characterizes the minimal distance between the mixture means to enable reliable mixture detection.  
As is customary, we leave the dependency on $n$ implicit in the sequel.

\medskip\noindent {\bf Contribution.} 
We distinguish between the cases where $\bSigma$ is known or unknown.  The case where $\bSigma$ is diagonal will play a special role, due to the fact that it combines well with the assumption that the mean difference vector $\mudif$ is assumed sparse in the canonical basis of $\bbR^p$.  
We also distinguish between the symmetric setting, where $\nu = 1/2$, and the asymmetric setting, where $\nu \ne 1/2$.

For each situation, we introduce an appropriate signal-to-noise ratio function $R$ and derive the minimax detection rate with an explicit dependency in the sample size $n$, the ambient dimension $p$, the sparsity $s$ of the difference in means $\mudif$, the mixture weight $\nu$. 
We also propose some tests --- some of them new --- which are shown to be minimax rate optimal.
\bitem \setlength{\itemsep}{0in}
\item When the covariance matrix $\bSigma$ is known, the test based on the top eigenvalue of the normalized sample covariance matrix is competitive when $s$ is relatively large; while the test based on the top sparse (in the eigen-basis of $\bSigma$) eigenvalue of the normalized sample covariance matrix is competitive when $s$  is relatively small.
\item When the covariance matrix $\bSigma$ is unknown, we propose some new projection tests based on moments \`a la \cite{malkovich1973tests}, which are shown to achieve the minimax rate.  The detection rates that we obtain for  the projection skewness and kurtosis statistics proposed in \citep{malkovich1973tests} are suboptimal.
\eitem 
Our results are summarized in Tables~\ref{tab:known} and~\ref{tab:unknown}.
Note that when $\bSigma$ is known, the signal-to-noise ratio is measured in terms of the Mahalanobis distance of $\mudif$ from 0 --- see Table~\ref{tab:known} --- while a different measure is used when $\bSigma$ is unknown --- see Table~\ref{tab:unknown}.  We show that using the Mahalanobis distance in the latter setting leads to exponential minimax bounds.  This is detailed in \secref{maha}.

\def\arraystretch{2}
\begin{table}
\caption{Minimax detection rates and near-optimal tests as a function of $s$ when $\bSigma$ is known and $p\geq n$. 
The minimax detection rates are expressed in terms of the signal-to-noise ratio $R_0 = \mudif^{\top}\bSigma^{-1}\mudif$. Here, $\gamma$ denotes any arbitrary constant in $(0,1)$.}
\bigskip
\label{tab:known}
\centering\small
\setlength{\aboverulesep}{2pt}
\setlength{\belowrulesep}{0pt}
\setlength{\tabcolsep}{0.2in}
\begin{tabular}{ccc}
\toprule
{\sc Sparsity regimes} & {\sc Minimax detection rates} & {\sc Near-optimal test} \\ \midrule 
$\displaystyle s\leq \frac{n}{\log(ep/n)}$ & $\displaystyle \left[\frac{s\log(ep/s)}{n}\right]^{1/2}$ &  {\sc Top sparse eigenvalue \eqref{eq:St_known_sparse}} \\
$\displaystyle \frac{n}{\log(ep/n)}\leq s\leq  (np)^{\gamma/2}$ &$\displaystyle \frac{s\log(ep/s)}{n}$&  {\sc Top sparse eigenvalue \eqref{eq:St_known_sparse}} \\ 
$\displaystyle s \ge   \sqrt{np}$ &  $\displaystyle \sqrt{p/n}$ & {\sc Top eigenvalue \eqref{eq:St_known}} \\[.02in]
\bottomrule
\end{tabular}
\end{table}

\def\arraystretch{2}
\begin{table}
\caption{Minimax detection rates and near-optimal tests when $\bSigma$ is unknown.   The minimax detection rates are expressed in terms of the signal-to-noise  $R_1 = \|\mudif\|^4/\mudif^\top \bSigma \mudif$. 
In this summary, we assume that $s \log(e p/s) = o(n)$.
 If this is not the case and $\bSigma$ is unknown, our lower bounds show that the problem is extremely hard.}
\bigskip
\label{tab:unknown}
\centering\small
\setlength{\aboverulesep}{2pt}
\setlength{\belowrulesep}{0pt}
\begin{tabular}{ccc}
\toprule
 &    {\sc Minimax detection rates} & {\sc Near-optimal test} \\ 
\midrule
symmetric ($\nu = 1/2$) &  $\displaystyle \left[\frac{s}{n}\log\left(\frac{ep}{s}\right)\right]^{1/4}$ & {\sc Projection 1st moment \eqref{malk3}} \\ 
asymmetric ($\nu \ne 1/2$) & $\displaystyle \left[\frac{s}{n}\log\left(\frac{ep}{s}\right)\right]^{1/3}$ & {\sc Projection 2nd signed moment \eqref{malk5}} \\[.05in]
\bottomrule
\end{tabular}
\end{table}

\subsection{Variable selection}
The second problem that we consider is that of {\em variable selection}, where the goal is to estimate the support of $\mudif$ in \eqref{mudif} under the mixture model \eqref{h1-0}. 
The support of a vector $v = (v_j)$ is $\{j : v_j \ne 0\}$.  
A problem of particular interest when $s$ is small compared to $p$ --- meaning $s = o(p)$ --- is that of estimating the support of $\mudif$, which corresponds to the variables that are responsible for separating the population into two groups.  This is what we mean by variable selection, and in a setting where the hypothesis testing problem is parameterized by the sample size $n$, we say that a certain estimator $\hat J_n$ is consistent for $J := \{j : \mudif_j \ne 0\}$ (which may depend on $n$) if 
\beq \label{var-select}
\frac{|\hat J_n \triangle J|}{|J|} \to 0\ , \quad n \to \infty\ . 
\eeq
The dependency on $n$ will often be left implicit.

For the problem of variable selection, we work under the assumption that the effective dynamic range of $\mudif$ and the $2s$-sparse Riesz constant of $\bSigma$ are both bounded.
We define the effective dynamic range of a set of real numbers $\{x_j\}$ (possibly organized as a vector) as $\sup_j |x_j|/\inf_{j \in J} |x_j|$, assuming $J := \{j : x_j \ne 0\} \ne \emptyset$.
Given a $p\times p$ positive semidefinite matrix $\bSigma \ne 0$ and an integer $1 \le s\le p$, we define the largest $s$-sparse eigenvalue of $\bSigma$ as $\lambda_s^{\rm max}(\bSigma) = \max_u u^\top \bSigma u$, where the maximum is over $s$-sparse unit vectors $u \in \bbR^p$.  The smallest $s$-sparse eigenvalue of $\bSigma$ is defined analogously, replacing `max' with `min'.
The $s$-sparse Riesz constant of $\bSigma$ is simply $\lambda^{\rm max}_s(\bSigma)/\lambda^{\rm min}_s(\bSigma)$.  Equivalently, it is the supremum of $u^{\top}\bSigma u/v^{\top}\bSigma v$ over all pairs of unit $s$-sparse vectors $u$ and $v$. 

\medskip\noindent {\bf Contribution.}
Since in each case our testing procedure in the sparse setting ($s = o(p)$) is based on maximizing some form of moment over direction vectors which are sparse (in some way made explicit later on), it is natural to use the support of the maximizing direction as an estimator for the support of $\mudif$:
\bitem \setlength{\itemsep}{0in}
\item When $\bSigma$ is known, we show that this estimator is indeed consistent in the sense of \eqref{var-select} at (essentially) the minimax rate for detection.    
\item When $\bSigma$ is unknown, surprisingly, this estimator may be suboptimal.  This leads us to propose nontrivial variants in \eqref{eq:malk3_estimation} (symmetric setting) and  \eqref{asym_estim} (asymmetric setting).  We are able to show that the support estimator \eqref{eq:malk3_estimation} is consistent at (essentially) the minimax rate for detection.
\eitem

\subsection{Consequences for clustering}
We see the problems of detection and variable selection as complementary to the problem of clustering.  We could imagine a work flow where detection is performed first, then variable selection if the test is significant, and then clustering based on the selected variables. 
 To keep this paper concise, we do not provide here an analysis of these multi-step clustering algorithms. 
 See \citep{azizyan13,azizyan2014efficient,jin2014important} for recent results in this direction.

The motivation for performing detection and variable selection first is meaningful because these can be successfully accomplished with a much smaller separation between the components than clustering.   Indeed, consider a Gaussian mixture of the form $\frac12 \cN(0,\bSigma) + \frac12 \cN(\mudif, \bSigma)$.  Even if $\bSigma$ and $\mudif$ are known --- in which case the best clustering method is the rule $\big\{x^\top\bSigma^{-1}\mudif  > \mudif^\top\bSigma^{-1}\mudif/2\big\}$ --- the expected clustering error is at least $\P(\cN(0,1) > \|\bSigma^{-1/2}\mudif\|/2)$, which converges to 0 only if $\|\bSigma^{-1/2}\mudif\| \to \infty$.

\subsection{Methodology, computational issues, and mathematical technique}
\medskip\noindent {\bf Methodology.}
Most of the tests that we propose are novel.
While the first test in Table~\ref{tab:known} is very natural, the second test is new.  It is a close cousin of the sparse eigenvalue \eqref{eq:St_known_sparse-bis}, considered in the sparse PCA literature (see \secref{literature}).  However, the latter appears suboptimal so that our variant brings a nontrivial improvement.  
The tests in Table~\ref{tab:unknown} are new.
They compete with the projection kurtosis \eqref{malk2} and skewness \eqref{malk4} that we adapted from the normality tests of \cite{malkovich1973tests}.  The motivation for introducing new tests is our inability to prove that these kurtosis and skewness tests achieve the minimax rate.  This is because they are based on higher-order moments, which we found harder to control under the null.

\medskip\noindent {\bf Computational issues.}
We emphasize that except for the top eigenvalue, the other test statistics in Tables~\ref{tab:known} and~\ref{tab:unknown} are very hard to compute even for moderate $p$.
We conjecture that no testing procedure with  polynomial  computational complexity is able to achieve the minimax rates of detection. 
When the covariance $\bSigma$ is known, our testing problem shares many similarities with the sparse PCA detection problem for which a gap between optimal and computationally amenable procedures has been established~\citep{berthet}.

Another contribution of this paper is to propose computationally feasible tests:
\bitem \setlength{\itemsep}{0in}
\item We study coordinate-wise methods based on moments.  
\item We study existing convex relaxations to the sparse eigenvalue problem.
\eitem
See Tables~\ref{tab:known-1} and~\ref{tab:unknown-1}.
The tests in Table~\ref{tab:unknown-1} are new and are the coordinate-wise equivalents of the tests appearing in Table~\ref{tab:unknown}.

\def\arraystretch{2}
\begin{table}
\caption{Detection rates achieved by some computationally feasible tests when $\bSigma$ is known. See Section \ref{sec:poly} for precise statements and assumptions. Compared to Table \ref{tab:known}, the rates are at most $\sqrt{s}$ slower than the optimal rates.}
\bigskip
\label{tab:known-1}
\centering\small
\setlength{\aboverulesep}{2pt}
\setlength{\belowrulesep}{0pt}
\setlength{\tabcolsep}{0.2in}
\begin{tabular}{ccc}
\toprule
{\sc Sparsity regimes} & {\sc Detection rates} & {\sc Test} \\ \midrule
$s\leq \sqrt{p/\log(p)}$ & $\displaystyle \left[s^2\frac{\log(p)}{n}\right]^{1/2}$ &  {\sc Maximal canonical variance \eqref{poly-know1}} \\
$s\geq \sqrt{p/\log(p)}$ &  $\sqrt{p/n}$ & {\sc Top eigenvalue \eqref{eq:St_known}} \\
\bottomrule 
\end{tabular}
\end{table}

\def\arraystretch{2}
\begin{table}
\caption{Detection rates achieved by some computationally feasible tests when $\bSigma$ is unknown. See Section \ref{sec:poly} for precise statements and assumptions.  Compared to Table~\ref{tab:unknown}, the rates are respectively at most $s^{1/4}$ and $s^{1/3}$ slower than the optimal rates.
In this summary, we assume that $\log(p) = o(n)$.}
\bigskip
\label{tab:unknown-1}
\centering\small
\setlength{\aboverulesep}{2pt}
\setlength{\belowrulesep}{0pt}
\setlength{\tabcolsep}{0.1in}
\begin{tabular}{cccc}
\toprule
& {\sc detection rates} & {\sc test} \\ 
\midrule
Symmetric ($\nu = 1/2$) &  $\displaystyle \left[\frac{s^4}{n}\log\left(\frac{ep}{s}\right)\right]^{1/4}$ & {\sc coordinatewise 1st moment \eqref{malk3-1}} \\ 
Asymmetric ($\nu \ne 1/2$) & $\displaystyle \left[\frac{s^3}{n}\log\left(\frac{ep}{s}\right)\right]^{1/3}$ & {\sc coordinatewise 2nd signed moment \eqref{malk5-1}} \\
\bottomrule 
\end{tabular}
\end{table}

\medskip\noindent {\bf A note on the mathematical technique.}
Regarding the technical arguments, the derivation of the information lower bounds for the detection problem is typical: we reduce the set of null hypotheses to the standard normal distribution and put a prior on the set of alternatives, and then bound the variance of the resulting likelihood ratio under the null.  The latter amounts to bounding the chi-squared divergence between the reduced null and alternative distributions; see \citep[Th.~2.2]{MR2724359}.  
That said, in the details, the calculations are both complicated and tedious.  
The test statistics that we study are based on sample moments of Gaussian random variables of degree up to 4.  To control these statistics under the null, we use a combination of  chaining \`a la Dudley \citep{MR1385671} and concentration bounds that we derive based on approximations of Gaussian random variables by sums of Rademacher random variables together with concentration bounds for these obtained by \cite{MR2123200}.

\subsection{Closely related literature} \label{sec:literature}
We already cited a number of publications proposing various methods for variable selection in the context of high-dimensional clustering.  None of these papers offers any real theoretical insights on the difficulty of this problem.
In fact, very few mathematical results are available in this area.

Most of them are on the estimation of Gaussian mixture parameters.  
Recent papers in this line of work include \citep{belkin2010polynomial, kalai2012disentangling, hsu2013learning, brubaker2008isotropic}, and references therein.  
These papers focus on designing polynomial time algorithms that work when there is sufficient parameter identifiability, which is often not optimized.
An exception to that is \citep{chaudhuri2009learning}, where a multistage variant of $k$-means is analyzed in the canonical setting of a symmetric mixture of two Gaussians with identity covariance, and showed to match an information-theoretic bound when the centers are at a distance exceeding 1.
We note that there is no assumption of sparsity made in this literature.     

Related to our proposal of coordinate-wise methods presented in \secref{coord}, \cite{chan10} test each coordinate for unimodality and prove variable selection consistency in a nonparametric setting.
Similar in spirit, \cite{jin2014important} propose\footnote{\label{note:after} This work appeared after the initial version of the present paper was made publicly available.} the selection of features based on coordinate-wise Kolmogorov-Smirnov goodness-of-fit testing. Their setting is slightly different from ours
as the number of mixtures in their paper is allowed to be larger than $2$ but the covariance matrix is restricted to be diagonal and the distributions are supposed to be asymmetric. Nevertheless, when specialized to a common framework (two components, diagonal unknown covariance matrix, $\nu\neq 1/2$), their detection rates and ours are the same.
\cite{azizyan13} consider the task of clustering a sparse symmetric mixture of two Gaussians in high-dimensions with identity covariance matrix.  They prove a minimax lower bound for some clustering error, but do not exhibit any method that matches that lower bound.  Instead, they propose a coordinate-wise approach which is almost identical to one of the methods considered by \citep{amini2009high} (see below) and is very similar to what we do in \secref{coord}.  
This work is closely related to what we obtain in \secref{known} (specialized to $\bSigma = \bI$) and in \secref{coord}. 
The same authors propose\cref{note:after} in \citep{azizyan2014efficient} to first learn the parameters of the Gaussian mixture model using \citep{2014arXiv1404.4997H} and then apply sparse linear discriminant analysis.  Their results are not directly comparable to ours as they assume that $\bSigma^{-1}\Delta\mu$ (instead of $\Delta\mu$) is sparse.

Close to our work is the recent literature on sparse principal component analysis, in view of the following expression for the covariance matrix:
\beq \label{cov}
\Cov(X) = \nu (1-\nu) \mudif \mudif^\top  + \bSigma\ .
\eeq
The difference is that, in this line of work, $X_1, \dots, X_n$ are iid centered normal with covariance matrix of the form \eqref{cov}.  We note that most of the work considers the case where $\bSigma$ is known and isotropic.
The most closely related is the work of \cite{berthet} on testing for a leading sparse principal direction.  From them we drew the idea of using the SDP relaxation of \cite{aspremont} for the sparse eigenvalue problem; see \secref{sdp}.  
Also closely related is \citep{amini2009high}, where the authors tackle the problem of variable selection in the same context.  They propose a coordinate-wise approach which selects the coordinates corresponding to the top $s$ largest variances, identical to a preprocessing step in \citep{johnstone2009consistency}.  They also study the SDP method of \cite{aspremont}, but under very strong constraints --- in particular, they assume that $s = O(\log p)$.     
The estimation of the leading principal component(s), which concerns for example \citep{johnstone2009consistency,cai2013optimal,cai2013sparse,birnbaum2013minimax,vu2012minimax,vu2013minimax}, is also closely related.

\begin{rem}
We note that most of the references in the sparse PCA literature assume that $\bSigma = \bI$ in \eqref{cov}.  
This can easily be extended to the case of a diagonal covariance matrix, which is also an important case in our work.
That said, it is important to realize that, even when more general covariance structures are considered --- as in \citep{vu2012minimax,vu2013minimax} --- the parallel with our work is essentially restricted to the case where the covariance matrix is known.  Indeed, once the covariance matrix is unknown, looking for unusually large eigenvalues in the (sample) covariance matrix becomes meaningless in the context of clustering.  
\end{rem}

\subsection{Organization and notation}
The paper is organized as follows.
In \secref{known}, we consider the case where the covariance is known.
In \secref{unknown}, we treat the case where the covariance is unknown, including the special case where it is known to be diagonal.
In \secref{poly} we suggest and study coordinate-wise methods and some relaxations.  We then compare some of them in small numerical experiments.  
We discuss extensions and important issues in \secref{discussion}, such as the case of unknown sparsity, the case of mixture models with different covariances, the case of mixtures with more than two components, and more.
The proofs are deferred to Sections~\ref{sec:lower-proofs} (lower bounds) and~\ref{sec:upper-proofs} (upper bounds).

\medskip\noindent {\bf Notation.}
For an integer $p$, $[p] = \{1, \dots, p\}$.
For a matrix $A \in \bbR^{p \times p}$ and a subset $S \subset [p]$, $A_S$ denotes the principal submatrix of $A$ indexed by $S$. 
For a finite set $S$, $|S|$ denotes its size.
For two vectors $u = (u_j)$ and $v=(v_j)$ in a Euclidean space, $\|u\|$ denotes the Euclidean norm, $\<u,v\>$ the inner product, $\|u\|_\infty = \max_j |u_j|$ the supnorm, and $\|u\|_0$ the cardinality of the support ${\rm supp}(u) := \{j: u_j \ne 0\}$. Finally, $C$, $C_1$, $C_2$, etc, will denote positive constants that may change with each appearance.

\section{Known covariance matrix} \label{sec:known}

In this section, the covariance $\bSigma$ is assumed to be known. The minimax detection rates are expressed with respect to the Mahalanobis distance
\[R_0(\theta) =  \mudif^{\top}\bSigma^{-1}\mudif \ .\]

\subsection{Minimax lower bound}	

Fix a mixing weight $\nu \in (0,1)$ and a sparsity $s$, and consider
\beq \label{omega0}
\Omega_0(\nu) = \big\{\theta = (\nu, \mu, \mu, \bSigma),\ \mu \in \bbR^p,\ \bSigma\text{ psd}\big\},
\eeq
and, for $r_n > 0$ and for a signal-to-noise ratio function $R$,
\beq \label{omega1}
\Omega_1(\nu,  R, r_n) := \big\{\theta = (\nu, \mu_0, \mu_1, \bSigma) : \mu_0, \mu_1 \in \bbR^p \text{ satisfying \eqref{mudif}, }\bSigma\text{ psd}, R(\theta) \ge r_n \big\},
\eeq
where we leave implicit the dependency of $\Omega_1(\nu, R_0, r_n)$ on $s$. 
As the tests considered in this section use the knowledge the covariance matrix $\bSigma$, we also consider for any covariance $\bSigma$, 
\beq \label{omega0-known}
\Omega_0(\nu,\bSigma) = \big\{\theta = (\nu, \mu, \mu, \bSigma),\ \mu \in \bbR^p\big\},
\eeq
and,  fixing a mixing weight $\nu \in (0,1)$, a sparsity $s$ and $r_n>0$, consider
\beq \label{omega1-known}
\Omega_1(\nu, \bSigma, R, r_n) := \big\{\theta = (\nu, \mu_0, \mu_1, \bSigma) : \mu_0, \mu_1 \in \bbR^p \text{ satisfying \eqref{mudif}}, R(\theta) \ge r_n \big\}\ .
\eeq
Then, the minimax detection risk with known variance is defined by
\[
\gamma^*_{\rm known}(\Omega_0(\nu), \Omega_1(\nu, R, r_n)) = \sup_{\bSigma} \inf_\phi \gamma(\phi;  \Omega_0(\nu,\bSigma), \Omega_1(\nu, \bSigma, R, r_n))\ .
\]

In order to emphasize the role of sparsity, we distinguish the sparse and non-sparse settings, corresponding to $s=p$ and $s= o(p)$, respectively.

\begin{prp}\label{prp:lower_test_sparse}
Consider testing \eqref{omega0-known} versus \eqref{omega1-known}.  For any fixed  $\nu\in(0,1)$, we  have ~\\$\liminf \gamma^*_{\rm known}(\Omega_0(\nu), \Omega_1(\nu,R_0,r_n)) = 1$ in the following two cases:
\begin{itemize}
 \item \emph{Non-sparse setting}. Assume $s=p\to\infty$ and  
\[r_n\ll  \sqrt{p/n}\ .\]
\item  \emph{Sparse setting}. Assume $p/s\to\infty$ and
\beq \label{lower_test_sparse1}
r_n\ll  \sqrt{p/n}\ 
\eeq
and 
\beq \label{lower_test_sparse2}
\lim\sup \frac{r_n}{\sqrt{\frac{s}{n}\log(\tfrac{ep}{s})}\vee \frac{s}{n}\log(1+ \sqrt{\tfrac{epn}{s^2}})}< 1\ .
\eeq
\end{itemize}
\end{prp}

\noindent {\bf Remark.} As usual for minimax lower bounds, 
it is sufficient to provide a lower bound on the risk for testing subclasses of $\Omega_0(\nu,\bSigma)$ and $\Omega_1(\nu,\bSigma,R_0,r_n)$. In fact, we reduce the problem to testing $\theta\in  \widetilde{\Omega}_0:= \big\{\theta = (\nu, 0, 0, \bI)\big\}$
against  
\[
\theta\in \widetilde{\Omega}_1(\nu, R_0,r_n):=\big\{(\nu,-(1-\nu)\mu,\nu\mu,\bI),\  \mu\text{ is $s$-sparse}, R_0(\theta)\geq r_n \big\}~.
\]

\subsection{Methodology based on (sparse) principal component analysis}
We now turn to designing tests that are asymptotically powerful just above the lower bound given in \prpref{lower_test_sparse}.
We note that	
the performance bounds for the tests based on \eqref{eq:St_known} and \eqref{eq:St_known_sparse} in Propositions~\ref{prp:upper_test_sparse1} and~\ref{prp:upper_test_sparse2} apply to a general (known) covariance matrix.

Our methodology is based on the expression for the covariance matrix of $X$ displayed in \eqref{cov}.
We standardize the observations to have identity covariance under the null, thus working with $X_\ddag = \bSigma^{-1/2} X$, which satisfies
\[
\bSigma_\ddag := \Cov(X_\ddag) = \bSigma^{-1/2} \Cov(X) \bSigma^{-1/2} = \nu (1-\nu) \mudif_\ddag \mudif_\ddag^\top + \bI \ ,
\]
where $\mudif_\ddag := \bSigma^{-1/2} \mudif$.
Thus $\Cov(X_\ddag)$ is a rank-one perturbation of the identity matrix under the alternative.  Since $\Cov(X_\ddag)$ is unknown, our inference is based on the sample equivalent, which is $\hat\bSigma_\ddag := \bSigma^{-1/2}\hat\bSigma\bSigma^{-1/2}$, where
\[\hat\bSigma := \frac1n \sum_{i=1}^n (X_i -\bar X) (X_i -\bar X)^\top, 
\quad \bar X := \frac1n \sum_{i=1}^n X_i\ ,\]
are the sample covariance matrix and sample mean, respectively.

\bitem
\item When $\mudif$ is not sparse ($s = p$), this  leads us to consider the top eigenvalue of $\hat\bSigma_\ddag$, namely
\beq  \label{eq:St_known}
\hat \lambda_{\bSigma}^{\rm \max} := \max_{\|u\|=1} u^\top \hat \bSigma_\ddag u \ . 
\eeq
We note that the maximizer of \eqref{eq:St_known} is the first principal direction of the standardized observations, and that $\hat \lambda_{\bSigma}^{\rm \max}$ is the variance along that direction.
As we shall see, this test is also competitive when $\mudif$ is moderately sparse.

\item When $\mudif$ is $s$-sparse, we restrict the maximization over the set of vectors that are $s$-sparse in some appropriate basis.  To guide our choice, we notice that $\mudif_\ddag$ is a top eigenvector for $\bSigma_\ddag$, and $\bSigma^{1/2} \mudif_\ddag = \mudif$ is $s$-sparse.  This leads us to the following form of $s$-sparse (top) eigenvalue
\begin{eqnarray}
\hat \lambda_{s, \bSigma}^{\rm \max}& := &  \max_{\|u\| = 1, \, \|\bSigma^{1/2}u\|_0 \le s} u^\top \hat \bSigma_\ddag u \ . 
\label{eq:St_known_sparse}
\end{eqnarray}
We note that the maximizer of \eqref{eq:St_known_sparse} is the first $s$-sparse (after standardization) principal direction of the standardized observations, and that $\hat \lambda_{s,\bSigma}^{\rm \max}$ is the variance along that direction.
\eitem

\begin{rem}
With the notable exception of \eqref{eq:St_known},  all the statistics studied in Sections~\ref{sec:known} and~\ref{sec:unknown} are difficult to compute, which effectively makes them useless in practical settings, which are often high-dimensional. 
For this reason, we leave implicit the critical values of the corresponding tests. The interested reader may obtain their expression by inspecting the proofs of the corresponding propositions. 
\end{rem}

The following performance bound says, roughly, that the test based on \eqref{eq:St_known} is reliable when \eqref{lower_test_sparse1} does not hold.

\begin{prp}\label{prp:upper_test_sparse1}
Consider testing \eqref{omega0-known} versus \eqref{omega1-known} with $\bSigma$ known, $\nu\in(0,1)$ fixed, $s\le p$, and $p\wedge n \to \infty$. 
Let $T$ denote the statistic \eqref{eq:St_known}.
The test  $\phi = \{T \ge 1+ p/n+ 12\sqrt{p/n}\}$ is asymptotically powerful, meaning $\gamma(\phi; \Omega_0(\nu,\bSigma), \Omega_1(\nu,\bSigma,R_0,r_n)) \to 0$,  if the minimum Mahalanobis distance $r_n$ satisfies
\beq \label{upper_test_sparse1}
\liminf r_n\nu (1-\nu)\sqrt{\frac{n}{p}} > C \ ,
\eeq
where $C$ is a universal constant.
\end{prp}
In view of \prpref{lower_test_sparse}, the above test is adaptive to the mixing weight $\nu$ as long as it is fixed.  

\medskip
The following performance bound says, roughly, that the test based on \eqref{eq:St_known_sparse} is reliable when \eqref{lower_test_sparse2} does not hold, and that consistent support estimation is possible with a slightly stronger signal-to-noise ratio. The procedure is also adaptive to $\nu$.

\begin{prp}\label{prp:upper_test_sparse2}
Assume $\bSigma$ is known and that $p\wedge n \to \infty$.  For any sequence $s$ of sparsity, the following results holds. 
\bitem
\item {\em Detection.}  Consider testing \eqref{omega0-known} versus \eqref{omega1-known} with $\nu\in(0,1)$ fixed.  
Let $T_s$ denote the statistic \eqref{eq:St_known_sparse}.
There is a sequence of critical values $t$ such that the test $\phi = \{T_s \ge t\}$ is asymptotically powerful, meaning $\gamma(\phi; \Omega_0(\nu,\bSigma), \Omega_1(\nu,\bSigma,R_0,r_n)) \to 0$, if the minimum  Mahalanobis distance $r_n$ satisfies
\beq \label{upper_test_sparse2}
\liminf \, \frac{\nu(1-\nu)r_n}{\sqrt{\frac{s}{n}\log(\tfrac{ep}{s}) }\vee \frac{s}{n}\log(\tfrac{ep}{s})}> C\ ,
\eeq
where $C$ is a universal constant.	
\item {\em Variable selection.}
Consider the model \eqref{omega1-known}.  Let $\hat u_s$ denote a maximizer of \eqref{eq:St_known_sparse} and let $\hat v_s = \bSigma^{1/2} \hat u_s$.  Then under the slightly stronger condition 
\beq \label{upper_test_sparse1-select}
\nu(1-\nu)\mudif^{\top} \bSigma^{-1}\mudif \gg \sqrt{\frac{s}{n}\log(\tfrac{ep}{s}) }\vee \frac{s}{n}\log(\tfrac{ep}{s})\ ,
\eeq
and the assumption that the effective dynamic range of $\mudif$ and the $2s$-sparse Riesz constant of $\bSigma$ are both bounded, the support of $\hat v_s$ is consistent for the support of $\mudif$.
\eitem
\end{prp}

We note that without a bound on the dynamic range of $\mudif$, its largest entries could overwhelm the smaller (nonzero) ones and make consistent support recovery difficult, or even impossible.  

\medskip

\noindent {\bf Special case: $\bSigma =\bI$.} As a consequence of the remark below \prpref{lower_test_sparse},
 the detection boundary is roughly at 
\[\|\mudif\|^2 \approx  \Big[\sqrt{\frac{s}{n}\log(\tfrac{ep}{s})} \vee \frac{s}{n}\log(\tfrac{ep}{s})\Big] \wedge \sqrt{\frac{p}{n}}\ ,\]
except in the regime where $s\geq n$ and $s\approx \sqrt{np}$ where there is a logarithmic gap between the upper and lower bounds. 
The statistic of choice is \eqref{eq:St_known_sparse}, the top $s$-sparse eigenvalue of $\hat\bSigma$, which is also known to be rate-optimal for the problem of testing for a top principal direction in a spiked Gaussian covariance model \citep{berthet}.

\begin{rem}
In general, the statistic defined in \eqref{eq:St_known_sparse} is not the top $s$-sparse eigenvalue of $\hat\bSigma_\ddag$, which is instead defined as  
\beq \label{eq:St_known_sparse-bis}
\lambda_s^{\rm \max}(\hat\bSigma_\ddag) = 
\max_{\|u\|=1, \, \|u\|_0 \le s} u^\top \hat\bSigma_\ddag u \ .
\eeq
We are only able to show that the statistic \eqref{eq:St_known_sparse-bis} is asymptotically powerful in the following sense, that $\gamma(\phi; \Omega_0(\nu,\bSigma), \Omega_1(\nu,\bSigma,R_1,r_n)) \to 0$ for $R_1(\theta) := \nu(1-\nu) \|\mudif\|^4/\mudif^{\top}\bSigma \mudif$ and $r_n$ satisfying 
\eqref{upper_test_sparse2} for some constant $C>0$. However, this bound is weaker than what we obtain in \prpref{upper_test_sparse2} for \eqref{eq:St_known_sparse}, simply because the function $R_1(\theta)$ is smaller than the Mahalanobis distance $R_0(\theta)$.  Indeed, using the Cauchy-Schwarz inequality
\[
\|\mudif\|^4 = \big[ (\bSigma^{-1/2} \mudif)^\top (\bSigma^{1/2} \mudif) \big]^2 \le \|\bSigma^{-1/2} \mudif\|^2 \|\bSigma^{1/2} \mudif\|^2 = (\mudif^\top \bSigma^{-1} \mudif) (\mudif^\top \bSigma \mudif)\ . 
\]
\end{rem}

\section{Unknown covariance matrix} \label{sec:unknown}
We distinguish between the symmetric case ($\nu = 1/2$) and the asymmetric case ($\nu \ne 1/2$).
In terms of methodology, skewness and kurtosis tests have played a major role in testing for multivariate normality, at least since the seminal work of \cite{MR0397994}.  
Some of these tests are based on estimating the covariance matrix, and therefore are not applicable in high-dimensional settings where $p > n$, at least not without additional assumptions on the covariance matrix.
More malleable approaches are projection tests such as those proposed by \cite{malkovich1973tests}.  We adapt such tests to the sparse setting considered here, and also design new variants to palliate some deficiencies.

\subsection{Symmetric setting} \label{sec:sym}

Consider the case where the covariance matrix is unknown and where the mixture distribution is symmetric, meaning that $\nu = 1/2$. 
The resulting mixture testing problem is more difficult than in the asymmetric setting treated in \secref{asym}.

\subsubsection{Minimax lower bound}

We start with a minimax lower bound with respect to the signal-to-noise ratio 
\beq \label{R1}
R_1(\theta) = \frac{\|\mudif\|^4}{\mudif^\top \bSigma \mudif} \ .
\eeq 
We will see in \prpref{lower_mahalanobis} that the minimax detection rate with respect to the Mahalanobis distance $R_0$ is degenerate in a sparse high-dimensional setting.

\begin{prp}\label{prp:lower_sparse_unknow_variance}
Consider testing \eqref{omega0} versus \eqref{omega1} with $\nu = 1/2$.  Then $\liminf \gamma^*(\Omega_0(\nu), \Omega_1(\nu,R_1,r_n))  = 1$ in the following cases:
\bitem
\item {\em Non-sparse setting:}  
$s=p \to \infty$ and 
\beq \label{eq:lower_unknow_variance}
r_n\ll (p/n)^{1/4}\ ;
\eeq
or $p\gg n$ and
\beq\label{eq:lower_unknow_variance_TGD}
\lim\sup r_n e^{-C p/n}<1 \ ,
\eeq
where $C > 0$ is a universal constant.
\item {\em Sparse setting:}
$p/s\to \infty$ and  
\beq \label{eq:lower_sparse_unknow_variance}
\limsup r_n \left[\frac{s}{n}\log\left(\frac{ep}{s}\right)\right]^{-1/4}   \leq C_1 \ ,
\eeq
where $C_1 > 0$ is a universal constant; or $n \le \frac1{C_2} s \log(e p/s)$ and 
\beq \label{eq:lower_sparse_unknow_variance_TGD}
\lim\sup r_n e^{- C_3 \frac{s}{n}\log\left(\frac{ep}{s}\right)}   \leq 1\ ,
\eeq
where $C_2, C_3  > 0$ are universal constants. 
\eitem
\end{prp}

\begin{rem}
From the above proposition, we deduce that the testing problem becomes extremely difficult when $\zeta:= \frac{s}n \log(ep/s) \to \infty$, in the sense that the minimax detection rate is at least exponentially large with respect to $\zeta$.  A similar phenomenon occurs in other high-dimensional detection problems such as in sparse linear regression~\citep{verzelen}.
\end{rem}

\begin{rem}
Similar to what we do in the proof of \prpref{lower_test_sparse}, we reduce to testing subclasses of hypotheses. 
Specifically, we reduce to testing $\theta\in  \widetilde{\Omega}_0:= \big\{\theta = (\tfrac{1}{2}, 0, 0, \bI)\big\}$
against  
\[\theta\in \widetilde{\Omega}_1(\tfrac{1}{2}, R_1,r_n):=\big\{(\tfrac{1}{2},-\mu,\mu,\bSigma_\mu),\  \mu\text{ is $s$-sparse\ and } R_1(\theta)\geq r_n \big\}\ ,\]
where $\bSigma_\mu := \bI -  \mu \mu^\top$.
Note that, in this testing problem, the variables are centered and $\Cov(X) =\bI$, both under the null and under the alternative.  
\end{rem}

\subsubsection{A classical approach based on the kurtosis}
Unlike in \secref{known}, here the covariance matrix $\cov{X}$, by itself, does not contain in any sensible information to distinguish the null hypothesis from the alternative. It is therefore natural to consider higher order moments of $X$. In the symmetric setting, a traditional approach is the use of a kurtosis test.  \cite{malkovich1973tests} propose a projection test based on the kurtosis for the problem of testing for multivariate normality.  This is easily adapted to the sparse setting.
The resulting test is based on rejecting for {\em small} values of 
\beq \label{malk2}
\min_{\|u\|_0 \le s} \frac{\sum_i \big[u^\top (X_i - \bar X)\big]^4}{\left(\sum_i \big[u^\top (X_i - \bar X)\big]^2\right)^2} \ .
\eeq
We note that \cite{malkovich1973tests} --- who are interested in testing for multivariate normality and do not make sparsity assumptions --- reject for unusually large or small values of the above ratio along a general (meaning, not necessarily sparse) direction $u$.


\begin{rem}
Although the null distribution of \eqref{malk2} depends on the unknown covariance matrix $\bSigma$, it can calibrated by a simple Bonferroni correction, which is possible because \eqref{malk2}  is the minimum over all subsets $S\subset [p]$ of size $s$ of variables which have a null distribution that is independent of $\bSigma$. 
The same applies to \eqref{malk3}, \eqref{malk4} and \eqref{malk5}.  
\end{rem}

\begin{prp} \label{prp:malk2}
Consider testing \eqref{omega0} versus \eqref{omega1} with the assumption  $|\nu-1/2|<\frac{\sqrt{3}}{6}$, and assume that $n \gg [s \log (ep/s)]^2$.
Let $T$ denote the statistic \eqref{malk2}.  There is a sequence of critical values $t$ such that such the test $\phi = \{T \le t\}$ 
is asymptotically powerful, meaning $\gamma(\phi; \Omega_0(\nu), \Omega_1(\nu, R_1, r_n)) \to 0$, if 
\beq\label{eq:upper_malk2}
r_n \gg  \Big[\frac{s}{n}\log\left(\frac{ep}{s}\right)\Big]^{1/4}\vee \Big[\frac{s}{ \sqrt{n}}\log\left(\frac{ep}{s}\right)\Big]\ .
\eeq	
\end{prp}

We see that there is a substantial discrepancy between the performance that we establish for the sparse kurtosis test \eqref{malk2} in \prpref{malk2} and the lower bound obtained in \prpref{lower_sparse_unknow_variance}.  
The issue comes from the control of the numerator in \eqref{malk2}, in that the estimator for the fourth moment has a heavy tail and does not concentrate enough when $s\log(ep/s)$ becomes large.

\subsubsection{A new approach based on the first absolute moment}
Instead of a kurtosis test, which is based on the fourth central moment, we propose a test based on the first central absolute moment in order to palliate the aforementioned issues.  The test rejects for large values of 
\beq \label{malk3}
\max_{\|u\|_0 \le s} \frac{\sum_i \big|u^\top (X_i - \bar X)\big|}{\left(\sum_i \big[u^\top (X_i - \bar X)\big]^2\right)^{1/2}}.
\eeq

\begin{prp} \label{prp:malk3}
Consider testing \eqref{omega0} versus \eqref{omega1} with the assumption $|\nu-1/2|<\frac 1 6$, and assume that $n \gg s \log (ep/s)$.
Let $T$ denote the statistic \eqref{malk3}.
There is a sequence of critical values $t$ such that the test $\phi = \{T \ge t\}$ is asymptotically powerful, meaning  $\gamma(\phi; \Omega_0(\nu), \Omega_1(\nu, R_1, r_n)) \to 0$, if 
\beq\label{eq:upper_malk3}
r_n \gg \left[\frac{s}{n}\log(ep/s)\right]^{1/4}.
\eeq
\end{prp}
Consequently, the test based on \eqref{malk3} achieves the minimax detection boundaries \eqref{eq:lower_unknow_variance} and \eqref{eq:lower_sparse_unknow_variance}. Note that the assumption $n \gg s \log (ep/s)$ is necessary in view of \prpref{lower_sparse_unknow_variance}.

\paragraph{Variable selection}
In regards to variable selection, we are unable to use the statistic \eqref{malk3} (or the original statistic \eqref{malk2}).  To see why, for concreteness, consider the situation where the variables have zero mean under the null and alternative, and assume the mixture is symmetric ($\nu = 1/2$).  Using the arguments provided in the proof of \prpref{malk3}, we can show that, if $n \to \infty$ fast enough, then the result of maximizing of the empirical ratio in \eqref{malk3} is consistent with
\beq \label{malk3-asymp}
\max_{\|u\|_0 \le s} \frac{\E \big|u^\top X\big|}{\left(\E \big[u^\top X \big]^2\right)^{1/2}} \ .
\eeq
Elementary calculations yield
\beq\label{eq:malk3-asymp-heuristic}
\frac{\E \big|u^\top X\big|}{\left(\E \big[u^\top X \big]^2\right)^{1/2}} 
= \frac{\E \big|h_u/2 + z\big|}{\left(1 + h_u^2/4\right)^{1/2}}
= \sqrt{\tfrac2\pi} \big(1 + \tfrac1{192} h_u^4 + O(h_u^6)\big) \ ,
\eeq
where $z \sim \cN(0,1)$ and when $h_u := u^\top \mudif/\sqrt{u^\top \bSigma u} \to 0$, which is allowed in \eqref{eq:upper_malk3}. 
The maximizer of \eqref{malk3-asymp} is therefore close to $\argmax_{\|u\|_0 \le s} |h_u|$, which does not necessarily have the same support as $\mudif$.

In view of \eqref{eq:malk3-asymp-heuristic}, we normalize \eqref{malk3} to cancel the denominator $(u^\top \bSigma u)^2$ in $h_u^4$, so that the maximizer is approximately aligned with $\mudif$.
This motivates us to consider the support estimator $\hat J = {\rm supp}(\hat u)$, where
\beq\label{eq:malk3_estimation}
\hat{u} \in \arg\max_{\|u\|_0 \le s, \|u\|=1} \left[\frac{\sum_i \big|u^\top (X_i - \bar X)\big|}{\left(\sum_i \big[u^\top (X_i - \bar X)\big]^2\right)^{1/2}}-\sqrt{\frac{2}{\pi}}\right]\bigg(\sum_i \big[u^\top (X_i - \bar X)\big]^2\bigg)^{2}\ .
\eeq

\begin{prp}\label{prp:sym_estim_malk3}
Consider the model \eqref{omega1-known} with the assumption  $|\nu-1/2|<\frac 1 6$, assume that $n \gg s \log (ep/s)$ and that $\mudif$ is $s$-sparse.
Then the estimator defined in \eqref{eq:malk3_estimation} is consistent for the support of $\mudif$ if
\beq \label{sym_estim_cond}
1\gg \frac{\|\mudif\|^4}{\mudif^\top \bSigma \mudif} \gg \left[\frac{s}{n}\log\left(\frac{ep}{s}\right)\right]^{1/4}\ ,
\eeq
and the effective dynamic range of $\mudif$ and the $2s$-sparse Riesz constant of $\bSigma$ are both bounded.
\end{prp}

Consequently, the estimator \eqref{eq:malk3_estimation} is consistent when the signal strength is just above the detection threshold. 
When the signal is strong, the procedure above seems to fail.
However, we mention that the simpler support estimator based on  
\[
\hat u \in 
\argmax_{\|u\|_0 \le s, \|u\| = 1} \ \sum_i \big|u^\top (X_i - \bar X)\big|\ ,
\]
is consistent when $\|\mudif\|^4/\mudif^\top \bSigma \mudif\to \infty$ under the same conditions otherwise.  Details are omitted as the arguments are similar, but simpler, than those underlying \prpref{sym_estim_malk3}.
Compare also with the coordinate-wise support estimator introduced in \secref{coord-unknown}.

\subsubsection{The Mahalanobis metric} \label{sec:maha}
The lower bounds obtained in \prpref{lower_sparse_unknow_variance} are in terms of $R_1$, while those we obtained for the case where the covariance matrix is known in \prpref{lower_test_sparse} are in terms of the Mahalanobis metric $R_0$. 
While these two metrics are equivalent if the $s$-sparse Riesz constant of $\bSigma$ is bounded, this is not so for any arbitrary $\bSigma$. We state below an information bound in terms of the Mahalanobis distance that is exponential in $p/n$, even when  $\mudif$ is 1-sparse. This suggests that $R_1$ is more relevant than $R_0$ in the present context. 

\begin{prp} 
\label{prp:lower_mahalanobis}
If $p\gg n$, then  $\lim\inf \gamma^*( \Omega_0(\nu), \Omega_1(\nu, R_0, r_n)) =1 $ when 
\beq\label{eq:lower_mahalanobis}
\mudif^\top \bSigma^{-1}\mudif \ll \frac{e^{p/(2n)}}{np}\ .
\eeq
\end{prp}

Again, the lower bound is proved by a reduction to the following simpler testing problem. 
Fix a 1-sparse vector $\mudif$, and consider $\theta\in  \Omega	^\ddag_0:= \big\{ (\tfrac{1}{2}, 0, 0, \bI)\big\}$
against  
\[
\theta\in \Omega^\ddag_1(\tfrac{1}{2}, R_0,r_n):=\big\{(\tfrac{1}{2},-\tfrac{1}{2}\mudif,\tfrac{1}{2}\mudif,\bSigma), \, \bSigma-\bI \text{ has rank 1}\text{ and }  R_0(\theta)\geq r_n \big\}\ .
\]
In contrast to the collection $\widetilde\Omega_1(\tfrac{1}{2}, R_0,r_n)$ used in the proof of Proposition \ref{prp:lower_sparse_unknow_variance}, $\Omega^\ddag_1(\tfrac{1}{2}, R_0,r_n)$ contains the collections of all rank 1 perturbation of the identity covariance matrix.

\subsection{Asymmetric setting} \label{sec:asym}
Consider the case where the covariance matrix is unknown and where the mixture distribution is asymmetric, meaning that $\nu \ne 1/2$. 
As we shall see, detection in the asymmetric setting is quantifiably easier than in the symmetric setting, due to the ability to test for asymmetry (in a particular manner).  

\subsubsection{Minimax lower bound}

We start with a minimax lower bound.
As in the symmetric setting covered in \secref{sym}, we use the signal-to-noise function $R_1$ defined in \eqref{R1}.

\begin{prp}\label{prp:lower_sparse_unknow_variance_asymmetric}
Consider testing \eqref{omega0} vs \eqref{omega1} with $\nu\neq 1/2$ fixed.
Then $\liminf \gamma^*(\Omega_0(\nu), \Omega_1(\nu,R_1,r_n)) = 1$ in the following cases:
\bitem
\item {\em Non-sparse setting.} 
Assume $s=p\to \infty$ and $p=o(n)$ and 
\beq \label{eq:lower_unknow_variance_asym}
r_n \ll (p/n)^{1/3} \ .
\eeq
\item {\em Sparse setting.}
Assume $p/s\to \infty$ and $n \gg s \log(e p/s)$ and  
\beq \label{eq:lower_sparse_unknow_variance_asym}
\lim\sup  r_n \left[\frac{s}{n}\log\left(\frac{ep}{s}\right) \right]^{-1/3} \leq C_{\nu}  \ ,
\eeq
where $C_{\nu}>0$ is a constant.
\eitem
\end{prp}

\subsubsection{A classical approach based on the skewness}
The classical approach in this asymmetric setting is a skewness test. 
We adapt the projection skewness test of \cite{malkovich1973tests} to our sparse setting.  This leads us to rejecting for large values of the following statistic:
\beq \label{malk4}
\max_{\|u\|_0 \le s} \frac{\sum_i \big[u^\top (X_i - \bar X)\big]^3}{\left(\sum_i \big[u^\top (X_i - \bar X)\big]^2\right)^{3 /2}}.
\eeq

\begin{prp} \label{prp:malk4}
Consider testing \eqref{omega0} versus \eqref{omega1} with the assumption  $\nu \ne 1/2$ fixed, and $n \gg s \log (ep/s)$.  
Let $T$ denote the statistic \eqref{malk4}.
There is a sequence of critical values $t$ such that the test $\phi = \{T \ge t\}$ is asymptotically powerful, meaning $\gamma(\phi;\Omega_0(\nu), \Omega_1(\nu,R_1,r_n)) 
\to 0$, if 
\beq\label{eq:condition_V}
\big(\nu(1-\nu)|1-2\nu|\big)^{2/3} r_n\gg \left[\frac{s}{n}\log\left(\frac{ep}{s}\right)\right]^{1/3} \vee \left[n^{1/3} \frac{s}{n}\log(ep/s)\right].
\eeq
\end{prp}

We notice a substantial discrepancy between this rate and the lower bound obtained in \prpref{lower_sparse_unknow_variance_asymmetric}.   As with the kurtosis statistic, the main issue is our difficulty with proving that the third moment concentrates enough under the null.  

\subsubsection{A new approach based on the signed second moment}
We replace the third moment with the second signed moment, leading to 
\beq \label{malk5}
\max_{\|u\|_0 \le s} \frac{\sum_i \big[u^\top (X_i - \bar X)\big]^2 \sign(u^\top(X_i - \bar X))}{\sum_i \big[u^\top (X_i - \bar X)\big]^2}\ .
\eeq

\begin{prp} \label{prp:malk5}
Consider testing \eqref{omega0} versus \eqref{omega1} with the assumption  $\nu \ne 1/2$ fixed, and $n \gg s \log (ep/s)$.
Let $T$ denote the statistic \eqref{malk5}.
There is a sequence of critical values $t$ such that the test $\phi = \{T \ge t\}$ is asymptotically powerful, meaning $\gamma(\phi; \Omega_0(\nu), \Omega_1(\nu,R_1,r_n)) \to 0$, if 
\beq \label{malk5_condition}
\liminf \big(\nu (1-\nu) |1 - 2 \nu|\big)^{2/3} r_n  \left[\frac{s}{n}\log\left(\frac{ep}{s}\right)\right]^{-1/3}\ge C \ ,
\eeq
where $C$ is a universal constant.
\end{prp}

We see that this test achieves the minimax rate established in \eqref{eq:lower_sparse_unknow_variance_asym}.
Note that the minimax detection rate is substantially faster in the asymmetric case compared with the symmetric case.

\paragraph{Variable selection}
Here too, we are unable to use the statistic \eqref{malk5} to perform variable selection.
In analogy with the symmetric case, we consider the estimator $\hat J = {\rm supp}(\hat u)$, where
\beq \label{asym_estim}
\hat u \in 
\argmax_{\|u\| = 1, \|u\|_0 \le s}  \Big[\sum_i \big[u^\top (X_i - \bar X)\big]^2 \sign(u^\top(X_i - \bar X))\Big]\Big[\sum_i \big[u^\top (X_i - \bar X)\big]^{2}\Big]^{1/2}\ .
\eeq
Despite the strong parallel with the statistic \eqref{eq:malk3_estimation}, we were not able to obtain a satisfactory performance for \eqref{asym_estim}.
We mention, as we did before, that other estimators may be needed when the signal is strong.  And we also refer the reader to \secref{coord-unknown}, where a coordinate-wise support estimator is introduced.

\subsection{Diagonal model} \label{sec:diagonal}

A popular approach in situations where the covariance is unknown is to assume it is diagonal.  In (supervised) classification this leads to diagonal linear discriminant analysis, which corresponds to the naive Bayes classifier in the Gaussian mixture model \citep{MR2108040}.
Define 
\beq \label{kappa}
\kappa = \|\mudif_\ddag\|_\infty/\|\mudif_\ddag\|, \quad \mudif_\ddag := \bSigma^{-1/2} \mudif\ .
\eeq
Given  $\nu\in(0,1)$,  $a\in (0,1)$, and  $s\leq p$, we consider the mixture testing problem with unknown diagonal covariance matrix, which we define as testing
\beq\label{omega0_diag}
\breve{\Omega}_0 = \big\{\theta = (\nu, \mu, \mu, \bSigma),\ \mu \in \bbR^p,\ \bSigma\text{ diagonal psd}\big\}
\eeq
versus
\beq\label{omega1_diag}
\breve{\Omega}_1(\nu,  R, r_n) := \breve{\Omega}_1(\nu)\cap \{\theta:\  R(\theta)\geq r_n\}~,
\eeq
where
\[
\breve{\Omega}_1(\nu) := \big\{\theta = (\nu, \mu_0, \mu_1, \bSigma) : \mu_0, \mu_1 \in \bbR^p \text{ satisfying \eqref{mudif}, }\bSigma\text{ diagonal psd}, \kappa\leq a \big\}\ .
\]

In this situation, it is natural to estimate the covariance matrix by the diagonal of the sample covariance matrix.  We can then use this estimator in place of $\bSigma$ in \eqref{eq:St_known_sparse}, yielding the following statistic
\beq \label{diagonal-stat}
\max_{\|u\|_0 \le s} \ \frac{u^\top \hat\bSigma u}{u^\top \diag(\hat\bSigma) u},
\eeq
with the convention that $0/0 = 0$, where for a square matrix $A = (a_{ij})$, $\diag(A)$ denotes the diagonal matrix with diagonal elements $(a_{ii})$.
The null distribution of the test statistic \eqref{diagonal-stat} does not depend on $\bSigma$ as long as it is diagonal.

\begin{prp}\label{prp:upper_diagonal_sparse}
Consider testing \eqref{omega0_diag} versus \eqref{omega1_diag} with $\nu\in (0,1)$ fixed, and $1 \ll \log p \ll n$.  
Assume that $\kappa$ in \eqref{kappa} is bounded away from 1.
\bitem
\item {\em Detection.}
Let $T$ denote the statistic \eqref{diagonal-stat}.
There is a sequence of critical values $t$ such that the test $\phi = \{T \ge t\}$ is asymptotically powerful, meaning $\gamma(\phi; \breve{\Omega}_0, \breve{\Omega}_1(\nu,  R, r_n)) \to 0$, if 
\beq \label{upper_diagonal_sparse}
\nu (1-\nu) \, r_n \ge \frac{C}{1 - a^2} \left[\sqrt{\frac{s}{n}\log(ep/s)} \vee \frac{s}{n}\log(ep/s)\right]\ ,
\eeq
where $C>0$ is a universal constant.
\item {\em Variable selection.}
Let $\hat u$ denote a maximizer of \eqref{diagonal-stat}.  Then under the slightly stronger condition \eqref{upper_test_sparse1-select}, and assuming that $\|\mudif\|_0>1$  and that the effective dynamic range of $\mudif_\ddag$ is bounded, the support of $\hat u$ is consistent for the support of $\mudif$.
\eitem
\end{prp}

\prpref{upper_diagonal_sparse} presents an interesting phenomenon. When the covariance matrix is supposed to be diagonal but is unknown, there is a qualitative difference between the case $\|\mudif\|_0=1$ and $\|\mudif\|_0>1$.
The conditions of \prpref{upper_diagonal_sparse} imply that $\|\mudif\|_0 >1$.  When $\|\mudif\|_0 = 1$, the statistic \eqref{diagonal-stat} is useless at either detection or variable selection, since in that situation $\Cov(X)$ is also diagonal under the alternative.
In that case, the optimal detection rate is the same as that for general unknown covariances, that is $(\log(p)/n)^{1/4}$ when $\nu=1/2$ and $(\log(p)/n)^{1/3}$ when $\nu \ne 1/2$. Indeed, when $s=1$, the proofs of Propositions~\ref{prp:lower_sparse_unknow_variance} and~\ref{prp:lower_sparse_unknow_variance_asymmetric} are based on diagonal covariance matrices.
When $\|\mudif\|_0> 1$, \eqref{upper_diagonal_sparse} is the same as \eqref{upper_test_sparse2}, meaning we can do as well as if $\bSigma$ were known, as long as $\kappa$ remains bounded away from 1, meaning that $\mudif_\ddag$ is not approximately 1-sparse. 


\section{Computationally tractable methods and numerical experiments} \label{sec:poly}

A test statistic of the form $\max\{G(u; X_1, \dots, X_n) : \|u\|_0 \le s\}$, where $G$ is a real-valued function, results in a  combinatorial maximization over the subsets of $[p]$ of size at most $s$, and this is very quickly intractable when $s \to \infty$ as $n \to \infty$, because there are $\binom{p}s \ge (p/s)^s$ such subsets.  

To be more precise, here we say that a method is computationally tractable if it can be computed in time polynomial in $(n, p, s)$.  Although such a method may still be practically intractable for large problems, on a theoretical level, it provides a qualitative definition in line with a central concern in theoretical computer science.  Among the statistics considered in Sections~\ref{sec:known} and~\ref{sec:unknown}, only the largest eigenvalue $\hat{\lambda}_{\bSigma}^{\max}$ defined in \eqref{eq:St_known} is known to be computable in polynomial time. 
All the other methods are tailored to the sparse setting and are combinatorial in nature.
This motivates the development of computationally tractable methods for this setting.


\subsection{Coordinate-wise methods} \label{sec:coord}
The simplest computationally tractable methods are arguably those based on testing each coordinate at a time.  Such a method is of the form 
\beq \label{coord-stat}
M(T_\ddag(X^1), \dots, T_\ddag(X^p))\ ,
\eeq 
where $X^j = (X_{i,j} : i \in [n])$ is the $j$th variable, $T_\ddag$ is a test statistic for mixture testing in dimension one, and $M$ implements a multiple testing procedure.  In what follows, we opt for the simple Bonferroni correction, which corresponds to $M(t_1, \dots, t_p) = \max_j t_j$.  
Coordinate-wise testing and/or variable selection of this type is considered in \citep{azizyan13,chan10} and also in \citep{amini2009high,johnstone2009consistency,berthet} in the context of sparse PCA.  
Such approaches are also considered in recent work\footnote{This work was made publicly available after ours.} by \cite{jin2014important} and \cite{jin2015phase}, who obtain very precise minimax results when the covariance matrix has relatively small condition number.
Except for \citep{chan10}, where a nonparametric setting is considered, these papers assume that the covariance matrix is known.

\subsubsection{Known covariance}
Denote $\bSigma = (\sigma_{jk})$ and $\hat\bSigma = (\hat\sigma_{jk})$.
Inspired by the statistic \eqref{eq:St_known_sparse-bis}, we arrive at the maximum canonical variance statistic
\beq \label{poly-know1}
\max_{j \in [p]} \frac{\hat \sigma_{jj}}{\sigma_{jj}}\ , 
\eeq 
and at the corresponding support estimator
\beq \label{J-hat}
\hat J = \big\{j \in [p] : \hat \sigma_{jj}/\sigma_{jj} > t\big\}\ , \quad t := 1 + 5  \big(\sqrt{\tfrac{\log(p)}{n}} \vee \tfrac{\log(p)}{n}\big)~,
\eeq
for a given threshold $\omega \to \infty$.    
Note that \eqref{poly-know1} corresponds to working with the test statistic $T_\ddag(x_1, \dots, x_n) = \frac1n \sum_i (x_i - \bar x)^2$ in \eqref{coord-stat}.

\begin{prp} \label{prp:coord-known}
Consider testing \eqref{omega0-known} versus \eqref{omega1-known} with  $\nu \in (0,1)$ fixed and $p\rightarrow \infty$.
Denoting $T$ the statistic \eqref{poly-know1}, we consider the test $\phi = \big\{T \ge t \big\}$ with $t$ defined in \eqref{J-hat}. 
The test $\phi$ has asymptotic level 0.  Moreover, it has asymptotic power 
one if
\beq \label{upper_known_1}
\nu (1-\nu) \max_{j \in [p]} \frac{\mudif_j^2}{\sigma_{jj}} > C_1 \bigg(\sqrt{\frac{\log(p)}{n}} \vee \frac{\log(p)}{n}\bigg)\ ,
\eeq
where $C_1 > 0$ is a universal constant.
Moreover, the estimator \eqref{J-hat} is consistent for the support of $\mudif$ if 
\[
\nu (1-\nu) \min_{j \in J} \frac{\mudif_j^2}{\sigma_{jj}} > C_2 \bigg(\sqrt{\frac{\log(p)}{n}} \vee \frac{\log(p)}{n}\bigg) \ ,
\]
where $C_2>0$ is a universal constant. 
\end{prp}
The proof is a straightforward adaptation of that of \prpref{upper_test_sparse2}, and is omitted.

\begin{rem}
A stronger result can be obtained by using \eqref{eq:St_known_sparse} instead of \eqref{eq:St_known_sparse-bis}, leading to $\hat \lambda^{\rm max}_{1, \bSigma}$ instead of \eqref{poly-know1}, but the approach is somewhat less natural and the resulting performance bound somewhat less intuitive.
\end{rem}

\medskip \noindent
{\bf Special case $\bSigma=\bI$.}  
In Section~\ref{sec:known} we proved that the test based on \eqref{eq:St_known_sparse} is asymptotically powerful under \eqref{upper_test_sparse2}, that is 
\[
\nu(1-\nu)\|\mudif\|^2\geq C \big[\sqrt{\tfrac{s}{n}\log(ep/s)}\vee \tfrac{s}{n}\log(ep/s)\big]~.
\] 
The coordinate-wise test was shown here to be asymptotically powerful under \eqref{upper_known_1}, that is 
\[
\nu(1-\nu)\|\mudif\|_\infty^2\geq C \big[\sqrt{\frac{\log(p)}{n}}\vee \frac{\log(p)}{n}\big]~.
\]  
($C$ is a sufficiently large constant.)
When the energy of $\mudif$ is spread over its support, we have $\|\mudif\|_\infty \approx \|\mudif\|/\sqrt{s}$, in which case the latter condition becomes 
\[\nu(1-\nu)\|\mudif\|^2\geq C \big[\sqrt{\tfrac{s^2}{n}\log(ep/s)}\vee (\tfrac{s}{n}\log(ep/s))\big]~.\]
Hence, the coordinate-wise method is shown to achieve a detection rate within a multiplicative factor $\sqrt{s}$ of the optimal rate.
In the special situation where $n = O(\log p)$, the coordinate-wise method even achieves the optimal rate.  In general, however, there is this multiplicative factor of $\sqrt{s}$ between the detection bounds.  We speculate that this factor of $\sqrt{s}$ is unavoidable and incurred by any polynomial time method.  
Our speculation is based on an analogy with the sparse PCA detection problem and the recent work of \cite{berthet2}.  These authors prove that a multiplicative factor of $\sqrt{s}$ applies to any polynomial time algorithm, if some classical problem in computational complexity, known as the Planted Clique Problem, is not solvable in polynomial time --- see \citep{berthet2} for definitions and pointers to the literature. 
(Although we focused on the case $\bSigma = \bI$, this discussion is in fact valid for general covariance matrices $\bSigma$ as long as the $s$-sparse Riesz constants are bounded.)

\subsubsection{Unknown covariance} \label{sec:coord-unknown}

We adapt the statistics \eqref{malk3} and \eqref{malk5} to coordinate-wise methods by considering $s=1$, thus working with
\beq \label{malk3-1}
T_1 = \max_{j\in [p]} T_{1,j}\ ,\quad \quad T_{1,j}:=\sum_{i=1}^n \frac{|X_{i,j} - \bar X_{j}|}{\sqrt{\hat{\sigma}_{jj}}} \ ,
\eeq
and
\beq \label{malk5-1}
T_2 = \max_{j\in [p]} T_{2,j}\ ,\quad \quad T_{2,j}:= \ \Big|\sum_{i=1}^n  \frac{(X_{i,j} - \bar X_{j})^2}{\hat{\sigma}_{jj}} \sign(X_{i,j} - \bar X_j)\Big|\ .
\eeq

Although the null distribution of \eqref{malk3-1} depends on the unknown covariance matrix $\bSigma$, it can be calibrated by a simple Bonferroni correction, which is possible because the terms in the maximum have a null distribution that is independent of $\bSigma$.  The same applies to \eqref{malk5-1}.

For any $u\in (0,1)$, denote by $q_{1}^{-1}(u)$ the $(1-u)$-quantile of the distribution of $T_{1,1}$ under the null hypothesis. Given some level $\alpha\in(0,1)$, denote by $\hat{J}_1$ the set of indices such that $T_{1,j}$ is significant at level $\alpha/p$, namely, 
\[\hat{J}_1 := \big\{j : T_{1,j} > q_{1}^{-1}(\alpha/p)\big\}\ .\] 
The estimator $\hat{J}_2$ is defined analogously based on \eqref{malk5-1}. 
In practice, the quantile functions $q_1^{-1}$ and $q_2^{-1}$ can be easily estimated by Monte Carlo simulations.

\begin{prp}\label{prp:coordinate_unknown}
Consider testing \eqref{omega0} versus \eqref{omega1} with $n \gg  \log (p) \gg 1$.  Consider a sequence of levels $\alpha$ satisfying $\alpha=o(1)$ and  $\alpha \ge p^{-a}$ for some fixed $a > 0$ in the definition of $\hat{J}_1$ and $\hat{J}_2$.
\bitem 
\item {\em Detection.} 
The test $\phi_1 := \{\hat{J}_1 \neq \emptyset \}$ has a level smaller than $\alpha$.  Moreover, it has asymptotic power 1 if 
 if 
$|\nu-1/2|<\frac 1 6$ and 
\beq\label{eq:upper_malk3-1}
\max_{j\in [p]} \frac{(\mudif_j)^2}{ \sigma_{jj}} \gg \left[\frac{\log(p)}{n}\right]^{1/4}.
\eeq
The test $\phi_2 := \{\hat{J}_2 \neq \emptyset \}$ has a level smaller than $\alpha$.  Moreover, it has asymptotic power 1 if 
\beq \label{malk5_condition-1}
\liminf \  \big(\nu (1-\nu) |1 - 2 \nu|\big)^{2/3} \ \max_{j\in [p]} \frac{(\mudif_j)^2}{ \sigma_{jj}} \left[\frac{\log(p)}{n}\right]^{-1/3}\ge C \ .
\eeq
\item {\em Variable selection.} 
Assume that the effective dynamic range of $\mudif$ and the $2s$-sparse Riesz constant of $\bSigma$ are both bounded. 
\renewcommand\labelitemii{$\circ$}
\bitem
\item If $|\nu-1/2|<\frac1 6$, then under the stronger condition that 
\beq \label{upper_test_sparse1-select_malk3-1}
\frac{\|\mudif\|^4}{\mudif^\top \bSigma \mudif}\gg  s \left[\frac{\log(p)}{n}\right]^{1/4}\ , 
\eeq
 $\hat{J}_1$ is consistent for the support of $\Delta\mu$. 
\item If $\nu\neq 1/2$ is fixed, then under the stronger condition that 
\beq \label{upper_test_sparse1-select_malk5-1}
\frac{\|\mudif\|^4}{\mudif^\top \bSigma \mudif}\gg s\left[\frac{\log(p)}{n}\right]^{1/3}\ , 
\eeq
$\hat{J}_2$ is consistent for the support of $\Delta\mu$. 
\eitem

\eitem
\end{prp}

\begin{rem}
Assuming the energy of $\mudif$ is spread over its support, and that the $s$-sparse Riesz constant of $\bSigma$ is bounded, \eqref{eq:upper_malk3-1} and \eqref{malk5_condition-1} reduce to
\beqn
\frac{\|\mudif\|^4}{\mudif^{\top}\bSigma\mudif} 
&\gg& s \left[\frac{\log(p)}{n}\right]^{1/4} = s^{3/4} \left[s\frac{\log(p)}{n}\right]^{1/4}\ , \\
\big(\nu (1-\nu) |1 - 2 \nu|\big)^{2/3} \frac{\|\mudif\|^4}{\mudif^{\top}\bSigma\mudif} 
&\ge&  C s \left[\frac{\log(p)}{n}\right]^{1/3} = C s^{2/3} \left[s\frac{\log(p)}{n}\right]^{1/3}\ .
\eeqn
Compared to \eqref{eq:upper_malk3} and \eqref{malk5_condition}, the performances of the coordinate-wise methods are within $s^{3/4}$  and $s^{2/3}$ multiplicative factors, respectively, of the optimal rates. 
We do not know to what extent this is intrinsic to the problem, namely, whether there are polynomial time methods with performance bounds that come closer to the optimal bounds.
\end{rem}

\subsection{Other computationally tractable methods} \label{sec:sdp}
Beyond methods based on examining $k$-tuples of coordinates, instead of just $k=1$ coordinate at a time, and other heuristics based on principal component analysis \citep{MR777837}, more sophisticated methods may be needed.  
We present two methods based on relaxations of the sparse eigenvalue problem, which we learned from \cite{berthet}, who applied it to the problem of detecting a top principal component in a spiked covariance model.  
See also \citep{amini2009high}.
Assume for simplicity that $\bSigma = \bI$ (or, equivalently, that it is diagonal and known), so that the sparse eigenvalues defined in \eqref{eq:St_known_sparse} and \eqref{eq:St_known_sparse-bis} coincide, and in both cases, the maximization is over $s$-sparse unit vectors.  
\bitem
\item The first relaxation is the semidefinite program (SDP) of \cite{aspremont}:
\beq \label{sdp}
{\rm SDP}_s(\bA) = \max_{\bZ} \, \trace(\bA \bZ), \quad \text{ subject to } \bZ \succeq 0, \, \trace(\bZ) = 1, \, |\bZ|_1 \le s~,
\eeq 
where the maximum is over positive semidefinite matrices $\bZ = (Z_{st})$ and $|\bZ|_1 := \sum_{s,t} |Z_{st}|$.
We would then use ${\rm SDP}_s(\hat\bSigma)$.  
\item The second relaxation leads to using the minimum
dual perturbation
\beq \label{mdp}
{\rm MDP}_s(\bA) := \min_{z \ge 0} \Big[ \lambda^{\rm max}(\tau_z(\bA)) + s z\Big] \ ,
\eeq
where $\tau_z$ is entry-wise soft-thresholding at $z$, meaning that for a matrix $\bA = (a_{jk})$, $\tau_z(\bA) = (b_{jk})$, where $b_{jk} = \sign(a_{jk}) \max(|a_{jk}| - z, 0)$.
We would then use ${\rm MDP}_s(\hat\bSigma)$.
\eitem 
Both relaxations operate in polynomial time.  That said, the semidefinite program does not scale well, while the second relaxation is computationally more friendly as it boils down to a one-dimensional grid search over $z \in \bbR$ requiring the computation of the top eigenvalue of symmetric matrix at every grid point.

\begin{prp} \label{prp:MDP}
Consider testing \eqref{omega0-known} versus \eqref{omega1-known} with  $\nu \in (0,1)$ fixed, and $n\wedge p\to  \infty$.
Let $T$ denote either of the statistics ${\rm SDP}_s(\hat\bSigma)$ or ${\rm MDP}_s(\hat\bSigma)$.
For some universal constant $C_0>0$, consider the test 
\[
\phi = \Big\{T \ge  1 + C_0 \big[\sqrt{\frac{s^2}{n}\log(ep/s)}\vee \frac{s}{n}\log(ep/s)\big]\Big\}~.
\] 
The test $\phi$ has asymptotic level 0.  Moreover, it has asymptotic power 1 if 
\[\nu(1-\nu)\mudif^\top\bSigma^{-1}\mudif\geq C_1 \Big[\sqrt{\frac{s^2}{n}\log\left(\frac{ep}{s}\right)}\vee \frac{s}{n}\log\left(\frac{ep}{s}\right)\Big] \ , \]
where $C_1>0$ is a universal constant.
\end{prp}

The proof of \prpref{MDP} is a straightforward adaptation of the work of \cite{berthet}.  The critical ingredient is the following inequality
\beq \label{SDP-ineq}
\lambda_s^{\rm max}(\bA) \le {\rm SDP}_s(\bA) \le {\rm MDP}_s(\bA) \ ,
\eeq
valid for any psd matrix $\bA$ and any sparsity level $s$.
Then, on the one hand, we find in \citep[Prop.~6.2]{berthet} that 
\[
{\rm MDP}_s(\hat\bSigma) \le 1 + C_1 \Big[\sqrt{\frac{s^2}{n}\log\left(\frac{ep}{s}\right)}\vee \frac{s}{n}\log\left(\frac{ep}{s}\right)\Big]
\]
with probability tending to one under the null (where the sample is iid standard normal); while, on the other hand, following what we did in the proof of \prpref{upper_test_sparse2}, we find that $\lambda_s^{\rm max}(\hat\bSigma) \ge 1 - \frac1n + C_2 \nu (1-\nu) \mudif^\top\bSigma^{-1}\mudif$ with probability tending to one under the alternative.  From these two bounds and \eqref{SDP-ineq}, we conclude.

\begin{rem}
We notice the same $\sqrt{s}$ multiplicative factor and one wonders whether the added sophistication of these relaxations (SDP or MDP) is worth it.  Clearly not from a theoretical standpoint, but it shows in our numerical experiments presented in \secref{numerics}. 
This is analogous to what \cite{berthet} observed in the context of detecting a first principal component.
\end{rem}

\begin{rem}
We do not know of any analogous relaxations for the statistics presented in \secref{unknown} for the case where the covariance matrix is unknown.
\end{rem}

\subsection{Numerical experiments} \label{sec:numerics}

We present here the result of some small-scale computer simulations meant to compare some of the computationally tractable tests introduced above.  
In all the experiments, we chose $p = 500$, $n = 200$, and the underlying covariance matrix $\bSigma$ (whether assumed known or unknown) was taken to be the identity.  
The variables were generated with zero mean under both the null and the alternative.  
The difference in means, $\mudif$, was chosen to be equally spread (in terms of energy) over all its nonzero coordinates.  Specifically, we chose $\mudif_j = A \IND{j \le s}/\sqrt{s}$, where the sparsity $s$ ranged over $\{1,5,10,30\}$, while the amplitude $A = \|\mudif\|$ varied the difficulty of the detection problem.
We focused entirely on the symmetric model where $\nu = 1/2$.
Each setting was repeated 100 times.  

\subsubsection{Known covariance}

In this set of experiments, we assume that $\bSigma$ is known to be the identity, and compared the maximum canonical variance \eqref{poly-know1}, the $s$-largest canonical variance $\hat \sigma_{jj}/\sigma_{jj}$, the top sample eigenvalue \eqref{eq:St_known}, and the MDP statistics defined in \eqref{mdp}.  
Note that the $s$-largest canonical variance and the ${\rm MDP}_s$ both require knowledge of $s$.
The results from these experiments are shown as power curves in \figref{known}. Among other things, they confirm that the maximum canonical eigenvalue performs best when  $\mudif$ is really sparse whereas top sample eigenvalues performs best for less sparse signals -- see Table \ref{tab:known-1}.
At least in the particular setting of these simulations, the combination of the maximum canonical variance and the top sample eigenvalue is competitive.
An alternative --- which we did not implement and is most relevant when $\bSigma$ is diagonal --- would be a higher-criticism approach applied to the canonical variances $\hat \sigma_{jj}/\sigma_{jj}$, which under the null are iid $\frac1n \chi^2_{n-1}$.     

\begin{figure}[htbp]
\centering
\begin{tabular}{cccc}
\includegraphics[width=.22\linewidth]{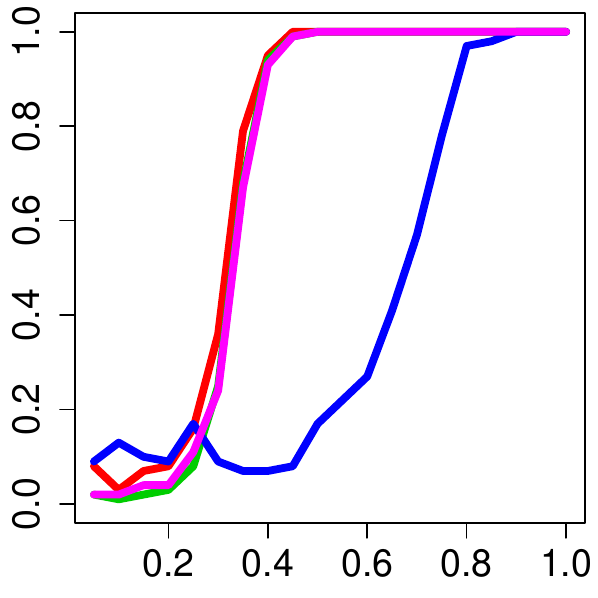} &
\includegraphics[width=.22\linewidth]{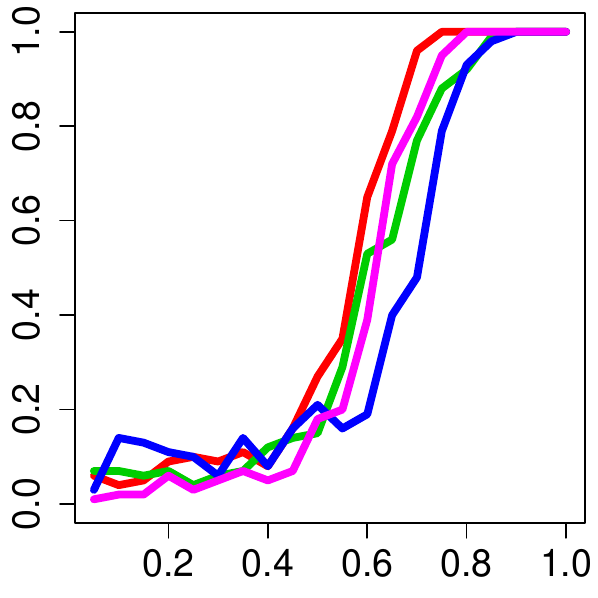} &
\includegraphics[width=.22\linewidth]{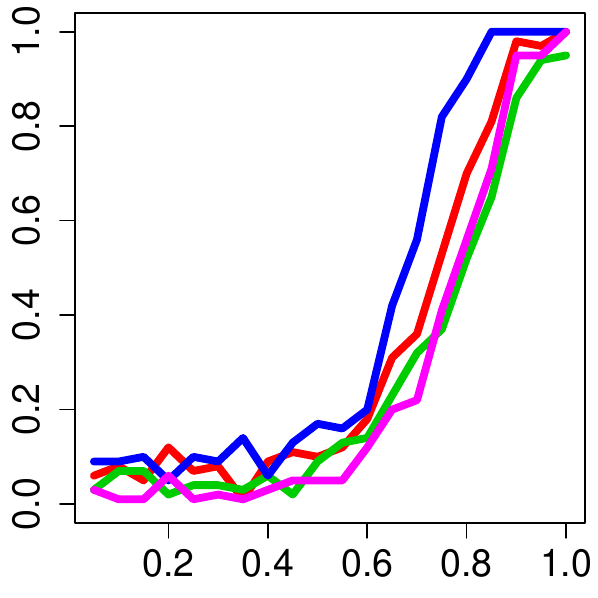} &
\includegraphics[width=.22\linewidth]{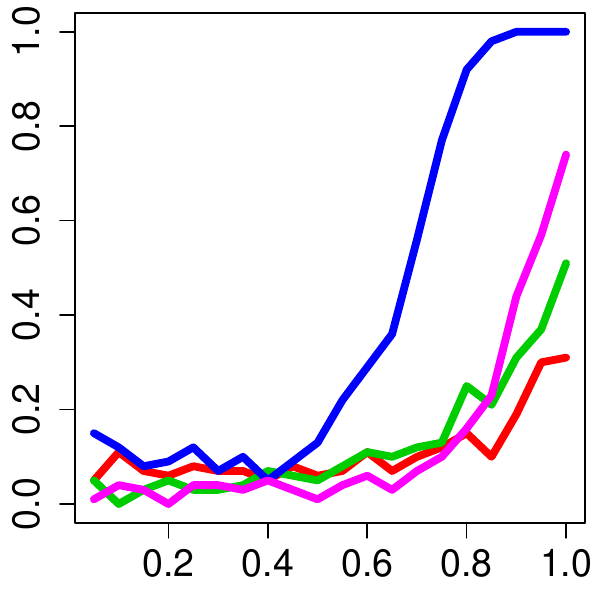} \\
$s=1$ & $s=5$ & $s=10$ & $s=30$
\end{tabular}
\caption{Power curves for the largest canonical variance (red), $s$-th largest canonical variance (green), top sample eigenvalue (blue) and MDP (magenta) for various sparsity levels $s$ as displayed.  The level was set at 0.05 by simulation.  On the horizontal axis is $A = \|\mudif\|$, while on the vertical axis is the proportion of rejections (out of 100 repeats).  }
\label{fig:known}
\end{figure}

\subsubsection{Unknown covariance (kurtosis versus first moment)}

In this set of experiments, we assume that $\bSigma$ is unknown (even though it remains the identity), and compared the coordinate-wise kurtosis and first absolute moment.  
We used the maximum canonical variance (whose calibration is only possible when $\bSigma$ is known) as an oracle benchmark.
Although we have a tighter control of the first moment under the null compared to the kurtosis, in these experiments the two behave very similarly.  

\begin{figure}[htbp]
\centering
\begin{tabular}{cccc}
\includegraphics[width=.22\linewidth]{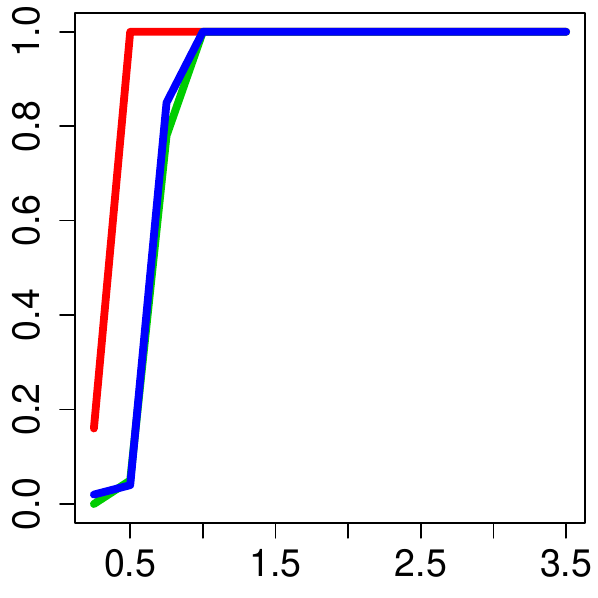} &
\includegraphics[width=.22\linewidth]{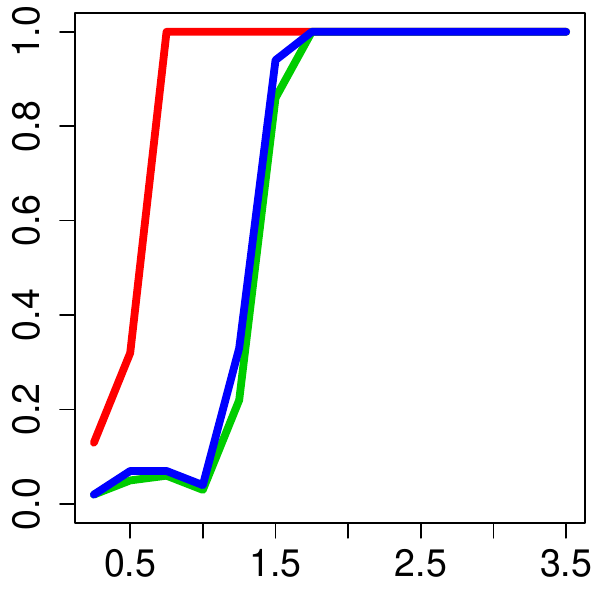} &
\includegraphics[width=.22\linewidth]{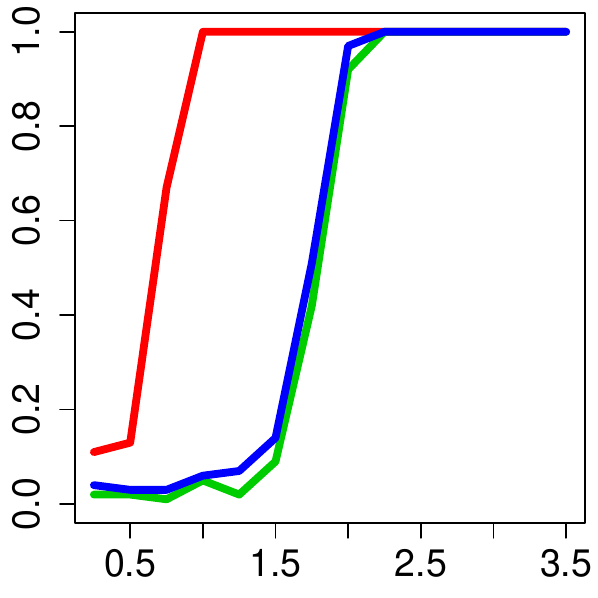} &
\includegraphics[width=.22\linewidth]{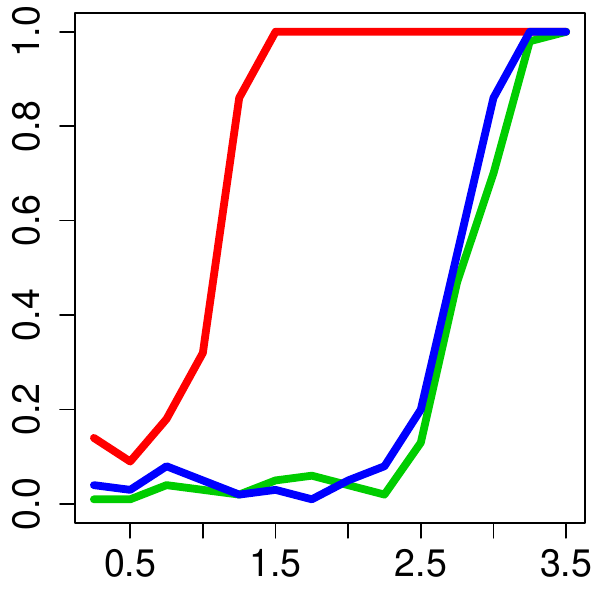} \\
$s=1$ & $s=5$ & $s=10$ & $s=30$
\end{tabular}
\caption{Power curves for the largest canonical variance (red), the largest canonical kurtosis (green), and the largest canonical absolute moment (blue), for various sparsity levels $s$ as displayed.  The level was set at 0.05 by simulation.  On the horizontal axis is $A = \|\mudif\|$, while on the vertical axis is the proportion of rejections (out of 100 repeats).  }
\label{fig:unknown}
\end{figure}

\section{Discussion} \label{sec:discussion}

This paper leaves a number of interesting open problems regarding the theory of clustering under sparsity.  We list a few of them below.

\paragraph{The generalized likelihood ratio test (GLRT)}
The GLRT performs well in very many testing problems.  In this paper we simply focused on obtaining tests that achieved the various optimal detection rates, and we are curious to know whether the GLRT is one of them, at least in some of the settings.
If anything, the GLRT is computationally very intensive in high dimensions, even more so than the moment-based methods analyzed here, and therefore not practical, while heuristic implementations \`a la EM are very hard to analyze.

\paragraph{Theoretical adaptation to unknown sparsity}
Throughout the paper, except in \secref{coord}, we work under the assumption that the sparsity level $s$ is known.  This is in fact a mild assumption.
Indeed, on the one hand, the problem is harder when $s$ is unknown (since the set of alternatives is larger), so that the minimax lower bounds developed in the paper apply to the case where $s$ is unknown.  
On the other hand, one can easily check that there is enough lee-way in the concentration bounds developed (under the null) for the various procedures that rely on $s$ to accommodate a scan over $s \in [p]$.

\paragraph{Adaptation to unknown sparsity for computationally tractable procedures} We also note that the coordinate-wise methods studied in Section~\ref{sec:coord} do not require the knowledge of the sparsity. When the population covariance matrix $\bSigma$ is known, one can rely on the maximal canonical variance statistic \eqref{poly-know1} and the top eigenvalue statistic \eqref{poly-know1} together with a Bonferroni correction to simultaneously achieve the rates of Table \ref{tab:known-1} for all $s$.

\paragraph{Unknown mixing probability}
We have assumed that $\nu$ is unknown.  However, when the covariance matrix is unknown, it matters whether $\nu = 1/2$ or $\nu \ne 1/2$, for the proposed methods are different --- based on the first absolute moment and the second signed moment, respectively.  Let us focus on the coordinate-wise methods introduced and studied in \secref{coord-unknown}.  
For the detection problem, an easy way to adapt to situations where it is unknown whether $\nu = 1/2$ or $\nu \ne 1/2$ is to combine the tests based on $T_1$ and $T_2$ with a Bonferroni correction. 
For the variable selection problem, one can simply consider the union $\hat{J}_1\cup \hat{J}_2$ of the variables selected by the two methods.

\paragraph{Mixture models with different covariance matrices}
We assumed everywhere in the paper that the two populations had the same covariance matrix.  When this is not the case, assuming the two population covariance matrices are known (both under the null and under the alternative) does not seem as meaningful, and the case where they  are unknown is more complex, and we speculate that more sophisticated methods that attempt to cluster the data into two groups (as the GLRT does) may be required.  We note, however, that the procedure presented in \secref{diagonal} applies in exactly the same way to the special case where the population covariance matrices are diagonal --- although the performance bound established in \prpref{upper_diagonal_sparse} is not valid.

\paragraph{Mixture models with more than two components}
Suppose the mixture, under the alternative, has $K \ge 2$ components, with the $k$th component having mean $\mu_k$ and proportion $\nu_k$, and consider for simplicity the case where the population covariance matrix is known to be the identity, both under the null and the alternative.
Then, under the alternative,
\[
\Cov(X) = \sum_{k=1}^K \nu_k (1-\nu_k) \mu_k \mu_k^\top + \sum_{1 \le k < \ell \le K} \nu_k \nu_\ell (\mu_k \mu_\ell^\top + \mu_\ell \mu_k^\top) + \bI\ .
\]
In the general situation where the group means are affine independent, $\Cov(X)$ is a rank $K-1$ perturbation of the identity matrix.  It is therefore natural to consider a test based on the top $K-1$ $s$-sparse eigenvalues of the sample covariance matrix.  We note, though, that when $K$ is fixed, the top $s$-sparse eigenvalue is still able to achieve the optimal detection rate. 
See \citep{hsu2013learning} for related results in a non-sparse setting. 
Recently, \cite{jin2014important} have also studied the case where the population covariance matrices are diagonal but unknown.
 
\paragraph{Computational issues}  
Computational considerations have lead a number of researchers to propose coordinate-wise methods as we did in \secref{coord}.  In \secref{sdp} we studied an SDP relaxation in the context of mixture detection when the covariance matrix is known to be the identity.  It seems possible to extend this to the case of a general known covariance matrix.
If anything, the suboptimal test based on \eqref{eq:St_known_sparse-bis} can be relaxed in the same exact way since it is based on computing a top sparse eigenvalue.  And the same is true of the diagonal model.  However, we do not know how to relax any of the tests considered in the case where the covariance matrix is unknown.

\section{Proofs: lower bounds} \label{sec:lower-proofs}

We start with proving the lower bounds.  The arguments follow standard lines, but the calculations are delicate at times.  The basic idea is to reduce the hypothesis testing problem to a simple versus simple hypothesis testing problem, by putting priors on the null and alternative sets of distributions.  

In the sequel, we use the notation
\beq \label{zeta}
\zeta = \frac{s}n \log(e p/s)\ .
\eeq

We first reduce to the case where the variables have zero mean.    
And when the null is composite, we focus on the isotropic sub-case.
We then put a prior on $\mudif$ to reduce the alternative to a simple hypothesis. 
Except in \secref{proof_lower_mahalanobis}, we let $\varrho$ denote the uniform distribution on the set of $s$-sparse vectors in $\mathbb{R}^p$ whose non-zero values either equal $r$ or $-r$ --- the dependency on $r$ being left implicit --- and choose it as prior.  We then determine the value of $r$ that makes the testing problem difficult. 

This reduction to a simple versus simple hypothesis testing provides a lower bound on the worst-case risk for the original testing problem.  
The last step consists in lower bounding the risk of the likelihood ratio (LR) test for the simple versus simple problem, which lower bounds the risk of any other test since the LR test is optimal by the Neyman-Pearson lemma.  If $L$ is the LR for a simple versus a simple, then its risk is equal to
\beq \label{LR-risk}
1 - \frac12 \E_0 |L - 1| \ge 1 - \frac12 \sqrt{\E_0(L^2) - 1}\ ,
\eeq
where $\E_0$ denote the expectation under the null and the inequality is Cauchy-Schwarz's.  Hence, the goal of the (long, and sometimes tedious) calculations that follow is to upper-bound the second moment of the LR.

We will reduce the hypergeometric to the binomial distribution using the following taken from \citep[p.173]{aldous85}.  Here, ${\rm Hyper}(m, n, N)$ refers to the hypergeometric distribution which arises when picking $m$ balls at random from an urn with $n$ red balls and $N-n$ blue balls, and counting the number of red balls. 

\begin{lem} \label{lem:aldous} 
For any integers $1 \le m, n \le N$, there is a $\sigma$-algebra $\mathcal{B}$ and a binomial random variable $W$ with parameters $(m, n/N)$ such that $\mathbb{E}(W|\mathcal{B}) \sim {\rm Hyper}(m, n, N)$. 
\end{lem}

We will also use Chernoff's bound for the binomial distribution.
\begin{lem}[Chernoff's bound] \label{lem:chernoff}
For any positive integer $n$ and any $0 < p < q < 1$, we have
\beq \label{chernoff}
\pr{\Bin(n, p) \ge q n} \leq \exp\left(- n H_{p}(q) \right)\ , 
\eeq
where the entropy function $H_{p}(q)$ satisfies
\[H_{p}(q)\geq q \log\left(\frac{q}{p}\right) -q + p \ .\]
\end{lem}

\subsection{Proof of \prpref{lower_test_sparse}}
We prove simultaneously both results ($s=p$ and $s=o(p)$).  Following the steps outlined at the beginning of \secref{lower-proofs}, we reduce the testing problem to 
\beqn
\tilde H_0: X \sim \cN(0,\bI) \quad \text{versus} \quad \tilde{H}_1: X\sim \nu \cN(-(1-\nu)\mu,\bI)+ (1-\nu)\cN(\nu \mu,\bI)\ ,\   \mu\sim \varrho \ .
\eeqn
(For brevity, we use $\mu$ in place of $\mudif:=\mu$.)
Observe that $\|\mu\|^2= sr^2$ with probability one under $\varrho$.  
Based on the statement of \prpref{lower_test_sparse}, we assume that 
\beq \label{lower_test_sparse_assumptions}
sr^2=o\left(\sqrt{p/n}\right) \quad \text{and} \quad 
\lim\sup  \frac{sr^2}{\sqrt{\frac{s}{n}\log\big(\frac{ep}{s}\big)} \bigwedge  \frac{s}{n}\log\big(1+ \frac{\sqrt{epn}}{{s}}\big)} < 1\ .
\eeq
Denote by $\tilde{\P}_0$ (resp.~$\tilde{\P}_1$) the distribution of the sample under $H_0$ (resp.~$\tilde{H}_1$), $\tilde{\E}_0$ the expectation with respect to $\tilde{\P}_0$. 
The likelihood ratio is therefore $L:= {\rm d}\tilde{\P}_1/ {\rm d}\tilde{\P}_0$.
By \eqref{LR-risk} above, it suffices to show that $\tilde{\E}_0(L^2) \le 1 + o(1)$, to prove that all test statistics are asymptotically powerless.

For an integer $p \ge 1$, let $\bbH^p = \{-1, 1\}^p$, and with some abuse of notation, for a set $S\subset[p]$, let $\bbH^S$ denote the set of vectors in $\bbR^p$ with support $S$ and nonzero entries equal to $\pm 1$.  We have
\beqn
L
&=&\frac{1}{\binom{p}{s}2^s}\sum_{S} \sum_{\gamma \in \bbH^S}\frac{\prod_{i=1}^n \left(\nu e^{-\frac{\|X_i- r(1-\nu) \gamma\|^2}{2}}+ (1-\nu)e^{-\frac{\|X_i+r \nu\gamma\|^2}{2}}\right)}{\prod_{i=1}^n e^{-\frac{\|X_i\|^2}{2}}}\\
& = &  \frac1{2^{s}} \frac{1}{\binom{p}{s}} \sum_{S} \sum_{\gamma \in \bbH^S} \prod_{i=1}^n \left(\nu e^{\sum_{j \in S} r(1-\nu) \gamma_j X_{i,j}}e^{-sr^2(1-\nu)^2/2} + (1-\nu)e^{- \sum_{j \in S} r\nu  \gamma_j X_{i,j}}\right)e^{-sr^2\nu^2/2} \\
& = &  \frac1{2^{s}} \frac{1}{\binom{p}{s}} \sum_{S} \sum_{\gamma \in \bbH^S} \E_{\vartheta}\left[ \exp\left\{r\sum_{i=1}^n  \sum_{j\in S} \gamma_j\vartheta_i X_{i,j}  -  \frac{sr^2}{2} \sum_{i=1}^n \gamma_j^2\right\}\right] \ .
\eeqn
where $\vartheta:=(\vartheta_1,\ldots, \vartheta_n)$ denotes a vector of $n$ random variables with $\P(\vartheta_i= \nu)=1-\nu$ and $\P(\vartheta_i=-(1-\nu))=\nu$, and $\E_{\vartheta}$ is the expectation with respect to $\vartheta$; 
$X_{i,j}$ denotes the $j$th entry of the vector $X_i \in \bbR^p$, and the sum is over $S \subset [p]$ with $|S| = s$.
Turning to the second moment of $L$, we denote $\bar{\vartheta}$ an independent copy of $\vartheta$. We derive 
\beqn
\tilde{\E}_0[L^2] 
&=& \frac{1}{\left(\binom{p}{s}2^{s}\right)^2}\sum_{S,\bar S} \sum_{\gamma \in \bbH^S,\bar\gamma\in\bbH^{\bar S}}\E_{\vartheta,\bar{\vartheta}}\tilde{\E}_0\left[\exp\left\{r\sum_{j\in S}\sum_{i=1}^n \gamma_j\vartheta_i X_{i,j}+ r\sum_{j\in \bar S}\sum_{i=1}^n \bar\gamma_j\bar\vartheta_i X_{i,j}- \frac{sr^2}{2}\sum_{i=1}^n (\vartheta_i^2+\bar{\vartheta}_i^2)\right\}\right]\\
&=& \frac{1}{\binom{p}{s}^2}  \sum_{S,\bar S} 2^{2|S \triangle \bar S| - 2s} \sum_{\gamma,\bar\gamma\in \bbH^{S\cap \bar S}} \E_{\vartheta,\bar\vartheta}\tilde{\E}_0\left[\exp\left\{r\sum_{j\in S\cap \bar S}\sum_{i=1}^n \left(\gamma_j\vartheta_i+ \bar\gamma_j\bar\vartheta_i\right)X_{i,j} - \frac{|S\cap \bar{S}|r^2}{2} 
\sum_{i=1}^n (\vartheta_i^2+\bar{\vartheta}_i^2)
\right\} \right]\\
& =& \frac{1}{\binom{p}{s}^2}\sum_{S,\bar S}  \frac1{2^{2|S\cap \bar S|}} \sum_{\gamma,\bar\gamma\in\bbH^{S\cap \bar S}}   \E_{\vartheta,\bar\vartheta}\exp\left\{ r^2 \sum_{j \in S \cap \bar S} \sum_{i=1}^n \gamma_{j}\bar\gamma_{j} \vartheta_i\bar\vartheta_i \right\}\ .
\eeqn
Define the random variable $U_i= \vartheta_i\bar\vartheta_i$.  Note that it is centered and takes value $\nu^2$ with probability $(1-\nu)^2$, value $(1-\nu)^2$ with probability $\nu^2$, and value $-\nu(1-\nu)$ with probability $2(1-\nu)\nu$. Relying on the fact that $S$ and $\bar{S}$ are distributed uniformly over the subsets of $[p]$ of size $s$, we denote $\E_{S,\bar{S}}$ the integration with respect to the distribution of $S$ and $\bar{S}$. Given $S\cap \bar{S}$, observe that $(V_j:= \gamma_{j}\bar\gamma_{j}, j\in S\cap\bar{S})$ are independent Rademacher. 
Denoting $U:= \sum_{i=1}^n U_i$, we find that  
\beqn
\tilde{\E}_0[L^2]& =& \E_{S, \bar S} \E_{\gamma, \bar\gamma} \E_{\vartheta, \bar\vartheta} \Big[ \exp\big\{r^2 U \sum_{j\in S\cap \bar{S}} V_j\big\} \Big] \\
&=& \E_{S, \bar S} \E_{\vartheta, \bar\vartheta} \Big[ \cosh\big( r^2  U\big)^{|S \cap \bar S|}  \Big]\  .
\eeqn
Since $|S\cap \bar S| \sim {\rm Hyper}(s, s, p)$, by \lemref{aldous} there exists a binomial random variable $W$ with parameters $(s,s/p)$ and a $\sigma$-field $\cB$ such that $|S\cap \bar S|\sim \E[W|\mathcal{B}]$. By Jensen inequality, we derive
\begin{eqnarray*}
 \tilde{\E}_0[L^2]
 &\leq& \E\left[\left\{1+\frac{s}{p}\big(\cosh\big\{r^2U \big\}-1\big)\right\}^s\right] \\
 &=& 1+\sum_{k=1}^s\binom{s}{k}\left(\frac{s}{p}\right)^k\E\left[\left\{\cosh\left(r^2U\right)-1\right\}^k\right] \ .
\end{eqnarray*}

For any positive $x$ smaller than $1$, $\cosh(x)-1\leq x^2$. 
It follows that 
\beqn
\tilde{\E}_0[L^2] -1  &\leq & \sum_{k=1}^s\binom{s}{k}\left(\frac{s}{p}\right)^k\E\left[ \left(r^4U^2\right)^k + e^{kr^2|U|}  \IND{r^2|U|\geq 1}\right]  
\eeqn
Using the fact that $\binom{s}{k} \le \big(\frac{e s}k\big)^k$, it is enough to prove that the two following terms go to zero:
\begin{eqnarray}
 ({\rm I})&:=& \sum_{k=1}^s\left(\frac{s^2e}{pk}\right)^k\E\left[\left(r^4U^2\right)^k\right]\label{eq:term1} \ , \\
({\rm II})&:=& \sum_{k=1}^s\left(\frac{s^2e}{pk}\right)^k\E\left[ e^{kr^2|U|}\IND{r^2|U|\geq 1}\right] \ .\label{eq:term2}
\end{eqnarray}

Recall that $U = \sum_{i=1}^n U_i$, where the $U_i$'s are  independent centered random variables.  Each term of the corresponding development of $U^{2k}$ into 
$n^{2k}$ term has an expectation of smaller than $1$. The expectation of one such term is zero when one of the $U_i$'s has power one.  Thus, 
\beqn
\E\left[U^{2k}\right]&\leq& \sum_{q=1}^{n\wedge k} \binom{n}{q} q^{2k}\leq \sum_{q=1}^{n\wedge k}\big(\frac{ne}{q}\big)^q q^{2k}\\ 
&\leq&  k (nek)^k\ ,
\eeqn
where we used in the last line that $q^k\leq k^k$ and $q^{k-q}\leq n^{k-q}$.  Incorporating this bound into $({\rm I})$, we get 
\beqn
({\rm I})&\leq &   \sum_{k=1}^s k\left(\frac{6 s^2 r^4 n}{p}\right)^k = o(1)\ ,
\eeqn
since $s^2r^4n/p=o(1)$, because of the left-hand side of \eqref{lower_test_sparse_assumptions}.

Let us turn to the second term~\eqref{eq:term2}. First, by a simple integration, we obtain 
\beqn
\E\left[ e^{kr^2|U|} \IND{r^2|U|\geq 1}\right]
&\le& 2  \E\left[ e^{kr^2 U} \right]
= 2\Big[(1-\nu)^2 e^{kr^2\nu^2}+\nu^2 e^{kr^2(1-\nu)^2}+2\nu(1-\nu) e^{-kr^2\nu(1-\nu)}\Big]^n\\
&\leq& 2\exp\left[n\left(kr^2\wedge \frac{k^2r^4}{2} \right)\right]\ ,
\eeqn
since for $kr^2\leq 2$, we have
\[(1-\nu)^2 e^{kr^2\nu^2}+\nu^2 e^{kr^2(1-\nu)^2}+2\nu(1-\nu) e^{-kr^2\nu(1-\nu)}\leq e^{k^2r^4/2}\ ,\]
by comparing the power expansions. 
We derive a second upper bound of the same expectation when $4kr^4n\leq 1$. By integration by parts and Hoeffding's inequality, 
\beqn
\frac12 \E\left[ e^{kr^2|U|} \IND{r^2|U|\geq 1}\right]& = & e^k\P[U\geq r^{-2} ]+ \int_{r^{-2}}^{\infty} kr^2 e^{kr^2x}\P[U\geq x]dx \\
&\leq & \exp\left[k - \tfrac{1}{2nr^4}\right] +  \int_{r^{-2}}^{\infty} kr^2 \exp\left[x\left(kr^2 - \frac{x}{2n}\right)\right]dx\\
&\leq & e^{-1/(4nr^4)} +  \int_{r^{-2}}^{\infty} kr^2 e^{- x/(4nr^2)} dx \\
&\leq & 2e^{-1/(4nr^4)}\ .
\eeqn

Define $k_*=\lfloor 1/(4nr^4)\rfloor$.  Then
\beqn
({\rm II}) &\leq& 4 \sum_{k=1}^{k_*} \left(\frac{s^2e}{pk}\right)^k e^{-1/(4nr^4)} + 2 \sum_{k=k_*+1}^{s} \left(\frac{s^2e}{pk}\right)^k\exp\left[n\left(kr^2\wedge \frac{k^2r^4}{2} \right)\right]\\
&= &({\rm III})+ ({\rm IV}) \ . 
\eeqn
First, we bound $({\rm III})$. Define the sequence $u_n$  by $u_n:=ns^2r^4/p$ and observe that $u_n$ goes to zero because of the left-hand side of \eqref{lower_test_sparse_assumptions}.

Consider two cases $s^2/p \ge \sqrt{u_n}$ and $s^2/p < \sqrt{u_n}$. When $s^2/p\geq \sqrt{u_n}$ we use the  the fact that $e^x \ge (xe/k)^k$ for any $x \ge 0$ and any $k \ge 1$ (which is established by differentiating with respect to $k$) to obtain $\big(\frac{s^2e}{pk}\big)^k \le e^{s^2/p}$. When $s^2/p < \sqrt{u_n}$, we use $u_n = o(1)$ to get $\big(\frac{s^2e}{pk}\big)^k \le  \frac{e s^2}p$  for any positive integer $k$.
We then derive 
\begin{eqnarray}\nonumber
({\rm III})&\leq& 4 k_*  e^{-1/(4r^4n)}\left[ e^{s^2/p} \IND{s^2\geq\sqrt{u_n}p}+ \frac{e s^2}{p}  \IND{s^2<\sqrt{u_n}p}\right]\\ 
&\lesssim &\frac{1}{r^4n}e^{ -(1-o(1))/(4 r^4n)}\IND{s^2\geq \sqrt{u_n}p}+  \frac{s^2}{pr^4n}e^{-1/(4 r^4n)}  \IND{s^2<\sqrt{u_n}p} \ ,\label{eq:upper_III}
\end{eqnarray}
because of the left-hand side of \eqref{lower_test_sparse_assumptions}.
If $s^2 \ge \sqrt{u_n}p$, $(r^4n)^{-1}=s^2/(pu_n)\geq (u_n)^{-1/2}\rightarrow \infty$ and it follows that  $({\rm III})=o(1)$. If $s^2<\sqrt{u_n}p$, note that $s^2/p=o(1)$ and that the function $x\mapsto xe^{-x}$ is bounded so that $({\rm III})$ also goes to zero.

To conclude, we control the expression $({\rm IV})$.
We have
\beqn
({\rm IV}) &\leq &  2\sum_{k=k_*+1}^{s} \exp\left[k\left\{\log\left(\frac{s^2e}{pk}\right)+ n\left(r^2\wedge \frac{kr^4}{2}\right)\right\}\right] =: 2\sum_{k=k_*+1}^{s}\   \exp\left[k\psi(k)\right]\ .
\eeqn
By differentiation, we obtain 
\[
\sup_{k=k_*+1}^s\psi(k)\leq \psi[(4r^4n)^{-1}]\vee \psi[(s\wedge \tfrac{2}{r^2})\vee (4r^4n)^{-1}]~.
\] 
Since $k_*\to \infty$ (left-hand side of  \eqref{lower_test_sparse_assumptions}),
all we have to prove is that this supremum is negative and bounded away from zero. First,
\beqn
\psi\left[(4r^4n^2)^{-1}\right]\leq \log\left(\frac{4nr^4s^2e}{p}\right)+ 1 \to -\infty\ ,
\eeqn
because of the left-hand side of \eqref{lower_test_sparse_assumptions}. 
We also have 
\beqn
\psi(s)\leq \log\left(\frac{se}{p}\right)+ n(sr^4\wedge r^2)=\log\left(\frac{pe}{s}\right)\left[-1+\frac{n(sr^4\wedge r^2)}{\log(ep/s)}\right]\ ,
\eeqn
which is negative and bounded away from zero because of the right-hand side of  \eqref{lower_test_sparse_assumptions}.

The last case occurs when $(s\wedge \tfrac{2}{r^2})\vee (4r^4n)^{-1} = 2/r^2$, that is when $sr^2\geq 2$ and  $nr^2\geq 1/8$. 
By \eqref{lower_test_sparse_assumptions}, we know that there exists a constant $\delta\in (0,1)$, such that, for $n$ large enough,
\beq\label{eq:sr_2_init}
 sr^2\leq \delta \Big[\sqrt{\frac{s}{n}\log\Big(\frac{ep}{s}\Big)}\bigvee \frac{s}{n}\log\Big(1+ \frac{\sqrt{epn}}{s}\Big)\Big]\ .
\eeq
First, we shall prove that this bound implies 
\beq\label{eq:sr_2}
 r^2\leq 2\delta  \frac{1}{n}\log\Big(1+ \frac{\sqrt{epn}}{s}\Big)\ .
\eeq
We only have to consider the case where $\sqrt{\frac{s}{n}\log\big(\frac{ep}{s}\big)}\geq \frac{s}{n}\log\big(1+ \frac{\sqrt{epn}}{s}\big)$, for otherwise the statement is trivial. 
The condition $sr^2\geq 2$ together with \eqref{eq:sr_2_init} enforces $s\log(ep/s)\geq n$. Since for any $x> 0$ and $a>0$, $\log(1+xa)\geq (x\wedge 1) \log(1+a)$, it follows that 
\beqn 
\frac{s}{n}\log\Big(1+ \frac{\sqrt{epn}}{s}\Big)&\geq & \frac{s}{n}\Big(\sqrt{\frac{n}{s}}\wedge 1 \Big) \log\Big(1+ \sqrt{\frac{ep}{s}}\Big)\\
&\geq & \frac{1}{2} \Big[ \frac{s}{n}\log\Big(\frac{ep}{s}\Big) \bigwedge \sqrt{\frac{s}{n}\log\Big(\frac{ep}{s}}\Big) \Big] \\
&\geq & \frac{1}{2} \sqrt{\frac{s}{n}\log\Big(\frac{ep}{s}\Big)} \ ,
\eeqn 
since $s\log(ep/s)\geq n$. We have proved \eqref{eq:sr_2}.
Since $s^2r^4n/p=o(1)$ (left-hand side of \eqref{lower_test_sparse_assumptions}) and $nr^2\geq 1/8$, it follows that $s\ll \sqrt{pn}$. Together with \eqref{eq:sr_2}, this leads to 
\beqn
\psi(2/r^{2})&=& \log\left(\frac{s^2e}{np}\right)+nr^2+ \log(nr^2)  \\
&\leq & \log\left(\frac{s^2}{enp}\right)+  2\delta  \log\Big(1+ \frac{\sqrt{epn}}{s}\Big) + 2+  \log\Big(2\log\Big(1+ \frac{\sqrt{epn}}{s}\Big) \Big)\\
&\leq& - \log\left(\frac{enp}{s^2}\right) [1- \delta]+ \log\log\Big(\frac{epn}{s^2}\Big) + 4\log(2) +2 \ ,
\eeqn
where we used $\log\Big(1+ \frac{\sqrt{epn}}{s}\Big) \le \log 2 + \frac12 \log\Big(\frac{epn}{s^2}\Big)$. The last expression goes to $-\infty$ and 
we conclude that ${\rm IV}=o(1)$.

\subsection{Proof of \prpref{lower_sparse_unknow_variance}}
\label{sec:proof_lower_sparse_unknow_variance}
We prove simultaneously both results ($s=p$ and $s=o(p)$) by
following the same approach as for the proof of \prpref{lower_test_sparse}.  We use the analogous notation.
Given a vector $\mu$ of norm strictly smaller than one, the matrix $\bSigma_{\mu} :=\bI- \mu \mu^\top$ is positive definite.
We reduce the problem to testing
\beqn
\tilde H_0: X \sim \cN(0,\bI) \quad \text{versus} \quad \tilde{H}_1: X\sim \frac{1}{2}\cN(-\mu,\bSigma_{\mu})+ \frac{1}{2}\cN(\mu,\bSigma_{\mu})\ ,\   \mu\sim \varrho \ .
\eeqn
Observe that $\tfrac{\|\mudif\|^4}{\mudif^{\top}\bSigma_{\mu}\mudif}= \mudif^{\top}\bSigma^{-1}_{\mu} \mudif = 4\frac{\kappa^2}{1-\kappa^2}$ with $\kappa^2:=sr^2<1$.  

Defining the likelihood ratio $L:= {\rm d}\tilde{\P}_1/{\rm d}\tilde{\P}_0$, we know that all tests are asymptotically powerless if $\tilde{\E}_0[L^2]\rightarrow 1$. Thus, it suffices to prove that $\tilde{\E}_0[L^2]\to 1$ 
for $\kappa^2/(1-\kappa^2)$ sufficiently small.

By definition of $\mu$, the eigenvalues of $\bSigma_{\mu}^{-1}$ are all equal to 1 except of one of them equal to $1/(1-\kappa^2)$.
Moreover, for any $a, b \in \bbR^p$, 
\[a^\top \bSigma^{-1}_\mu b = (a^\top b) + \frac{(a^\top \mu)(b^\top \mu)}{1-\kappa^2}\ .\]
Thus, the likelihood $L$ writes as
\begin{eqnarray}
L&=&\frac{(1-\kappa^2)^{-n/2}}{\binom{p}{s}2^{s+n}}\sum_{S} \sum_{\gamma \in \bbH^S} \sum_{\vartheta\in \{-1,1\}^n}\frac{\prod_{i=1}^n\exp\left[-\sum_{j\notin S}\frac{X^2_{i,j}}{2}-\frac{(X_{i,S}-\vartheta_i r\gamma)^\top \bSigma^{-1}_{r \gamma}(X_{i,S}-\vartheta_i r\gamma)}{2}\right]}{\prod_{i=1}^n e^{-\frac{\|X_i\|^2}{2}}}\nonumber\\
& = &  (1-\kappa^2)^{-n/2}\frac{1}{\binom{p}{s}2^{s+n}}\sum_{S} \sum_{\gamma \in \bbH^S} \sum_{\vartheta}\exp\left[-\sum_{i=1}^n \frac{(X_{i,S}-\vartheta_i r\gamma)^\top \bSigma^{-1}_{r \gamma}(X_{i,S}-\vartheta_i r\gamma)}{2}+ \frac{\|X_{i,S}\|^2}{2}\right]\nonumber\\
& =&  \frac{e^{-\frac{n \kappa^2}{2(1-\kappa^2)}}}{(1-\kappa^2)^{n/2}}\frac{1}{\binom{p}{s}2^{s+n}}\sum_{S} \sum_{\gamma \in \bbH^S} \sum_{\vartheta}\prod_{i=1}^n \exp\left[-\frac{r^2 \<X_{i,S}, \gamma\>^2}{2(1-\kappa^2)} + \frac{r}{1-\kappa^2} \vartheta_i\<X_{i,S}, \gamma\> \right] \ .\label{eq:expression_l_hetero}
\end{eqnarray}

Let us turn to the second moment.
\begin{lem}\label{lem:expression_l2_hetero}
Let $T$ be distributed  a sum of $K$ independent Rademacher variables where $K$ is a hypergeometric random variable with parameters $(p,s,s)$. We have
\beq
\tilde{\E}_0[L^2] =  \E\left[\left(1-r^4T^2\right)^{-n/2}\exp\left(\frac{-nr^4T^2}{1-r^4T^2}\right) \cosh^n\left(\frac{r^2T}{1-r^4T^2}\right)\right]\label{eq:expression_l2_hetero}\ . 
\eeq
\end{lem}
This result is proved in \secref{proof-expression_l2_hetero}.
Given this expression, we consider two upper bounds of $\tilde{\E}_0[L^2]$ depending on the value of $\zeta$ (defined in \eqref{zeta}). 

\medskip

\noindent 
{\bf CASE A}: No assumption on $\zeta$. This corresponds to the minimax lower bounds \eqref{eq:lower_unknow_variance} and \eqref{eq:lower_sparse_unknow_variance}. We assume in the following that
\begin{eqnarray}
\label{eq:A1}\frac{\kappa^2}{1-\kappa^2} \ll (p/n)^{1/4} \quad & \text{if} & \quad s=p\ , \\ 
\label{eq:A2}\lim \sup  \frac{\kappa^2}{1-\kappa^2} \left(\frac{n}{s\log(p/s)}\right)^{1/4} < \frac12 \quad & \text{if} & \quad s=o(p) \ .
\end{eqnarray}

Using $-\log(1-x)\leq x/(1-x)$ for all $x \in [0,1)$ and $\cosh(x)\leq \exp(x^2/2)$ for all $x \ge 0$, we obtain 
\begin{eqnarray}
\tilde{\E}_0[L^2]&\leq& \E\left[\exp\left(\frac{nr^8T^4}{2(1-r^4T^2)^2}\right)\right]\label{eq:upper_L2T}\\
&\leq & \E\left[\exp\left(\frac{nr^8T^4}{2(1-s^2r^4)^2}\right)\right]\ ,\nonumber
\end{eqnarray}
where $T$ is as in \lemref{expression_l2_hetero}, meaning it has the distribution of a sum of $K$ independent Rademacher variables where $K \sim {\rm Hyper}(s,s,p)$, and we used the fact that $T \le s$ and $sr^2<1$. 
Applying \lemref{aldous}, $K \sim \mathbb{E}(W|\mathcal{B})$ where $W \sim \Bin(s, s/p)$ and $\mathcal{B}$ is some suitable $\sigma$-algebra. Let $V$ be the sum of $W$ independent Rademacher variables. Consequently, $T$ has the same distribution as $\E[V|\cB]$. 
Then, Jensen's inequality yields 
\beqn
\tilde{\E}_0[L^2]&\leq&  \E\left[\exp\left(\frac{nr^8V^4}{2(1-s^2r^4)^2}\right)\right]\ .
\eeqn
Let us upper bound the deviations of $V$. 
We use Hoeffding's inequality: for a positive integer $k \le s$,
\beqn
\P[|V|\geq k]&=& \sum_{q=k}^s \P[W=q]\P[|V|\geq k|W=q]\\
&\leq& \sum_{q=k}^s \P[W=q]2e^{-\frac{k^2}{2q}}
\ \leq \ 2\P[W\geq k]e^{-\frac{k^2}{2s}}\ .
\eeqn
If $k/s\geq e^2 s/p$, then we use Lemma \ref{lem:chernoff} to derive
\beqn 
\P[W\geq k]&\leq& \exp\left[-sH_{s/p}(k/s)\right]\leq \exp\left[-k\left( \log\left(\frac{kp}{s^2}\right)-1+ \frac{s^2}{kp}\right)\right]\\
&\leq & \exp\left[-\frac{k}{2} \log\left(\frac{kp}{s^2}\right)\right]\ ,
\eeqn
since $\log(x)-1+1/x\geq \log(x)/2$ when $x\geq e^2$. We have proved that 
\beq\label{eq:deviations_V}
\P[V\geq k]\leq  e^{-\frac{k^2}{2s}}\exp\left[-\frac{k}{2} \log\left(\frac{kp}{s^2}\right) \IND{k p/s^2 \geq e^2}\right]\ .
\eeq
Define $k_0:= \frac{s^{3/4}}{\log(ep/s)^{1/2}}$ and $k_1= \frac{s}{\log(ep/s)^{1/2}}$. We decompose the second moment into 
\[\tilde{\E}_0[L^2] \le {\rm I} + 2{\rm II} + 2{\rm III}\ ,\]
where
\beqn
{\rm I} &:=& \exp\left(\frac{nr^8k_0^4}{2(1-s^2r^4)^2}\right)\ , \\ 
{\rm II}&:=& \sum_{k=\lceil k_0\rceil}^{\lfloor k_1\rfloor} \exp\left(\frac{nr^8k}{2(1-s^2r^4)^2}\right)\P[V=k]\ ,\\
{\rm III}&:=  &\sum_{k=\lfloor k_1\rfloor +1}^{s} \exp\left(\frac{nr^8k}{2(1-s^2r^4)^2}\right)\P[V=k]
\eeqn
Relying on \eqref{eq:A1} and \eqref{eq:A2}, we have
\[{\rm I} = \exp\left(\frac{n\kappa^8}{2 s \log^2(ep/s) (1-\kappa^4)^2}\right)= 1+o(1)\ .\]
Let us now study the two remaining terms depending on the value of $s$.\\

\noindent 
{\bf CASE A.1}: $\log(ep/s)\leq s^{1/4}$.  We have 
\beqn
{\rm II} &\leq& \sum_{k=\lceil k_0\rceil}^{\lfloor k_1\rfloor } \exp\left[\frac{nr^8k^4}{2(1-s^2r^4)^2}- \frac{k^2}{2s}\right]\\
&\leq & \sum_{k=\lceil k_0\rceil}^{\lfloor k_1\rfloor } \exp\left[\frac{k^2}{2s}\left(\frac{nr^8sk^2}{(1-\kappa^4)^2} - 1\right)\right]\\
&\leq & \sum_{k=\lceil k_0\rceil}^{\lfloor k_1\rfloor } \exp\left[\frac{k^2}{2s}\left(\frac{nr^8s^3}{\log\left(\frac{ep}{s}\right)(1-\kappa^4)^2} - 1\right)\right]\\
&\leq & \sum_{k=\lceil k_0\rceil}^{\lfloor k_1\rfloor } \exp\left[-\frac{k^2}{4s}(1+o(1))\right]\ ,
\eeqn
where we use $k\leq k_1$ in the third  line 
and $\lim\sup \frac{n\kappa^8}{s\log\left(\frac{ep}{s}\right)(1-\kappa^4)^2}< \lim\sup \frac{n\kappa^8}{s\log\left(\frac{ep}{s}\right)(1-\kappa^2)^2} < 1/16$ (Conditions \eqref{eq:A1} and \eqref{eq:A2}) in the fourth line.  
Since $k_0^2/s\geq s^{1/4}\to \infty$, the last sum goes to $0$ and therefore ${\rm II}=o(1)$. Observe that $k_1=s$ if $p=s$, so that we only need to consider ${\rm III}$ when $s=o(p)$. Applying again \eqref{eq:deviations_V} and noting that $k_1p/s^2\geq p/s\log^{-1/2}(ep/s)\geq e^2$, 
we have

\beqn
{\rm III} &\leq& \sum_{k=\lfloor k_1\rfloor +1}^{s } \exp\left[k\left(\frac{nr^8k^3}{2(1-s^2r^4)^2}- \frac{1}{2}\log\left(\frac{pk}{s^2}\right)\right)\right]\\
 &\leq& \sum_{k=\lfloor k_1\rfloor +1}^{s } \exp\left[k\left(\frac{n\kappa^8}{2s(1-\kappa^4)^2}- \frac{1}{2}\log\left(\frac{p}{s}\right)+ \frac{1}{4}\log\log\left(\frac{ep}{s}\right)\right)\right]\\
&\leq & \sum_{k=\lfloor k_1\rfloor +1}^{s } \exp\left[- \frac{k}{4}\log\left(\frac{ep}{s}\right)(1+o(1))\right]=o(1)\ ,
\eeqn
where we used Condition \eqref{eq:A2} and $\log\left(\frac{ep}{s}\right)\to \infty$ in the last line.

\noindent 
{\bf CASE A.2}: $\log(ep/s)> s^{1/4}$. This entails  $p/s^2\to \infty$ and for $k\geq 1$, $\log(s/k)\leq \log(s)=o[\log(ep/s)]$. Applying, as before, \eqref{eq:deviations_V}, we obtain
\beqn
{\rm II} + {\rm III} &\leq& \sum_{k=\lceil k_0 \rceil }^{s } \exp\left[k\left(\frac{nr^8k^3}{2(1-s^2r^4)^2}- \frac{1}{2}\log\left(\frac{pk}{s^2}\right)\right)\right]\\
&\leq & \sum_{k=1 }^{s } \exp\left[k \left(\frac{n\kappa^8}{2s(1-\kappa^4)^2}- \frac{1}{2}\log\left(\frac{p}{s}\right)+ o\left(\log\left(\frac{ep}{s}\right)\right)\right)\right]\\
&\leq & \sum_{k= 1}^{s } \exp\left[- \frac{k}{4}\log\left(\frac{ep}{s}\right)(1+o(1))\right]=o(1)
\eeqn
where we used Condition \eqref{eq:A2} and $\log\left(\frac{ep}{s}\right)\to \infty$ in the last line.

\medskip

\noindent 
{\bf CASE B}: $p\gg n$ if $s=p$ or  $\lim\sup s\log(ep/s)/n \geq c$ if $s=o(p)$, for a numerical constant $c$.  This corresponds to the minimax lower bounds \eqref{eq:lower_unknow_variance_TGD} and \eqref{eq:lower_sparse_unknow_variance_TGD}. 
 We assume in the following 
\begin{eqnarray}
\label{eq:B1}\frac{\kappa^2}{1-\kappa^2}\ll \exp\left[\frac{p}{64n}\right] &\text{ if }& \quad s=p\ , \\ 
\label{eq:B2}\lim \sup  \frac{\kappa^2}{1-\kappa^2} \exp\left[- \frac{s\log(p/s)}{16n}\right] < 1 &\text{ if }&  \quad s=o(p) \ .
\end{eqnarray}

We again wield Lemma~\ref{lem:expression_l2_hetero} to control the $\tilde{\E}_0[L^2]$.
We use $\cosh(x)\leq \exp(x)$, valid for all $x \ge 0$, to derive
\beqn
\exp\left[\frac{-nr^4T^2}{1-r^4T^2}\right] \cosh^n\left[\frac{r^2T}{1-r^4T^2}\right]\leq \exp\left[\frac{nr^2T}{1+r^2T}\right]\leq e^{nr^2T}
\eeqn
Coming back to \eqref{eq:expression_l2_hetero} and relying on the same bound that got us \eqref{eq:upper_L2T} , we obtain
\begin{eqnarray}
\tilde{\E}_0[L^2]&\leq& \E\left[\exp\left(\frac{nr^8T^4}{2(1-r^4T^2)^2}\right)\IND{|T|\leq s/2}\right] \nonumber + \E\left[\left(1-r^4T^2\right)^{-n/2}e^{nr^2T} \IND{|T|\geq s/2}\right]\nonumber \\
&\leq & \E\left[e^{nr^8T^4}\right]+ (1- \kappa^4)^{-n/2}\P^{1/2}[|T|\geq s/2]\E^{1/2}\left[e^{2nr^2T}\right]\nonumber \\
&\leq &\E\left[e^{nr^8V^4}\right]+ {\rm IV}\ , \quad {\rm IV}:= (1- \kappa^4)^{-n/2}\P^{1/2}[|T|\geq s/2]\E^{1/2}\left[e^{2nr^2V}\right]\ ,
\label{eq:lower_TGD}
\end{eqnarray}
where in the second line we use fact that $\kappa < 1$, which in particular implies $1-r^4T^2 \geq 3/4$ when $T\leq s/2$, and the Cauchy-Schwarz inequality, and in the third line we used Jensen inequality with $T\sim \E[V|\cB]$, where the $\sigma$-algebra $\cB$ is as in CASE A above.

Arguing exactly as in Case $A$, we have $\E[e^{nr^8V^4}]=1+o(1)$ if $\kappa^2\ll (p/n)^{1/4}$ (for $p=s$) or if $\kappa^2\leq C(s\log(p/s)/n)^{1/4}$. 
Consequently, we have $\E[e^{nr^8V^4}]=1+o(1)$ for any $\kappa^2\leq 1$ if either $p/n\to \infty$ or $s\log(p/s)/n$ is large compared to one. 

It therefore suffices to prove that ${\rm IV} = o(1)$. 
First, 
\begin{eqnarray}
 \E\left[e^{2nr^2V}\right]= \left(1+ \frac{s}{p}(\cosh(2nr^2)-1)\right)^s
\leq \exp\left[\frac{s^2}{p}(\cosh(2n/s)-1)\right]\ ,\label{eq:upper_exp_T}
\end{eqnarray}
since $\kappa^2=sr^2\leq 1$. 

\medskip

\noindent
{\bf CASE B.1}:   $s=p$.  Then, $T$ follows the distribution of a sum of $p$ independent Rademacher variables. By Hoeffding inequality, \[P\left[|T| \geq \frac{s}{2}\right]\leq 2 e^{-p/16}\ .\] 
Since we assume that $n/p=o(1)$, we have $\cosh(2n/p)-1\leq 1/32$ eventually, which with \eqref{eq:upper_exp_T} then implies that
\beqn
{\rm IV} \leq 2(1- \kappa^4)^{-n/2} e^{-p/64}\ , 
\eeqn
which goes to zero by Condition \eqref{eq:B1}.
\medskip

\noindent
{\bf CASE B.2}:   $s=o(p)$. Then, $|T|$ is stochastically upper bounded by a binomial distribution with parameter $(s, s/(p-s))$. Applying Chernoff inequality (Lemma \ref{lem:chernoff}), we derive
\beqn
P\left[|T|\geq \frac{s}{2}\right]&\leq &\exp\left[-sH_{s/(p-s)}\left(\frac{1}{2}\right)\right]\\
&\leq & \exp\left[-\frac{s}{2}\left\{\log\left(\frac{p-s}{2 s}\right)-1+\frac{2s}{p-s} \right\} \right]\\
&\leq & \exp\left[-\frac{s}{2}\left\{\log\left(\frac{p}{s}\right)+ O(1)\right\}\right]\ .
\eeqn
Combining this upper bound with \eqref{eq:upper_exp_T}, we obtain
\beqn
{\rm IV} &\leq& (1-\kappa^4)^{-n/2}\exp\left[- \frac{s}{4}\left\{\log\left(\frac{p}{s}\right)+ O(1)- \frac{2s}{p}e^{2n/s} \right\}\right]\\
 &\leq & (1-\kappa^4)^{-n/2} \exp\left[- \frac{s}{4}\left\{\log\left(\frac{p}{s}\right)+ O(1) \right\}\right]\\ &=&o(1)\ ,
\eeqn
where we use in the second line $\lim\sup s\log(p/s)/n \geq c$ for some numerical constant $c>0$ and Condition \eqref{eq:B2} in the last line.

\subsubsection{Proof of \lemref{expression_l2_hetero}} \label{sec:proof-expression_l2_hetero}
We start from \eqref{eq:expression_l_hetero}:
\[
\tilde\E_0[L^2] = \frac{e^{-\frac{n \kappa^2}{1-\kappa^2}}}{(1-\kappa^2)^{n}}\frac{1}{\binom{p}{s}^2 2^{2s+2n}} 
\sum_{S, \bar S} \sum_{\gamma \in \bbH^S, \bar\gamma \in \bbH^{\bar S}} \sum_{\vartheta, \bar\vartheta} \prod_{i=1}^n \Xi_i \ ,
\]
where
\beqn
\Xi_i 
&:=& \tilde\E_0 \exp\left[-\frac{r^2 \<X_{i,S}, \gamma\>^2}{2(1-\kappa^2)} + \frac{r}{1-\kappa^2} \vartheta_i\<X_{i,S}, \gamma\> -\frac{r^2 \<X_{i,\bar{S}}, \bar\gamma\>^2}{2(1-\kappa^2)} + \frac{r}{1-\kappa^2} \bar\vartheta_i\<X_{i,\bar{S}}, \bar\gamma\>\right] \\
&=& \tilde\E_0 \exp\left[-\frac{\kappa^2}{2(1-\kappa^2)} V_i^2 + \frac{\kappa}{1-\kappa^2} V_i -\frac{\kappa^2}{2(1-\kappa^2)} \bar V_i^2 + \frac{\kappa}{1-\kappa^2} \bar V_i\right]\ ,
\eeqn
where $V_i := \frac{\vartheta_i}{\sqrt{s}} \<X_{i,S}, \gamma\>$ and $\bar V_i := \frac{\bar\vartheta_i}{\sqrt{s}} \<X_{i,\bar{S}}, \bar\gamma\>$.
Fix $S, \bar S$ and $i \in [n]$, and let $\alpha_i =\frac1s \vartheta_i \bar\vartheta_i \sum_{j\in S\cap \bar S} \gamma_j \bar\gamma_j$, and then $U_i$ such that $\bar V_i =\alpha_i  V_i +\sqrt{1-\alpha_i^2} U_i$. 
Observe that $U_i$ and $V_i$ are iid standard normal.
We use the fact that, for $Z \sim \cN(0,1)$ and any $a < 1/2$ and $b \in \bbR$, 
\[\E[\exp(a Z^2 + b Z)] = (1-2a)^{-1/2} \exp[b^2/(2-4a)] \ ,\] 
to derive
\beqn
\Xi_i 
&=&\tilde{\E}_0 \exp\left[-\frac{\kappa^2}{2(1-\kappa^2)} (V_i^2 + (\alpha_i V_i +\sqrt{1-\alpha_i^2}U_i)^2) + \frac{\kappa}{1-\kappa^2} (V_i + \alpha_i V_i +\sqrt{1-\alpha_i^2}U_i) \right] \\
&=& \tilde{\E}_0 \exp\left[-\frac{\kappa^2(1-\alpha_i^2)}{2(1-\kappa^2)}U_i^2 + \frac{\kappa \sqrt{1-\alpha_i^2}}{1-\kappa^2} (1-\kappa \alpha_i V_i) U_i -\frac{\kappa^2(1+\alpha_i^2)}{2(1-\kappa^2)} V_i^2 + \frac{\kappa(1+\alpha_i)}{1-\kappa^2} V_i\right] \\
&=& \sqrt{\frac{1-\kappa^2}{1-\kappa^2\alpha_i^2}} \ \tilde{\E}_0 \exp\left[\frac{\kappa^2(1-\alpha_i^2)\left(1-\kappa\alpha_i V_i\right)^2}{2(1-\kappa^2)(1-\kappa^2\alpha_i^2)} -\frac{\kappa^2(1+\alpha_i^2)}{2(1-\kappa^2)} V_i^2 + \frac{\kappa(1+\alpha_i)}{1-\kappa^2} V_i\right] \\
&=& \sqrt{\frac{1-\kappa^2}{1-\kappa^2\alpha_i^2}}
e^{\frac{\kappa^2(1-\alpha_i^2)}{2(1-\kappa^2)(1-\kappa^2\alpha_i^2)}} 
\tilde\E_0 \exp\left[-\frac{\kappa^2 (1 + \alpha_i^2 -2 \kappa^2\alpha_i^2)}{2(1-\kappa^2)(1-\kappa^2\alpha_i^2)} V_i^2 + \frac{\kappa (1 + \alpha_i) (1 -\kappa^2 \alpha_i)}{(1-\kappa^2)(1-\kappa^2\alpha_i^2)} V_i\right] \\
&=& \sqrt{\frac{1-\kappa^2}{1-\kappa^2\alpha_i^2}} 
e^{\frac{\kappa^2(1-\alpha_i^2)}{2(1-\kappa^2)(1-\kappa^2\alpha_i^2)}} 
\sqrt{\frac{(1-\kappa^2)(1-\kappa^2\alpha_i^2)}{1-\kappa^4\alpha_i^2}}
e^{\frac{\kappa^2 (1+\alpha_i)^2(1-\kappa^2\alpha_i)^2}{2(1-\kappa^2)(1-\kappa^2\alpha_i^2)(1-\kappa^4\alpha_i^2)}}\ .\\
&=& \frac{1-\kappa^2}{\sqrt{1-\kappa^4\alpha_i^2}}\exp\left[\frac{\kappa^2(1-\kappa^2\alpha_i^2)}{(1-\kappa^2)(1-\kappa^4\alpha_i^2)}+\frac{\kappa^2\alpha_i}{1-\kappa^4\alpha_i^2}\right]
\eeqn
Gathering this expression with the definition of $\tilde{\E}_0[L^2]$ and defining $\alpha=\tfrac{1}{s}\sum_{j\in S\cap \bar S} \gamma_j \bar\gamma_j$ and  $U'=\sum_{i=1}^n \vartheta_i\bar\vartheta_i$ which is  distributed as the sum of $n$ independent Rademacher random variables, we get
\begin{eqnarray}
\tilde{\E}_0[L^2]&=&\E_{\alpha, U'}\left[\left[1-\kappa^4\alpha^2\right]^{-n/2}\exp\left[\frac{-n\kappa^4\alpha^2}{1-\kappa^4\alpha^2}+ \frac{\kappa^2U'\alpha}{1-\kappa^4\alpha^2}\right]\right]\nonumber \\
& = & \E_{\alpha}\left[\left[1-\kappa^4\alpha^2\right]^{-n/2}\exp\left[\frac{-n\kappa^4\alpha^2}{1-\kappa^4\alpha^2}\right]\cosh^n\left[\frac{\kappa^2\alpha}{1-\kappa^4\alpha^2}\right]\right]\nonumber\ . 
\end{eqnarray}
Observing that $s\alpha$ follows the distribution of a sum of $K$ independent Rademacher variables where $K$ is an Hypergeometric random variable with parameters $(p,s,s)$ conclude the proof.

\subsection{Proof of Propositions \ref{prp:lower_sparse_unknow_variance_asymmetric}}

Fix some mixing weights $\nu \in (0,1/2)$. 
We prove simultaneously both propositions ($s=p$ and $s=o(p)$) following closely the arguments of \secref{proof_lower_sparse_unknow_variance}.  
We use the same prior $\varrho$ and almost notation, except that here $\bSigma_{\mu}=\bI- \nu(1-\nu)\mu \mu^\top$ 
and  
\[
\tilde H_0: X \sim \cN(0,\bI) \quad \quad \tilde{H}_1: X=  \xi\mu + \bSigma_{\mu}^{1/2}Z\ ,\   \mu\sim \varrho,\ Z\sim \cN(0,1) \ ,
\]
where $\xi$ is a variable taking values in $\{-\nu,1-\nu\}$ with probability $1-\nu$ and $\nu$, respectively.
Observe that $\mudif=\mu$ and $\tfrac{\|\mu\|^4}{\mu^{\top}\bSigma_{\mu}\mu}= \mu^{\top}\bSigma^{-1}_{\mu} \mu = \frac{\kappa^2\nu^{-1}(1-\nu)^{-1}}{1-\kappa^2}$ with $\kappa^2:=\nu(1-\nu)sr^2<1$. 
Denote, as before, $\tilde{\P}_0$ (resp.~$\tilde{\P}_1$) the distribution of the sample under $\tilde H_0$ (resp.~$\tilde{H}_1$), $\tilde{\E}_0$ the expectation with respect to $\tilde{\P}_0$.  Defining the likelihood ratio $L:= {\rm d}\tilde{\P}_1/{\rm d}\tilde{\P}_0$, all tests are asymptotically powerless if $\tilde{\E}_0[L^2]\rightarrow 1$. Thus, it suffices to prove that $\tilde{\E}_0[L^2]\to 1$ when $s=o(p)$ and $\kappa^2/(1-\kappa^2)\le c\zeta^{1/3}$ for a sufficiently small constant $c$, or when $s=p$ and $\kappa^2/(1-\kappa^2)\ll (p/n)^{1/3}$.

Given $\gamma\in \bbH^S$ and $r>0$, denote $\bSigma_{r, \gamma}$ the covariance matrix $\bSigma_{\mu}$ with $\mu=r\gamma$.
Noting that, for any $a, b \in \bbR^p$, we have
\[a^\top \bSigma^{-1}_\mu b = (a^\top b) + \nu (1-\nu) \frac{(a^\top \mu)(b^\top \mu)}{1-\kappa^2}\ ,\]
we express the likelihood ratio $L$ as 
\begin{eqnarray*}
L&=&\frac{(1-\kappa^2)^{-n/2}}{\binom{p}{s}2^{s}}\sum_{S} \sum_{\gamma\in\bbH^S} \sum_{\vartheta}\prod_{i=1}^n \tilde{\P}_0[\xi=\vartheta_i] \frac{e^{-\frac{(X_{i,S}-\vartheta_i r\gamma)^\top \bSigma^{-1}_{r, \gamma}(X_{i,S}-\vartheta_i r\gamma)}{2}}}{ e^{-\frac{\|X_{i,S}\|^2}{2}}}\nonumber\\
& =&  \frac{(1-\kappa^2)^{-n/2}}{\binom{p}{s}2^{s}}\sum_{S} \sum_{\gamma\in\bbH^S} \sum_{\vartheta}\prod_{i=1}^n \tilde{\P}_0[\xi=\vartheta_i]
e^{-\frac{\vartheta_i^2 s r^2}{2(1-\kappa^2)}} 
e^{-\frac{\kappa^2}{2(1-\kappa^2)} \left(\tfrac1{\sqrt{s}} \<X_{i,S}, \gamma\>\right)^2}  
e^{\frac{r}{1-\kappa^2}\vartheta_i \<X_{i,S}, \gamma\>}\ , \nonumber
\end{eqnarray*}
where the sum is over $S \subset [p]$ of size $|S| = s$ and over $\vartheta = (\vartheta_1, \dots, \vartheta_n) \in \{-\nu, 1-\nu\}^n$.
Turning to the second moment, we use the same approach as in Lemma \ref{lem:expression_l2_hetero}. After some tedious computations, we obtain the following formula.
Let $T$ be distributed  a sum of $K$ independent Rademacher variables, where $K$ is an hypergeometric random variable with parameters $(s,s,p)$. We have
\begin{eqnarray*}
\tilde{\E}_0[L^2]& =&  \E\left[\left[1-\kappa^4T^2/s^2\right]^{-n/2} U^n\right]\ , \\
U &:= &(1-\nu)^2 e^{-\frac{\nu \kappa^4 T^2}{(1-\nu)s^2(1- \kappa^4T^2/s^2) } + \frac{\kappa^2 \nu T}{s(1-\nu)(1- \kappa^4T^2/s^2)}}+ \nu^2e^{-\frac{(1-\nu) \kappa^4 T^2}{\nu s^2(1- \kappa^2T^4/s^2) } + \frac{\kappa^2(1- \nu) T}{s\nu(1- \kappa^4T^2/s^2)}} \nonumber\\
&& + \ 2\nu(1-\nu)  e^{(1-\frac{1}{2\nu(1-\nu)})\frac{ \kappa^4 T^2}{s^2(1- \kappa^4T^2/s^2) } - \frac{\kappa^2T}{s(1-\kappa^4T/s^2)} }\ . \nonumber
\end{eqnarray*}
By assumption, $\kappa=o(1)$ so that we can upper bound $U$  using Taylor formula: for $n$ large enough
\[U\leq  1 - \frac{1}{2}\kappa^4 \frac{T^2}{s^2} + C \kappa^6 \frac{|T|^3}{s^3}\ , \]
where $C$ is a positive  constant that does not depend on $T$ (but depends on $\nu$). Relying again on a Taylor development of $[1-\kappa^4T^2/s^2]^{-1/2}$, we conclude that 
\beqn
\tilde{\E}_0[L^2]\leq \E \left[\left(1+ C \kappa^6 \frac{|T|^3}{s^3} \right)^n\right]\leq \E \left[\exp\left( Cn \kappa^6 \frac{|T|^3}{s^3} \right)	 \right]\ .
\eeqn
Using the same comparison argument as in \secref{proof_lower_sparse_unknow_variance} leads to 
\beqn
\tilde{\E}_0[L^2]\leq  \E \left[\exp\left( Cn \kappa^6 \frac{|V|^3}{s^3} \right)	\right]\ , 
\eeqn
where $W$ is a binomial random variable with parameters $(s,s/p)$  and  $V$ is the sum of $W$ independent Rademacher variables.
Define $k_0:= \frac{s^{2/3}}{\log(ep/s)^{2/3}}$ and $k_1= \frac{s}{\log(ep/s)}$. Note that $k_0$ and $k_1$ are slightly different from \secref{proof_lower_sparse_unknow_variance}. We decompose the second moment into the sum 
\[
\tilde{\E}_0[L^2] \leq {\rm I} + 2\cdot {\rm II} + 2\cdot {\rm III},
\] 
where
\[
{\rm I} := \exp\left(Cn \kappa^6 \frac{k_0^3}{s^3}\right)\ , 
\qquad 
{\rm II} := \sum_{k=\lceil k_0\rceil}^{\lfloor k_1\rfloor} \exp\left(Cn \kappa^6 \frac{k^3}{s^3}\right)\P[V=k]\ ,\]
and
\[ {\rm III} :=  \sum_{k=\lfloor k_1\rfloor +1}^{s} \exp\left(Cn \kappa^6 \frac{k^3}{s^3}\right)\P[V=k]\ . 	
\]
Considering separately the case $\log(ep/s)\leq s^{1/8}$ and $\log(ep/s)\geq s^{1/8}$ and following closely the arguments given in \secref{proof_lower_sparse_unknow_variance}, we prove again that ${\rm I}=1+o(1)$, ${\rm II}=o(1)$ and ${\rm III}=o(1)$.

\subsection{Proof of \prpref{lower_mahalanobis}} \label{sec:proof_lower_mahalanobis}

Without loss of generality, we may assume that $\mudif = (r,0,\ldots,0)$ for some positive number $r>0$.  
As in the previous proofs, we reduce the composite alternative to a simple alternative by putting a prior on $\bSigma$. 
Given $t\in (0,1)$ and a unit vector $v \in \bbR^p$, define the covariance $\bSigma_{t,v}:=\bI-tv v^{\top}$. 
In this section, we let $\varrho$ denote the uniform distribution over $\big\{-\frac1{\sqrt{p}},\frac1{\sqrt{p}}\big\}^{p}$.  We use it as a prior to reduce the testing problem to
\beqn
H^\dag_0: X \sim \cN(0,\bI) \quad \quad \tilde{H}^\dag_1: X\sim \tfrac{1}{2}\cN(-\tfrac{1}{2}\mudif,\bSigma_{t,v})+ \tfrac{1}{2}\cN(\tfrac{1}{2}\mudif,\bSigma_{t,v})\ ,\   v\sim \varrho \ .
\eeqn
Observe that 
\beq\label{eq:expression_mahal}
\mudif^{\top}\bSigma^{-1}_{\mu} \mudif = r^2\left[1+ \frac{t}{p(1-t)}\right]\ .
\eeq
Let $\tilde{\P}_0$ denote the distribution of the sample under $H^\dag_0$ and $\tilde{\P}_{r,t}$ its distribution under $\tilde{H}^\dag_1$; moreover, let $\bar{\P}_{r,t}$ be its distribution when $X\sim \cN(\frac{1}{2}\mudif,\bSigma_{t,v})$ with $v\sim \varrho$. Observe that $\bar{\P}_{0,0} = \tilde{\P}_{0}$.

Let $\|\cdot\|_{\rm TV}$ denote the total variation metric.  We claim that 
\beq \label{TV-bound1}
\|\tilde{\P}_{0}-\tilde{\P}_{r,t}\|_{\rm TV}\leq \|\bar{\P}_{0,0}-\bar{\P}_{r,t}\|_{\rm TV} \ .
\eeq 
We start by noticing that, if $(X_1, \dots, X_n) \sim \bar{\P}_{r,t}$, and $\epsilon_1,\ldots, \epsilon_n$ are iid Rademacher variables, independent of $X_1, \dots, X_n$, then 
$(\epsilon_1X_1,\ldots, \epsilon_nX_n)\sim\tilde{\P}_{r,t}$.  
Let $\P_\eps$ denote the distribution of $(\eps_1, \dots, \eps_n)$.  
We then have
\[\|\tilde{\P}_{0}-\tilde{\P}_{r,t}\|_{\rm TV}
\le \|\bar{\P}_{0} \otimes \P_\eps -\bar{\P}_{r,t} \otimes \P_\eps \|_{\rm TV} = \|\bar{\P}_{0} -\bar{\P}_{r,t}\|_{\rm TV} \ ,\]
where the equality is by simple integration with respect to $\P_\eps$, and the inequality is due to the following contraction property of the total variation metric.

\begin{lem}
Consider two probability spaces $(S, \cA)$ and $(T, \cB)$.  Suppose $\P_1$ and $\P_2$ are probability distributions on $(\cS, \cA)$, and that $g: \cS \to \cT$ is a measurable function.  Then
\[\|g \circ \P_1 - g \circ \P_2\|_{\rm TV} \le \|\P_1 - \P_2\|_{\rm TV} \ .\]
\end{lem}

\begin{proof}
By definition,
\[\|g \circ \P_1 - g \circ \P_2\|_{\rm TV} = \sup_{B \in \cB} \big| \P_1(g^{-1}(B)) - \P_2(g^{-1}(B)\big| \le \sup_{A \in \cA} \big|\P_1(A) - \P_2(A) \big| = \|\P_1 - \P_2\|_{\rm TV} \ . \] 
\end{proof}

Then, by the triangle inequality and then translation invariance,
\beqn
\|\bar{\P}_{0,0}-\bar{\P}_{r,t}\|_{\rm TV} 
&\leq & \|\bar{\P}_{0,0}- \bar{\P}_{r,0}\|_{\rm TV}+ \|\bar{\P}_{r,0}-\bar{\P}_{r,t}\|_{\rm TV}\\
&\leq & \|\bar{\P}_{0,0}- \bar{\P}_{r,0}\|_{\rm TV}+ \|\bar{\P}_{0,0}-\bar{\P}_{0,t}\|_{\rm TV}\ .
\eeqn

The first total variation distance in the last line is between two isotropic Gaussian distributions with different means.  Some calculations and an application of the Cauchy-Schwarz inequality lead to
\[
 \|\bar{\P}_{0,0}- \bar{\P}_{r,0}\|_{\rm TV}\leq (e^{r^2n/4}-1)^{1/2}\ ,
\]
which goes to zero if $r=o(1/\sqrt{n})$. 

For the second total variation distance, we have the following.
\begin{lem}\label{lem:TV_variance}
If $p\gg n$ and $(1-t)^{-1}\leq e^{p/(2n)}$, then $\|\bar{\P}_{0,0}-\bar{\P}_{0,t}\|_{\rm TV}=o(1)$.
\end{lem}

In conclusion, as long as $r=o(1/\sqrt{n})$ and $(1-t)^{-1}\leq e^{p/(2n)}$, then no test is able to distinguish $H^\dag_0$ and  $\tilde{H}^\dag_1$. Translating these bounds in terms of the Mahalanobis distance \eqref{eq:expression_mahal} leads to the desired result.

\subsubsection{Proof of Lemma \ref{lem:TV_variance}}
As usual, we use  Cauchy-Schwarz inequality to get $\|\bar{\P}_{0,0}-\bar{\P}_{0,t}\|^2_{\rm TV}\leq \E_0[L_t^2]-1$, where $L_t$ is the likelihood of $\bar{\P}_{0,t}$ with respect to $\bar{\P}_{0,0}$, and $\E_0$ denotes the expectation with respect to $\bar{\P}_{0,0}$. This likelihood writes as 
\beqn
L_t&=& \frac{1}{2^p}\sum_{\vartheta\in \{-1,1\}^p}\frac{1}{(1-t)^{n/2}}\exp\left[-\sum_{i=1}^n\frac12 X^{\top}_i \left\{\left(\bI-\frac{t}{p}\vartheta\vartheta^\top\right)^{-1}-\bI\right\}X_i\right]\\
&=& \frac{1}{2^p}\sum_{\vartheta\in \{-1,1\}^p}\frac{1}{(1-t)^{n/2}}\exp\left[-\frac12 \frac{t}{p(1-t)} \sum_{i=1}^n(\vartheta^\top X_i)^2 \right]\ .
\eeqn
Thus, its second moment equals
\beqn
\E_{0}[L_t^2]
&=& \frac{1}{(1-t)^n}\frac{1}{2^{2p}}\sum_{\vartheta,\bar\vartheta\in \{1,1\}^p} \prod_{i=1}^n \E_0\left[\exp\left\{-\frac12 \frac{t}{p(1-t)} \big((\vartheta^\top X_i)^2 + (\bar\vartheta^\top X_i)^2\big)\right\}\right] \\
&=& \frac{1}{2^{2p}}\sum_{\vartheta, \bar\vartheta \in \{-1,1\}^p}\left(1- \frac{t^2}{p^2} (\vartheta^\top \bar\vartheta)^2\right)^{-n/2} \\
&=& \E\left[\left(1- \frac{t^2}{p^2} W^2\right)^{-n/2}\right]\ ,
\eeqn
where $W$ is distributed like a sum of $p$ independent Rademacher variables.
Since we assume that $p\gg n$, there exists a sequence $u_n$ going to infinity such that $n u^2_n=o(p)$. We decompose the expectation into a sum of three terms
\beqn 
\E\left[\left(1- \frac{t^2}{p^2} W^2\right)^{-n/2}\right]  
&\leq& \left(1-  \frac{u_n^2}{p}\right)^{-n/2} + 2\sum_{k=\lfloor \sqrt{p}u_n\rfloor}^{\lfloor p/2\rfloor }\frac{\P[W= k]}{\left(1- k^2/p^2 \right)^{n/2}}+ 2\frac{\P[W\geq p/2]}{ (1-t^2)^{n/2}} \\
&=:& A_1+A_2+A_3\ .
\eeqn
Since $nu_n^2/p=o(1)$, $A_1=1+o(1)$. By Hoeffding inequality, $\P[W\geq k]\leq e^{-2k^2/p}$. Thus, 
\beqn
A_2&\leq& 2 \sum_{k=\lfloor \sqrt{p}u_n\rfloor}^{\lfloor p/2\rfloor }\exp\left[-2\frac{k^2}{p}- \frac{n}{2}\log\left(1-\frac{k^2}{p^2}\right)\right]\\
&\leq &  2 \sum_{k=\lfloor \sqrt{p}u_n\rfloor}^{\lfloor p/2\rfloor }\exp\left[-\frac{k^2}{p}\left(2- \frac{n}{2p(1-k^2/p^2)}\right)\right]\\
&\leq & 2 \sum_{k=\lfloor \sqrt{p}u_n\rfloor}^{\lfloor p/2\rfloor }\exp\left[-\frac{k^2}{p}\left(2- \frac{4n}{3p}\right)\right]=o(1)\ ,
\eeqn
where in the second line we used $-\log(1-x) \le x/(1-x)$ for all $x \in (0,1)$, and in the third line we used $k\leq p/2$, $n=o(p)$ and $k^2/p\geq u_n^2 \to \infty$.   
Finally, applying Hoeffding's inequality once again, we get
\beqn
A_3 \leq 2 \exp\left[ -\frac{p}{2}- \frac{n}{2}\log(1-t^2)\right]\ , 
\eeqn
which goes to zero since we assume that $(1-t)^{-1}\leq e^{p/(2n)}$.

\section{Proofs: upper bounds} \label{sec:upper-proofs}

As in the previous section, we use the notation $\zeta = \frac{s}n \log(e p/s)$ introduced in \eqref{zeta}.
Define the standardized observations:
\beq \label{standard}
X_{\ddag i} = \bSigma^{-1/2} X_i = \eta_i \mu_{\ddag 0} + (1-\eta_i) \mu_{\ddag 1} + Z_i\ ,
\eeq
based on \eqref{model}, where $\mu_{\ddag k} = \bSigma^{-1/2} \mu_k, k = 0,1$.  
Define the corresponding sample mean and sample covariance matrix:
\[\hat\bSigma_\ddag := \frac1n \sum_{i=1}^n (X_{\ddag i} -\bar X_\ddag) (X_{\ddag i} -\bar X_\ddag)^\top, 
\quad \bar X_\ddag := \frac1n \sum_{i=1}^n X_{\ddag i}\ .\]
Note that 
\[X_{\ddag i} - \bar X_\ddag = Z_i - \bar Z - (\eta_i - \bar \eta) \mudif_\ddag\ ,\]
where $\bar \eta := \frac1n \sum_i \eta_i$ and $\mudif_\ddag := \mu_{\ddag 1} - \mu _{\ddag 0} = \bSigma^{-1/2} \mudif$.  

\bigskip
The following concentration bounds will be useful to us.

\begin{lem}\label{lem:birge_chi2}\cite{MR1836557}
 Let $X$ be a non central $\chi^2$ variable with $D$ degrees of freedom and a non centrality parameter $B$, then for all $x>0$, 
\beqn
\P\left[X\leq (D+B)-2 \sqrt{(D+2B)x}\right]&\leq& e^{-x}\\
\P\left[X\geq (D+B)+2 \sqrt{(D+2B)x}+ 2x\right]&\leq& e^{-x}
\eeqn
\end{lem}

\begin{lem} \label{lem:concentration_vp_wishart}\cite{MR1863696}
Let $\bW$ be a standard Wishart matrix of parameters $(n,d)$ with $n>d$. For any number $0<x<1$, 
\begin{eqnarray}
\mathbb{P}\left[\lambda^{\rm max}(\bW) \geq
n\left(1+\sqrt{d/n}+\sqrt{2x/n}\right)^2 \right]& \leq &e^{-x}\ 
,\nonumber	\\
\mathbb{P}\left[\lambda^{\rm min}(\bW) \leq
n\left(1-\sqrt{d/n}-\sqrt{2x/n}\right)_+^2 \right]& \leq
&e^{-x}\ . \nonumber
\end{eqnarray}
\end{lem}

\subsection{Proof of \prpref{upper_test_sparse1}} \label{sec:proof_upper_test_sparse1}

Since $\bSigma$ is known, we may assume that $\bSigma = \bI$ by working with the standardized observations \eqref{standard}.

\subsubsection{Under $H_0$}
Under the null hypothesis, we control $\hat \lambda_{\bSigma}^{\rm max}$  by applying Lemma \ref{lem:concentration_vp_wishart}, to get
\beqn
\P_0\left[\hat \lambda_{\bSigma}^{\rm max}\geq \left(1+ \sqrt{p/n}+ \sqrt{2x/n}\right)^2\right]\leq \exp(-x)\ .
\eeqn
Taking $x=n\wedge p$, leads to
\beq\label{eq:control_maxeigen_H0}
\hat \lambda_{\bSigma}^{\rm max}\leq 1+ \frac{p}{n}+12\sqrt{\frac p n}\ ,
\eeq
with probability larger than $1-e^{-p\wedge n}$. 

\subsubsection{Under $H_1$}
In this section, $C$ refers to the constant in Condition \eqref{upper_test_sparse1}.
We now turn to the alternative hypothesis.
Let $\bX$ denote the data matrix, meaning the $n \times p$ matrix with rows the $X_i$'s.  Define $\bZ$ similarly.

\bigskip
\noindent {\bf CASE 1: $p>n$.} 
We have
\[
\hat \lambda_{\bSigma}^{\rm max} \geq \frac{\|\bW v_1\|^2}{n}\ ,
\] 
where $\bW=(\bI-\tfrac1n \bJ)\bX$, with $\bJ$ being the matrix with all 1's, and $v_1 :=\kappa \omega+\sqrt{1-\kappa^2}\, t_1$, with 
\[\omega:=\frac{\mudif}{\|\mudif\|}, \quad t_1:=\argmax_{\|t\|=1,\, \<t,\omega\>=0 }\<{\bf W}t,{\bf W}\omega\>, \quad \kappa :=\frac{\<{\bf W}\omega,{\bf W} t_1\>}{\|{\bf W}t_1\|^2}\wedge \frac{1}{4}\ .\]
Note that here the simple lower bound $\hat \lambda_{\bSigma}^{\rm max}\geq \|\bW \omega\|^2/n$ does not yield the right performances. 
We prove below that
\begin{lem}\label{lem:upper_max_eig}
We have
\beq \label{eq:lower2_lambda1}
\hat \lambda_{\bSigma}^{\rm max} \geq \frac{\|{\bf W}t_1\|^2}{n} + \frac{1}{5}\left(\frac{\<{\bf W}\omega,{\bf W} t_1\>}{n}\wedge \frac{\<{\bf W}\omega,{\bf W} t_1\>^2}{n\|{\bf W}t_1\|^2}\right)\ . 
\eeq
\end{lem}

In view of \eqref{eq:lower2_lambda1}, we need to control $\|{\bf W}t_1\|$ and $\<{\bf W}\omega,{\bf W} t_1\>$.
By definition of ${\bf W}$,  we have 
\[\langle {\bf W}\omega,{\bf W}t_1 \rangle = \max_{t:\, \|t\|=1,\, \<t,\omega\>=0}\langle (\bI-\tfrac{1}{n}\bJ){\bf X}\omega,{\bf X}t \rangle\ .\]
By Cochran's theorem, $({\bf X}t : \<t,\omega\>=0)$ is independent of ${\bf X}\omega$, so that upon a change of basis we have
\[\langle {\bf W}\omega,{\bf W}t_1 \rangle =\max_{t\in \mathbb{R}^{p-1},\, \|t\|=1 } \<(\bI-\tfrac{1}{n}\bJ){\bf X}\omega, {\bf U}t\> = \|{\bf U}^\top{\bf W}\omega\| \ ,\]
where the entries of $n\times (p-1)$ matrix ${\bf U}$ are iid standard normal and independent of ${\bf X}\omega$.  Thus, conditionally to ${\bf W}\omega$, 
\beq \label{inner0}
T^2:= \frac{\<{\bf W}t_1,{\bf W}\omega\>^2}{\|{\bf W}\omega\|^2}= \frac{\|{\bf U}^\top{\bf W}\omega\|^2}{\|{\bf W}\omega\|^2}\sim \chi^2_{p-1}\ .
\eeq

We have
\[
\bW \omega = (\bI - \tfrac1n \bJ) \big({\bf 1} \mu_0^\top + ({\bf 1} - \eta) \mudif^\top + \bZ\big) \frac{\mudif}{\|\mudif\|} = (\bI - \tfrac1n \bJ) \bZ \omega - (\bI - \tfrac1n \bJ) \eta \|\mudif\| \ ,
\]
so that, conditionally on $\eta = (\eta_1, \dots, \eta_n)$, $\|{\bf W}\omega\|^2$ follows a non-central $\chi^2$ distribution with $n-1$ degrees of freedom and non-centrality parameter 
\beq \label{noncentral-param}
B := \sum_{i=1}^n (\eta_i-\bar{\eta})^{2} \|\mudif \|^2 = n \bar\eta(1-\bar \eta)\|\mudif\|^2\ ,
\eeq 
where $\bar\eta := \frac1n \sum_i \eta_i$. Applying Lemma \ref{lem:birge_chi2}, we get
\beq \label{W-omega}
\P\left[\|{\bf W}\omega\|^2\leq n -1 + B - 3\sqrt{(n-1+B) x} \ \big| \ \eta\right]\leq e^{- x}\ .
\eeq
Assume without loss of generality that $\nu \le 1/2$, and define the event 
\beq\label{eq:defi_A}
\cA:=\left\{ |\bar\eta-\nu|\leq \tfrac12 \nu (1-\nu)\right\}.
\eeq
Note that, conditionally on $\cA$, $\bar\eta(1-\bar\eta) \ge \frac12 \nu (1-\nu)$.
Since $n \bar\eta \sim \Bin(n, \nu)$, by Bernstein's inequality,
\beq \label{inner1}
\P\left[\cA^c\right]\leq 2\exp\left[-\tfrac{n}{10} \nu (1-\nu)\right]\ .
\eeq
Since we assume that \eqref{upper_test_sparse1} holds, conditionally on $\cA$ we have, for $n$ large enough, $B \geq b := \frac{C}2 \sqrt{n p}$, and from this we derive
\beq \label{inner2}
\P\left[\|{\bf W}\omega\|^2 > \tfrac{C}4 \sqrt{n p} \ \big| \ \cA\right]\geq 1 - e^{- (n-1+b)/36}\ ,
\eeq
by choosing $x = (n-1+b)/36$.
Based on \eqref{inner0}, \eqref{inner1} and \eqref{inner2}, and \lemref{birge_chi2}, we conclude that, with probability tending to one,
\[
\langle {\bf W}\omega,{\bf W}t_1 \rangle^2 \ge \tfrac{C}5 p \sqrt{n p}\ .
\]

Let us turn to $\|{\bf W}t_1\|^2$. We have the decomposition
\[\|{\bf W}t_1\|^2 =T^2+ \max_{u: \,\|u\|=1,\, \<u,{\bf W}\omega\>=0}\<{\bf W}t_1,u\>^2\ .\] 
Again by Cochran's theorem, for any $t$ such that $<t,\omega>=0$, 
\beq \label{cochran2}
(\<(\bI-\tfrac1n \bJ){\bf X}t,u\>: \<t,\omega\> = \<u,{\bf W}\omega\>=0) \quad \text{and} \quad (\<(\bI-\tfrac{1}{n} \bJ){\bf X}t,\bW\omega\>: \<t,\omega\>=0)
\eeq
are independent conditionally on ${\bf W}\omega$.  Since $t_1$ is a function of the right-hand side of \eqref{cochran2}, the distribution of the above maximum is the same as if $t_1$ were fixed, say equal to $t$.  (Note that $t$ is necessarily a unit vector satisfying $\langle t,\omega\rangle =0$.)  Then that maximum is equal to 
\[ \sup_{u: \, \|u\|=1,\,  \<u,{\bf W}\omega\>=0}\<\bW t, u\>^2 
= \sup_{u: \, \|u\|=1,\,  \<u,{\bf W}\omega\>=0}\< \bZ t, (\bI-\tfrac1n \bJ)u\>^2 = \|P \bZ t\|^2 \ ,
\]
where $P$ is the orthogonal projection onto ${\rm span}\{{\bf 1}, \bW \omega\}^\perp$, and the first equality comes from 
\[
\bW t = (\bI - \tfrac1n \bJ) \bX t = (\bI - \tfrac1n \bJ) \big({\bf 1} \mu_0^\top + ({\bf 1} - \eta) \mudif^\top + \bZ\big) t\ ,
\]
and the fact that $\<t, \mudif\> = 0$.
Since the $Z_i$'s are standard normal and $t$ is normed, $\|P \bZ t\|^2$ has the $\chi^2$ distribution with $n-2$ degrees of freedom.  
Then using \eqref{inner0} and the deviations of the chi-squared distribution (e.g., \lemref{birge_chi2}), we derive that 
\[
\|{\bf W}t_1\|^2 = n + p + O_P(\sqrt{p})\ .
\]

Plugging these bounds into \eqref{eq:lower2_lambda1}, we get
\beqn 
\hat \lambda_{\bSigma}^{\rm max}&\geq& 1+ \frac{p}{n} + O_P(\sqrt{p}/n) + \frac1{5n} \left[ \left(\frac{C}5 p \sqrt{n p}\right)^{1/2} \wedge \frac{\frac{C}5 p \sqrt{n p}}{n+p+O_P(\sqrt{p})} \right]\\ 
&\geq & 1+ \frac{p}{n}+ \frac{1}{5}\left[\left(\frac{C}{5}\right)^{1/2}\wedge \frac{C}{10}\right]\sqrt{\frac p n}(1+o_P(1))\ ,
\eeqn
since $p\geq n$. If condition \eqref{upper_test_sparse1} is satisfied for $C$ large enough, then this last quantity is larger than the RHS in \eqref{eq:control_maxeigen_H0} with probability going to one.  In conclusion,  the risk of the test  for $p> n$ is smaller than $o(1)+\P(\cA^c)$ with $\P(\cA^c)=o(1)$ under condition  \eqref{upper_test_sparse1}.

\bigskip
\noindent {\bf CASE 2: $p\le n$.} 
Here we simply use the lower bound 
\beqn
\hat \lambda_{\bSigma}^{\rm max}\geq \frac{\|\bW \omega\|^2}{n}\ ,
\eeqn
and we use \eqref{W-omega}, and the fact that, conditionally on $\cA$, we have $B \geq b$, to derive
\[
\P\left[\|{\bf W}\omega\|^2 > n-1+b - 3 \sqrt{(n-1+b) x} \ \big| \ \cA\right]\geq 1 - e^{- x}\ ,
\]
when $0 < x \le \frac19 (n-1+b)$.  Choosing $x = p/9 \to \infty$, and since $\P(\cA^c) = o(1)$, with probability tending to 1, we have
\[
\hat \lambda_{\bSigma}^{\rm max} \ge 1 -\frac1n + \tfrac{C}2 \sqrt{\frac{p}n} - 3 \sqrt{(1 + \tfrac{C}2) \frac{p}n} = 1 -\frac1n + \left(\tfrac{C}2 - 3 \sqrt{1 + \tfrac{C}2}\right) \sqrt{\frac{p}n}\ .
\] 
When the constant $C$ in  Condition \eqref{upper_test_sparse1} is chosen large enough, the RHS here is larger than the RHS in \eqref{eq:control_maxeigen_H0}, which implies that the test is asymptotically powerful.

\subsubsection{Proof of Lemma \ref{lem:upper_max_eig}}
Suppose that $\kappa<1/4$. We have
\beqn
\|{\bf W} v_1\|^2
&=& \|{\bf W}t_1\|^2 +\kappa^2\|{\bf W} \omega\|^2+  (-1+2\sqrt{1-\kappa^2})  \frac{\<{\bf W}\omega,{\bf W} t_1\>^2}{\|{\bf W}t_1\|^2} \\
&\ge& \|{\bf W}t_1\|^2 + b \frac{\<{\bf W}\omega,{\bf W} t_1\>^2}{\|{\bf W}t_1\|^2}\ ,
\eeqn
where $b := (-1+2\sqrt{1-(1/4)^2}) > 4/5$. And when $\kappa=1/4$, we have $\|{\bf W}t_1\|^2\leq 4\<{\bf W}\omega,{\bf W} t_1\>$, so that 
\[\|{\bf W} v_1\|^2\geq \|{\bf W}t_1\|^2 + \left(2\kappa\sqrt{1-\kappa^2}- 4\kappa^2\right)\<{\bf W}\omega,{\bf W} t_1\>\ ,\]

From this we conclude.

\subsection{Proof of \prpref{upper_test_sparse2}} \label{sec:proof_upper_test_sparse2}

We work again the standardized data \eqref{standard}.
Note that 
\[
\hat\lambda_{s, \bSigma}^{\rm \max} :=   \max_{\|\bSigma^{1/2}u\|_0 \le s} \ \frac{u^\top \hat\bSigma_\ddag u}{u^\top  u}\ . \]

\subsubsection{Under $H_0$}
Under the null hypothesis, $n\hat\bSigma_\ddag$ follows a Wishart distribution with parameters $(n-1,p)$. Thus, $n \hat\lambda_{s, \bSigma}^{\rm \max}$ is the supremum of $\binom{s}{p}$ largest eigenvalues of Wishart matrices with parameters $(n-1,s)$.  Although they are not independent, we may apply the union bound, and then use the deviation bound in Lemma \ref{lem:concentration_vp_wishart} to get
\beqn
\P_0\left[\hat\lambda_{s, \bSigma}^{\rm \max}\leq \left(1+ \sqrt{\frac{s}{n}}+ \sqrt{\frac{4s}{n}\log\left(\frac{ep}{s}\right)}\right)^2\right]\leq \left(\frac{ep}{s}\right)^{-s}\ . 
\eeqn
Hence, with probability going to one,
\beq\label{eq:upper_lambda_s}
\hat\lambda_{s, \bSigma}^{\rm \max}\leq 1+ 9 \zeta + 6  \sqrt{\zeta}\ .
\eeq

\subsubsection{Under $H_1$}
In this section, $C$ refers to the universal constant that appears in \eqref{upper_test_sparse2}.
Under the alternative, we use $u=\bSigma^{-1/2}\mudif$ to get the lower bound
\[
\hat\lambda_{s, \bSigma}^{\rm \max}\geq  \sum_{i=1}^{n} \frac{ \left[(X_{\ddag i}- \bar{X_\ddag})^{\top } \bSigma^{-1/2}\mudif\right]^{2}}{n \|\bSigma^{-1/2}\mudif\|^{2}}:= \frac{W}{n}\ ,
\] 
where
\beq \label{W}
W= \frac{\sum_{i=1}^{n}\left[(Z_i-\bar{Z})^{\top} \bSigma^{-1/2}\mudif -(\eta_i-\bar{\eta})\mudif^{\top} \bSigma^{-1}\mudif \right]^2}{\|\bSigma^{-1/2}\mudif\|^{2}}\ ,
\eeq
and $\bar \eta = \frac1n \sum_i \eta_i$.
Since $\eta_1, \dots, \eta_n, Z_1, \dots, Z_n$ are independent, given the $\eta_i$'s, $W$ follows a $\chi^2$ distribution with $n-1$ degrees of freedom with non-centrality parameter 
$B := n\bar{\eta}(1-\bar{\eta}) \|\bSigma^{-1/2}\mudif \|^2$ as in \eqref{noncentral-param}, and by Lemma \ref{lem:birge_chi2},
\[
\P\big[W \geq n-1 + B - 3 \sqrt{(n + B) x}\ \big| \ \eta\big] \ge 1 - e^{-x}\ ,
\]
for any $x > 0$.  
As in the proof of \prpref{upper_test_sparse1}, when \eqref{upper_test_sparse2} holds we have, for $n$ large enough, that $B \ge b := \frac{C}2 n (\zeta \vee \sqrt{\zeta})$ under the event $\cA$ defined in \eqref{eq:defi_A}.
This gives
\[
\P\big[W \geq n-1 + b - 3 \sqrt{(n -1 + b) x}\ \big| \ \cA\big] \ge 1 - e^{-x}\ ,
\]
as long as $x \le \frac19 (n-1+b)$.  We choose $x = \frac{b^2}{6^2(n+b)} \wedge \frac19 (n-1+b)$, so that $W \ge n-1+b/2$ when the event above holds.  Note that $x \to \infty$, and since $\P(\cA) \to 1$ by \eqref{inner1}, with probability tending to one, we have 
\beq \label{eq:lower_lambda_s}
\hat\lambda_{s, \bSigma}^{\rm \max}\geq 1- \frac{1}{n} + \frac{C}4 (\zeta \vee \sqrt{\zeta})\ .
\eeq
Comparing this lower bound (under the alternative) with upper bound \eqref{eq:upper_lambda_s} (under the null) concludes the proof.

\subsubsection{Variable selection} \label{sec:proof-selection-known}
 
We continue with the same notation and work with the standardized variables, but now assume that \eqref{upper_test_sparse1-select} holds.  
For any $u \in \bbR^p$,
\[
u^\top \hat\bSigma_\ddag u = \bar\eta (1-\bar\eta) (u^\top \mudif_\ddag)^2 - 2  (u^\top \mudif_\ddag) (u^\top Y) + u^\top \hat\bSigma_Z u\ ,
\]
where $Y := \frac1n \sum_{i=1}^n (\eta_i -\bar\eta) (Z_i - \bar Z)$ and $\hat\bSigma_Z$ is the sample covariance of $Z_1, \dots, Z_n$.  
Fix $\delta \in (0,1)$ and define 
\[
\cU_\delta = \big\{u \in \bbR^p : \|u\| = 1, \|\bSigma^{1/2}u\|_0 \le s \text{ and } |u^\top \mudif_\ddag| \le (1-\delta) \|\mudif_\ddag\|\big\}\ .
\]
In particular, any $u \in \cU_\delta$ makes an angle of at least $\acos(1-\delta)$ with $\mudif_\ddag$.  Exactly as in \eqref{eq:upper_lambda_s}, we have 
\[\max_{u \in \cU_\delta} u^\top \hat\bSigma_Z u \le 1 + 9 \zeta + 6 \sqrt{\zeta}\ ,\]
with probability going to one.
And given $\eta_1, \dots, \eta_n$, $Y \sim \cN(0, \frac{1}{n} \bar\eta (1-\bar\eta) \bI)$, with $\bar\eta (1-\bar\eta) \le 1/4$, so that 
\[\max_{u \in \cU_\delta} |u^\top Y| \ \stackrel{\rm sto}{\le} \ \tfrac1{\sqrt{n}} \max_{u \in \cU_0} |u^\top Z_0| = \tfrac1{\sqrt{n}} \Big(\max_{|S| = s} \max_{u \in \cU_S} |u^\top Z_0|\Big)\ ,\]
where $Z_0 \sim \cN(0, \bI)$, and for a subset $S \subset [p]$, we define
$
\cU_S = \big\{u \in \bbR^p : \|u\| = 1, {\rm supp}(\bSigma^{1/2} u) = S\big\}.
$ 
For each subset $S$ of size $s$, $\max_{u \in \cU_S} (u^\top Z_0)^2$ has the chi-squared distribution with $s$ degrees of freedom.  Since there are $\binom{p}{s}$ such subsets, a union bound and an application of Lemma~\ref{lem:birge_chi2} yields 
\[
\P\left[\max_{u\in\cU_0} (u^{\top} Z_0)^2 \ge s + 2 \sqrt{s x} + 2 x\right] \leq e^{s \log(e p/s) -x} \ ,
\]
for all $x > 0$, and choosing $x = 2 s \log(ep/s)$, we get
\beq \label{Z0-zeta}
\P\left[\max_{u\in\cU_0} |u^{\top} Z_0| \ge 3 \sqrt{s \log(e p/s)}\right] \le e^{- s \log(e p/s)} = o(1) \ .
\eeq

We also have $\bar\eta (1-\bar\eta) \le (1+\delta) \nu(1-\nu)$ with probability tending to one, by Chebyshev's inequality.  Hence, with probability tending to one,
\beqn
\max_{u \in \cU_\delta} u^\top \hat\bSigma_\ddag u 
&\le& (1+\delta)\nu(1-\nu) (1-\delta)^2 \|\mudif_\ddag\|^2 + 2 (1-\delta) \|\mudif_\ddag\| \, 3 \sqrt{\zeta} + 1 + 9 \zeta + 6 \sqrt{\zeta} \\
&\le& 1 + (1-\delta) \nu(1-\nu) \|\mudif_\ddag\|^2\ ,
\eeqn
eventually, using \eqref{upper_test_sparse1-select} in the second line.
(Note that $\|\mudif_\ddag\|^2 = \mudif^\top \bSigma^{-1} \mudif$.)

One the other hand, since $\bar\eta (1-\bar\eta) \ge (1-\delta/3) \nu(1-\nu)$ with probability tending to one, when $u_\star = \mudif_\ddag/\|\mudif_\ddag\|$ we have 
\begin{eqnarray}
u_\star^\top \hat\bSigma_\ddag u_\star 
&\ge& (1-\delta/3) \nu(1-\nu) \|\mudif_\ddag\|^2 - 6 (1-\delta) \|\mudif_\ddag\| \sqrt{\zeta} + 1 - O_P(1/\sqrt{n}) \notag \\
&\ge& 1 + (1-\delta/2) \nu(1-\nu) \|\mudif_\ddag\|^2\ , \label{u-star}
\end{eqnarray}
eventually, using \eqref{upper_test_sparse1-select} again, and using the fact that $u_\star^\top \hat\bSigma_Z u_\star \sim \frac1n \chi_{n-1}^2$.

Using these two bounds, with probability tending to one,
\[
u_\star^\top \hat\bSigma_\ddag u_\star - \max_{u \in \cU_\delta} u^\top \hat\bSigma_\ddag u  > 0\ , 
\]
eventually.  Let $A_\delta$ be that event.  We just showed that, for any fixed $\delta > 0$, $\P(A_\delta) \to 1$.

Let $\hat u$ be a maximizer of \eqref{eq:St_known_sparse} and define $\hat v = \bSigma^{1/2} \hat u$.  
Define $J = {\rm supp}(\mudif)$ and $\hat J = {\rm supp}(\hat v)$,  let $\tau = \max_{j \in J} |\mudif_j| / \min_{j \in J} |\mudif_j|$ denote the effective dynamic range of $\mudif$, and let $\Upsilon = \lambda^{\rm max}_{2s}(\bSigma)/ \lambda^{\rm min}_{2s}(\bSigma)$ denote the $2s$-sparse Riesz constant of $\bSigma$. 
Under $A_\delta$, $\hat u \notin \cU_\delta$, which by definition implies that $|\hat u^\top u_\star| \ge (1-\delta)$, which is equivalent to
\[\|\hat u - u_\star\|^2 \leq 2-2 (1-\delta) = 2\delta\ , \] 
since $\hat{u}^*$ and $u_\star$ are both unit vectors. 
Since $\hat{v}$ and $\mudif=\bSigma^{1/2}\mudif_\ddag$ are $s$ sparse, 
\[\|\hat v - \mudif/\|\mudif_\ddag\|\|^2= \|\bSigma^{1/2}(\hat u - u_\star)\|^2\leq \lambda^{\rm max}_{2s}(\bSigma)2\delta\ .\]
On the other hand, using again the operator norm, we get
\beqn
\|\hat v - \mudif/\|\mudif_\ddag\|\|^2\geq  \frac{\sum_{j \in J \setminus \hat J}\mudif_j^2}{\|\mudif_\ddag\|^2}\geq \frac{\sum_{j \in J \setminus \hat J}\mudif_j^2}{\lambda^{\rm max}_{2s}(\bSigma^{-1/2})^2 \sum_{j \in J}\mudif_j^2}\geq \lambda^{\rm min}_{2s}(\bSigma)\frac{| J \setminus \hat J|}{|J|} \frac1{\tau^2}\ .
\eeqn 
Since $J$ and $\hat J$ are of same size, we have $|\hat J \triangle J| = 2 |J \setminus \hat J|$, and we conclude that, under $A_\delta$, $|\hat J \triangle J|/|J| \le \delta \tau^2 \Upsilon$.
Since this is true for any fixed $\delta > 0$, and since $\tau$ and $\Upsilon$ are  bounded, \eqref{var-select} holds and the proof is complete.

\subsection{Proof of \prpref{malk2}}
Define
\[
V = \min_{\|u\|_0\leq s} \frac{n \sum_i \big[u^\top (X_i - \bar X)\big]^4}{\left(\sum_i \big[u^\top (X_i -\bar X)\big]^2\right)^2}\ .
\]
We work with the standardized data \eqref{standard}, and without loss of generality, we assume that $\E(X) = 0$ always.
By a simple change of variables, one may write $V$ as 
\[
V = \min_{v \in \cV} \frac{n Q_4(v)}{Q^2_2(v)} \ ,
\]
where
\beq \label{Q2Q4}
Q_2(v) := \sum_i \big[v^\top (X_{\ddag i} - \bar X_\ddag)\big]^2\ , \quad
Q_4(v) := \sum_i \big[v^\top (X_{\ddag i} - \bar X_\ddag)\big]^4\ ,
\eeq
and
\beq \label{cV}
\cV := \big\{v : \|v\| = 1, \, \|\bSigma^{-1/2}v\|_0\leq s\big\} \ .
\eeq

\subsubsection{Under $H_0$}
Suppose we are under the null, so that $X_{\ddag i} = Z_i$ are iid standard normal.  
We lower bound $V$ by
\[
V \ge \frac{n \min_{v\in\cV} Q_4(v)}{\max_{v\in\cV} Q^2_2(v)}\ ,
\]
and control the numerator and denominator separately.

We first build a net for $\cV$. 
For $\eps \in (0,1)$, let $w_1, \dots, w_{N_\eps}$ be an $\eps$-net (with respect to the Euclidean metric) of $\cV$.  Since $\cV$ is the union of unit spheres of $\binom{p}{s}$  subspaces of dimension $s$, we may take 
\beq \label{V_net}
N_\eps \le \binom{p}s (1+2/\epsilon)^s
\eeq 
by \citep[Lem.~1.2]{vershynin}.  

We first bound the denominator from above.  Since $\cV$ is the union of $\binom{p}{s}$ unit balls of subspaces of dimension $s$, $\max_{v\in \cV} Q_2(v)$ is distributed like the maximum of $\binom{p}{s}$ (possibly dependent) largest eigenvalues of Wishart matrices with parameters $(n-1,s)$.  Applying the union bound and then Lemma~\ref{lem:concentration_vp_wishart}, we derive that
\[
\P\left[\max_{v\in \cV} Q_2(v) \ge (n-1) \big(1 + \sqrt{s/(n-1)} + \sqrt{2x/(n-1)}\big)^2\right]\leq \binom{p}s e^{-x} \le e^{s \log (ep/s) - x} \ .
\]
for all $x \ge 0$.
For $x \ge s \ge 1$, we have
\[
(n-1) (1 + \sqrt{s/(n-1)} + \sqrt{2x/(n-1)})^2 \le n + 2\sqrt{ns} + 2 \sqrt{2nx} + 6 x\ ,
\]
so that changing $x$ into $s \log (ep/s) + x$, and using the fact that $\sqrt{a+b} \le \sqrt{a} + \sqrt{b}$ for all $a,b > 0$, and also that $\zeta \le 1$, we get 
\beq\label{eq:lower_bound_Q2}
\P\left[\min_{v\in \cV} Q_2(v)\geq n + 11 n \sqrt{\zeta} + 2 \sqrt{2nx} + 6 x\right]\leq e^{-x}\ ,
\eeq 
for any $x>0$. 

We now bound the numerator from below, still under the null. We have
\beq \label{eq:Q4_dec}
\inf_{v\in \cV}Q_4(v) \ge \inf_{v\in\cV}Q_4^\circ(v) - 3 \sup_{v\in \cV}Q_3^\circ(v)\max_{v\in\cV} v^{\top} \bar Z- 2n \max_{v\in\cV} (v^{\top} \bar Z)^4\ ,
\eeq
where
\[Q_3^\circ(v) := \sum_{i=1}^n (v^\top Z_{ i})^3 \ ,\quad \quad Q_4^\circ(v) := \sum_{i=1}^n (v^\top Z_{ i})^4\ .
\]
We have that $n\max_{v\in\cV} (v^{\top} \bar Z)^2$ is distributed as the maximum of $\binom{p}{s}$ (possibly dependent) $\chi^2_s$ random variables.  Applying the union bound and then Lemma~\ref{lem:birge_chi2}, as above we derive
\[
\P\left[n\max_{v\in\cV} (v^{\top} \bar Z)^2 \ge s + 2 \sqrt{s x} + 2 x\right] \leq e^{s \log(e p/s) -x} \ ,
\]
for all $x > 0$, and choosing $x = 2 s \log(ep/s)$, we get
\beq \label{barZ}
\P\left[\max_{v\in\cV}|v^{\top} \bar Z| \ge 3 \sqrt{\zeta}\right] = o(1) \ .
\eeq
The random variable $Q_3^\circ(v)$ is controlled in \eqref{eq:upper_Q3} (proof of \prpref{malk4}).

To control $Q_4^\circ(v)$, we use a chaining argument together with some deviation inequalities.
\begin{lem}\label{lem:deviation_polynom4}
For any $x>0$, and any unit vectors $v$, $w$, we have
\beq\label{eq:deviation_moment_4}
\P\big[Q^\circ_4(v)\leq 3n - C\sqrt{nx}\big] \leq x^{-1}\ ,
\eeq
and
\beq \label{eq:deviation_difference_4}
\P\big[|Q^\circ_4(v)-Q^\circ_4(w)|\geq\|v - w\|\big(  C_1\sqrt{n(x\vee 1)}+ C_2 x^{2}\big) \big]\leq 8e^{-x}\ ,
\eeq
where $C_1, C_2, C_3$ are positive universal constants.
\end{lem}
The proof is postponed to Section~\ref{sec:proofs_deviation}. 

Fix some $x>0$. For any integer $j$, set $\epsilon_j=2^{-j+1}$, and let $N_j$ denote a minimal $\eps_j$-covering number of $\cV$.  Note that $N_j \le \binom{p}s (1 + 2^{j+1})^s$ by \eqref{V_net}. Let $\cV_j \subset \cV$ be an $\eps_j$-net for $\cV$ of cardinality $N_j$. Let $\Pi_j : \cV \mapsto \cV_j$ be such that $\|\Pi_j v - v\| \le \eps_j$ for all $v \in \cV$.  Since $v \mapsto Q_4^\circ(v)$ is almost surely continuous, we have the following decomposition
\[Q_4^\circ(v)= Q_4^\circ(\Pi_0 v)+ \sum_{j=1}^{\infty } \big[ Q_4^\circ(\Pi_{j+1}v)- Q_4^\circ(\Pi_{j}v) \big]\ , \]
from which we deduce 
\[\inf_{v\in\cV}|Q_4^\circ(v)| \ge \inf_{v\in \cV} |Q_4^\circ(\Pi_0 v)|- \sum_{j=1}^{\infty } \sup_{v\in \cV}|Q_4^\circ(\Pi_{j+1}v)- Q_4^\circ(\Pi_{j}v)|\ ,
\]
We simultaneously control the deviations of all these suprema.

Using  \eqref{eq:deviation_moment_4} together with the fact that $N_0\leq 2$ (the diameter of $\cV$ is $2$ and only opposite vectors lie at a distance 2),
\[\inf_{v\in \cV} |Q_4^\circ(\Pi_0 v)|\geq 3n - C\sqrt{nx} \]
with probability larger than $1-x^{-1}$.

For any integer $j \ge 0$, the range of $v \mapsto (\Pi_j v, \Pi_{j+1}v)$ is a set with cardinality at most $N_j N_{j+1} \le N_{j+1}^2$.
Moreover, by the triangle inequality, $\|\Pi_{j}v-\Pi_{j+1}v\|\leq \|\Pi_{j}v-v\| + \|\Pi_{j+1}v -v\| \le 3\epsilon_{j+1}$, for any $v\in \cV$. Hence, by \eqref{eq:deviation_difference_3}, we get 
\beqn
\frac1{3\epsilon_{j+1}} \sup_{v\in \cV}|Q_4^\circ(\Pi_{j+1}v)- Q_4^\circ(\Pi_{j}v)|
&\leq&  C_1\left[\sqrt{ns\log\left(\frac{ep}{s}\right)}+ \left[s\log\left(\frac{ep}{s}\right)\right]^{2}\right]\\ 
&+& C_2\left[ \sqrt{ns\log\left(1+ \frac{2}{\epsilon_{j+1}}\right)}+ \left(s\log\left(1+ \frac{2}{\epsilon_{j+1}}\right)\right)^{2}\right]\\
&+ & C_3\left[\sqrt{nx}+x^{2}\right]+ C_4\left[\sqrt{nj}+ j^{2}\right]
\eeqn
with probability larger than $1- 8e^{-j}e^{-x}$. 
Gathering all these deviation inequalities leads to 
\beq\label{eq:Q4ddag}
\inf_{v\in\cV}|Q_4^\circ(v)|\geq 3n - C_1\left[\sqrt{ns\log\left(\frac{ep}{s}\right)}+ \left[s\log\left(\frac{ep}{s}\right)\right]^{2}\right]-  C_2\left[\sqrt{nx}+x^{2}\right]\ ,
\eeq 
with probability larger than $1-Ce^{-x}-x^{-1}$.

Gathering the decomposition \eqref{eq:Q4_dec} with the deviation bounds \eqref{eq:Q4ddag}, \eqref{eq:upper_Q3}, \eqref{eq:lower_bound_Q2} with the choice $x = s \log(e p/s)$, and \eqref{barZ}, we arrive at the following: with probability tending to one under the null, 
\beq\label{eq:control_V1_H0}
V \geq \frac{3- C\sqrt{\zeta} - C n\zeta^2}{1+11\sqrt{\zeta}}\ . 
\eeq

\subsubsection{Under $H_1$}
Under the alternative, let $v = \bSigma^{1/2}\mudif/\|\bSigma^{1/2} \mudif\|$ and $t = \|\mudif\|^2/\|\bSigma^{1/2} \mudif\|$.  We have 
\[V \leq \frac{n \sum_i \big|v^\top (X_{\ddag i}- \bar X_\ddag)\big|^4 }{\left( \sum_i \big[v^\top (X_{\ddag i}- \bar X_\ddag)\big]^2\right)^{2}} 
\]
Note that 
\[v^\top(X_{\ddag i}- \E{X_\ddag}) = w_i t + z_i, \quad w_i := \nu - \eta_i, \quad z_i = v^\top Z_i\ .\]
For any integer $k$, define $S_k:= \sum_i(w_i t + z_i)^k$. Then,
\[V\leq n\frac{S_4-4S_3S_1/n+ 6S_2S_1^2/n^2 -3S_1^4/n^3}{\left(S_2-S_1^2/n\right)^2}\ .\] 
By Chebyshev's inequality, $S_1= \sqrt{n}O_P\left(1+ t\right)$, $S_2= n+ n\nu(1-\nu)t^2 + \sqrt{n}O_P\left(1+ t^2\right)$, $S_3= n\nu(1-\nu)(2\nu -1) t^3+ \sqrt{n}O_P\left[1+t^3\right]$ and $S_4=n[3+6\nu(1-\nu)t^2+\nu(1-\nu)(\nu^3+(1-\nu)^3)t^4]+\sqrt{n}O_P\left[1+t^4\right]$.
\beqn
V 		
&\le& \frac{3+6\nu(1-\nu)t^2+\nu(1-\nu)(\nu^3+(1-\nu)^3)t^4+ n^{-1/2}O_P(1 + t^4)}{1+2\nu(1-\nu)t^2+\nu^2(1-\nu)^2t^4 +n^{-1/2}O_P(1+t^4)} \\
&& = 3 +  \frac{ \nu(1-\nu)[1-6\nu(1-\nu)]t^4 + n^{-1/2}O_P(1+t^4)}{1+2\nu(1-\nu)t^2+\nu^2(1-\nu)^2t^4 +n^{-1/2}O_P(1+t^4)}\ , 
\eeqn
where $(1-6\nu(1-\nu))$ is negative for $|\nu-1/2|< \sqrt{3}/6$.  So for the test based on $V$ to be powerful, it suffices that $t^4 \gg (\zeta^{1/2}\vee   n\zeta^2)$.

\subsection{Proof of \prpref{malk3}} \label{sec:malk3_proof}
Define
\[
V = \max_{\|u\|_0\leq s} \frac{\sum_i \big|u^\top (X_i-\bar{X})|}{\left(n \sum_i \big[u^\top (X_i- \bar{X})\big]^2\right)^{1/2}}\ .
\]
The proof is similar to that of \prpref{malk2}, but the numerator is controlled via Chernoff's bound, which is applicable since it has finite moment generating function.
We work with the standardized data \eqref{standard}, and without loss of generality, we assume that $\E(X) = 0$ always.
By a simple change of variables, one may write $V$ as 
\[
V = \max_{v \in \cV} \frac{Q_1(v)}{\sqrt{n} \, Q_2(v)^{1/2}}\ ,
\]
where $Q_2$ is defined in \eqref{Q2Q4}, $\cV$ in \eqref{cV}, and 
\[Q_1(v) := \sum_i \big|v^\top (X_{\ddag i} - \bar X_\ddag)\big|\ .\]

\subsubsection{Under $H_0$} \label{sec:proof-malk3-h0}
First, assume the null hypothesis is true.
We upper bound $V$ by
\beq \label{eq:V_bound}
 V \le \frac{\max_{v\in\cV} Q_1(v)}{\sqrt{n} \min_{v\in\cV} Q_2(v)^{1/2}}\ .
\eeq
Note that $X_{\ddag i} = Z_i$ under the null.

We first bound the denominator from below using the same approach as in \prpref{malk2}. 
\beq\label{eq:upper_bound_Q2}
\P\left[\min_{v\in \cV} Q_2(v)\leq n- 5n \sqrt{\zeta} - 2 \sqrt{2nx}\right]\leq e^{-x}\ ,
\eeq 
for any $x>0$. 

We now bound the numerator from above, still under the null.  We have
\beq \label{Q1_bound}
Q_1(v) \le Q_1^\circ(v) + n |v^{\top} \bar Z|, \quad Q_1^\circ(v) := \sum_{i=1}^n |v^\top Z_i|\ .
\eeq
We have that $n\max_{v\in\cV}|v^{\top} \bar Z|^2$ is distributed as the maximum of $\binom{p}{s}$ (possibly dependent) $\chi^2_s$-distributed random variables.  Applying the union bound and then Lemma~\ref{lem:birge_chi2}, as above we derive
\[\P\left[n\max_{v\in\cV}|v^{\top} \bar Z|^2 \ge s + 2 \sqrt{s x} + 2 x\right] \leq e^{s \log(e p/s) -x} \ ,\]
for all $x > 0$, and choosing $x = 2s \log(ep/s)$, we get
\beq \label{barZ_bound}
\P\left[n\max_{v\in\cV}|v^{\top} \bar Z| \ge 3 n \sqrt{\zeta}\right] = o(1) \ .
\eeq

The function $g: \bbR^{np} \to \bbR$, $g(z_1, \dots, z_n) = \max_{v\in \cV} \sum_i |v^\top z_i|$, is $\sqrt{n}$-Lipschitz with respect to the Euclidean norm, since
\beqn
g(z_1, \dots, z_n) -g(z'_1, \dots, z'_n)
&\le& \max_{v\in \cV} \sum_i \big(|v^\top z_i| - |v^\top z'_i|\big) 
\le \max_{v\in \cV} \sum_i |v^\top (z_i- z'_i)| \\
&\le& \sum_i \|z_i - z'_i\| \le \sqrt{n} \sqrt{\sum_i \|z_i - z'_i\|^2}\ ,
\eeqn
where we used, in order, the triangle inequality, the Cauchy-Schwarz inequality with the fact that $v\in\cV$ is normed, and again the Cauchy-Schwarz inequality.
Therefore, by the Gaussian isoperimetric inequality \citep[Prop.~2.1]{MR1600888}, 
\beq\label{eq:concentration_A1}
\P_0\left[\max_{v\in \cV}Q^\circ_1(v)\geq \E_{0}\Big[\max_{v\in \cV}Q^\circ_1(v)\Big]+ \sqrt{2nx}\right]\leq e^{-x}\ , 
\eeq
for any $x>0$. 
We now upper bound $\E_{0}[\max_{v\in \cV}Q^\circ_1(v)]$ using a chaining argument. 

\begin{lem}\label{lem:process_subgaussian}
 The process $(Q^\circ_1(v), v\in \cV)$ is subgaussian with respect to the metric $\rho(v,w) :=2\sqrt{n}\|v-w\|$ in the following sense
\[\P\left[|Q^\circ_1(v)-Q^\circ_1(w)|>x\right]\leq 2 \exp\big[-\tfrac{x^2}{8 n \|v-w\|^2}\big]\ , \quad \forall v,w \in \cV,\ x>0\ .\]
\end{lem}
The proof is postponed to Section~\ref{sec:proofs_deviation}.

Below, $C$ denotes a positive universal constant that may change with each appearance.  
Since the process is subgaussian, we can use the following maximal inequality \citep[Cor.~2.2.8]{MR1385671}
\beq \label{chaining}
\E_{0}\left[\max_{v\in \cV} Q^\circ_1(v)\right]\leq \E_0[Q^\circ_1(v)] + C \int_{0}^{\infty}\sqrt{\log D(\epsilon,\rho)} {\rm d}\epsilon\ ,
\eeq
where $D(\epsilon,\rho)$ is the $\epsilon$-packing number of $\cV$ with respect to the semi-metric $\rho$.
The diameter of $\cV$ with respect to $\rho$ is equal to $4\sqrt{n}$. Thus,  we are only interested in $\epsilon$ smaller than $4\sqrt{n}$.
Furthermore, by comparing the packing number with the covering number $N (\epsilon, \rho)$, we obtain 
\[
D(\epsilon,\rho)\leq N(\epsilon/2,\rho)= N(\epsilon/(4\sqrt{n}),\|.\|) \le \binom{p}s (1+8\sqrt{n}/\epsilon)^s\ ,
\]
by \eqref{V_net}.
Hence, the second term on the right-hand side of \eqref{chaining} is bounded by
\[
4 C \sqrt{n} \sqrt{s \log\left(\frac{ep}{s}\right)}+ C\sqrt{s} \int_{0}^{4\sqrt{n}}\sqrt{\log\left[1+ \frac{8\sqrt{n}}{\epsilon}\right]}{\rm d}\epsilon \\
\leq C n \sqrt{\zeta}\ .
\]
Therefore, coming back to \eqref{chaining} and adding in the fact that $\E_0[Q^\circ_1(v)]=n\sqrt{2/\pi}$, we get 
\beq\label{eq:upper_expect}
\E_{0}\left[\max_{v\in \cV} Q^\circ_1(v)\right] \le n\sqrt{\frac{2}{\pi}} + C n \sqrt{\zeta}\ .
\eeq
Then choosing $x = s \log(ep/s)$ in \eqref{eq:concentration_A1} and combining that with \eqref{chaining}, \eqref{Q1_bound} and \eqref{barZ_bound}, we come to
\beq\label{eq:upper_bound_Q1}
\P_0\left[\max_{v\in \cV} Q_1(v)\geq  n\sqrt{\frac{2}{\pi}} + C n \sqrt{\zeta}\right] = o(1)\ .
\eeq

Gathering the three deviation bounds \eqref{eq:V_bound}, \eqref{eq:upper_bound_Q2} with $x = s \log(e p/s)$, and \eqref{eq:upper_bound_Q1}, we arrived at the following: with probability tending to one under the null, 
\beq\label{eq:upper_bound_V1}
V \le \frac{\sqrt{\frac{2}{\pi}} + C\sqrt{\zeta}}{\left[1 - C\sqrt{\zeta}\right]^{1/2}} \le \sqrt{\tfrac2\pi} +  C\sqrt{\zeta}\ .
\eeq

\subsubsection{Under $H_1$} \label{sec:proof-malk3-h1}

Under the alternative, let $v = \bSigma^{1/2}\mudif/\|\bSigma^{1/2} \mudif\|$ and $t = \|\mudif\|^2/\|\bSigma^{1/2} \mudif\|$.  We have 
\beq \label{V-lb}
V \geq \frac{\sum_i \big|v^\top (X_{\ddag i}- \bar X_\ddag)\big| }{\left(n \sum_i \big[v^\top (X_{\ddag i}- \bar X_\ddag)\big]^2\right)^{1/2}} \geq \frac{\sum_i \big|v^\top (X_{\ddag i}- \E{X_\ddag})\big|- n \big|v^\top (\bar X_\ddag-\E{X_\ddag})\big| }{\left(n \sum_i \big[v^\top (X_{\ddag i}-\E{X_\ddag})\big]^2\right)^{1/2}}\ .
\eeq
Note that 
\[v^\top(X_{\ddag i}- \E{X_\ddag}) = w_i t + z_i, \quad w_i := \nu - \eta_i, \quad z_i := v^\top Z_i\ .\]
Simple moment calculations and an application of Chebyshev's inequality leads to
\beq \label{sq-sum}
\sum_i \big[v^\top (X_{\ddag i}-\E{X_\ddag})\big]^2= n + n\nu(1-\nu)t^2 + O_{P}(\sqrt{n}+\sqrt{n}t^2)\ ,
\eeq
and
\beq \label{bar-sum}
n \big|v^\top (\bar X_\ddag-\E{X_\ddag})\big|= O_{P}(\sqrt{n}+ \sqrt{n}t)\ . 
\eeq
Chebyshev's inequality also implies
\beq \label{abs-sum}
\sum_i \big|v^\top (X_{\ddag i}- \E{X_\ddag})\big| = n \Psi_1(t) + O_P(\sqrt{n \Psi_2(t)})\ ,
\eeq
where
\begin{eqnarray}
\Psi_1(t) := \E[|w_i t + z_i|] &=& 2 \nu (1-\nu) t \left[\Phi((1-\nu) t) - \Phi(-\nu t)\right] \label{eq:moment_1} \\ 
&& \qquad +\  2(1-\nu)\phi(\nu t) +2\nu \phi((1-\nu)t) \ , \notag \\ 
\Psi_2(t) := \Var[|w_i t + z_i|]& =& \nu(1-\nu)t^2+1 - \Psi^2_1(t)\ ,\nonumber
\end{eqnarray}
where $\phi$ and $\Phi$ denote the standard normal density and distribution function, respectively.
In order to prove that the test is powerful, we use an extraction argument: we only need to prove that for any subsequence of $(n,p,\mu)$ there exists a subsequence such that the test is powerful. This allows us to assume that $t\to a \in [0,\infty]$. 
\bitem
\item If $t \to \infty$, $\Psi_1(t) \sim 2 \nu (1-\nu) t$ while $\Psi_2(t) = \nu (1-\nu) [1 - 4 \nu (1-\nu)] t^2 + 1 + o(1)$, so that \eqref{abs-sum} is equal to $2\nu(1-\nu)n t + O_P(\sqrt{n}t)$.  With \eqref{sq-sum} and \eqref{bar-sum}, this implies that $V \to 2\sqrt{\nu(1-\nu)}$. This quantity is larger than $\sqrt{2/\pi}$ as soon as $|\nu-\frac12|<\sqrt{(\pi-2)/(4\pi)}$.
\item If $t = o(1)$, a Taylor development gives
\beq\label{eq:taylor_moment_1}
\Psi_1(t) = \sqrt{\tfrac2\pi} \big[1 + \tfrac{t^2}{2} \nu (1-\nu) - \tfrac{t^4}{24} \nu (1-\nu) (1 - 3\nu + 3\nu^2)\big] + O(t^6), 
\eeq
while $\Psi_2(t) = O(1)$,  
so that, with \eqref{sq-sum} and \eqref{bar-sum}, we have
\beqn
V 
&\ge& \frac{\sqrt{\tfrac2\pi} \big[1 + \tfrac{t^2}{2} \nu (1-\nu) - \tfrac{t^4}{24} \nu (1-\nu) (1 - 3\nu + 3\nu^2)\big] + O(t^6) + O_P(1/\sqrt{n})}{\big(1 + \nu(1-\nu) t^2 + O_P(1/\sqrt{n})\big)^{1/2}} \\
&=& \sqrt{\tfrac2\pi} \big[1 + \tfrac{t^4}{24} \nu(1-\nu) (6 \nu -6\nu^2-1)\big] + O(t^6) + O_P(1/\sqrt{n}),
\eeqn
so that the test is powerful when $t^4 \gg \zeta^{1/2} $ and $|\nu -\tfrac12| < \tfrac1{2\sqrt{3}}$ so that $6 \nu -6\nu^2-1 > 0$. 

\item If $t\to a\in (0,\infty)$, the right-hand side in \eqref{V-lb} converges to $f(a) := \Psi_1(a)/\sqrt{1+\nu(1-\nu)a^2}$. Thus, we only need to show that $f(a) > \sqrt{2/\pi}$ for any $a\in (0,\infty)$. Since $f(0) = \sqrt{2/\pi}$, it suffices to show that $f'(a) > 0$ for $a > 0$. This amounts to studying the sign of the following expression: 
\beqn
A:= \Psi_1'(a)(1+\nu(1-\nu)a^2)- \Psi_1(a)\nu(1-\nu)a\ .
\eeqn
After elementary calculations, we obtain 
\[
A = 2\nu(1-\nu) \left[\int^{\nu a}_{0} \phi(x)dx+ \int^{(1-\nu) a}_{0} \phi(x)dx- (1-\nu) a\phi(\nu a)-\nu a \phi((1-\nu)a)\right] \ .
\]
Since the function $\phi$ is decreasing on $\mathbb{R}^+$, it follows that $A>0$ for $\nu=1/2$. By symmetry, we can assume that $\nu > 1/2$, then 
\beqn
\frac{A}{2\nu(1-\nu)}&>& [(1-\nu)a \phi((1-\nu)a) + (\nu a - (1-\nu)a) \phi(\nu a)] + (1-\nu) a \phi((1-\nu)a) \\[-.1in] 
&& \quad - \ (1-\nu) a\phi(\nu a)-\nu a \phi((1-\nu)a) \\
&= & (3\nu -2) a \left[\phi(\nu a)- \phi((1-\nu) a)\right]\ ,
\eeqn
which is positive for $3\nu-2\leq 0$. In conclusion, the test is powerful as long as $|\nu-\frac{1}{2}|< \frac{1}{6}$.
\eitem

\subsection{Proof of \prpref{sym_estim_malk3}} \label{sec:proof_sym_estim_malk3}

The arguments are analogous to those in \secref{proof-selection-known}, but more technical in the details.
We continue with the notation introduced in \secref{malk3_proof} and introduce some more.
Define 
\[
\cU_\delta = \big\{u \in \bbR^p : \|u\| = 1, \|u\|_0 \le s \text{ and } |u^\top \mudif| \le (1-\delta) \|\mudif\|\big\}\ .
\]
(Note that this differs from the definition in \secref{proof-selection-known}.)
As in \secref{proof-selection-known}, it suffices to show that, for any fixed $\delta\in (0,1)$, with probability tending to one, $\cU_\delta$ does not contain any $\hat u$ defined as in \eqref{eq:malk3_estimation}. We shall provide uniform controls of the first absolute and the second centered (sample) moments in a direction $u\in \cU_\delta$, namely, $\sum_i \big|u^\top (X_i -\bar X)\big|$ and $\sum_i \big[u^\top (X_i -\bar X)\big]^2$. 
Recall that $C$ denotes a positive constant that may change with each appearance. 

\medskip

\noindent
{\bf STEP 1: Control of the first absolute moment.} Denote $t_{\star}= \|\mudif\|^2/\|\bSigma^{1/2}\mudif\|$. For any $u \in \bbR^p$, let $t_u = u^\top \mudif/\|\bSigma^{1/2}u\|$ and $v_u = \bSigma^{1/2} u/\|\bSigma^{1/2} u\|$. 
Observe that $v_u\in \cV$ with $\cV$ defined in \eqref{cV}. Uniformly, over $u\in\cU_{\delta}$, we have
\begin{eqnarray}
\frac{\sum_i \big|u^\top (X_i -\bar X)\big|}{\|\bSigma^{1/2}u\|} \nonumber
&\le& \sum_i \big|(\eta_i-\nu)t_u+ v_u^\top Z_i\big| + n|\bar{\eta}-\nu||t_u| + n |v_u^{\top}\bar{Z}| \nonumber\\
&\leq & Q_1^{\ddag}(v_u)+ n|\bar{\eta}-\nu|\sup_{u\in \cU_{\delta}}|t_u| + n \sup_{v\in \cV}|v^{\top}\bar{Z}|\ ,
\label{eq:upper_numerator}
\end{eqnarray}
where $Q_1^{\ddag}(v_u):=  \sum_i \big|(\eta_i-\nu)t_u+ v_u^\top Z_i\big|$.
First, $\sup_{v\in \cV}|v^{\top} \bar Z|^2$ is distributed as the supremum of $\binom{s}{p}$ (possibly dependent) $\chi^2_s$ random variables. Using an union bound together with Lemma \ref{lem:birge_chi2} leads to $\sup_{v\in \cV} n |v^{\top} \bar Z|\leq 4n\sqrt{\zeta}$ with probability going to one. By Chebyshev's inequality, 
\beq \label{eta_cheb}
\bar{\eta}-\nu=O_{P}(n^{-1/2}) \ . 
\eeq
Since the  the $2s$-sparse Riesz constant of $\bSigma$ is bounded, 
\beq \label{t_u}
\sup_{u\in \cU_{\delta}}|t_u| = O(t_{\star}) = o(1) \ . 
\eeq
Hence, $n|\bar{\eta}-\nu|\sup_{u\in \cU_{\delta}}|t_u| = o_P(\sqrt{n})$, with $\sqrt{n} \le n\sqrt{\zeta}$ eventually, since $n \zeta = s \log(e p/s) \to \infty$ by assumption.

Let us turn to $Q_1^{\ddag}(v_u)$.   
Let $\eta = (\eta_i)_{i=1}^n$ and let $\P_\eta$ denote the probability given $\eta$.
As before, we use the Gaussian isoperimetric inequality \citep[Prop.~2.1]{MR1600888} to prove that, for any $x>0$, 
\[
\P_\eta\left[\sup_{u\in \cU_\delta} R_{u,\eta} \ge \E_\eta\big(\sup_{u\in \cU_\delta} R_{u,\eta}\big) + \sqrt{2nx}\right]\leq e^{-x}\ ,
\]
where
\[
R_{u,\eta} := Q_1^\ddag(v_u) -\E_\eta[ Q_1^{\ddag}(v_u)] \ .
\]
The deviations of the differences also follow a subgaussian distribution as  proved in Section~\ref{sec:proofs_deviation}.
\begin{lem}\label{lem:process_subgaussian2}
The process $( R_{u,\eta}, u\in \cU_\delta)$ is subgaussian in the sense that there is a constant $C > 0$ such that
\[\P\left[| R_{u_1,\eta}-  R_{u_2,\eta} |>x\right]\leq 2 \exp\big[-C\tfrac{x^2}{  n \|u_1-u_2\|^2}\big]\ , \quad \forall v,w \in \cV,\ x>0\ .\]
\end{lem}
Thus, we can apply a maximal inequality \citep[Cor.~2.2.8]{MR1385671} based on the metric entropy of $\cU_{\delta}$ with respect to the Euclidean metric, and obtain
\beqn
\E_\eta\Big[\sup_{u\in \cU_\delta} R_{u,\eta}\Big]\leq Cn \sqrt{\zeta}\ .
\eeqn
Furthermore, 
uniformly in $u \in \cU_\delta$,
\[
\Big|\E_\eta[ Q_1^{\ddag}(v_u)]- \E[ Q_1^{\ddag}(v_u)]\Big| \leq Cn|\nu-\bar{\eta}|(1\vee |t_u|) = O(n|\nu-\bar{\eta}| (1\vee |t_\star|) = o_{P}(n\sqrt{\zeta}) \ ,
\] 
as we saw earlier.
Noting that $\E[ Q_1^{\ddag}(v_u)]=n\Psi_1(t_u)$ (with $\Psi_1$ defined in \eqref{eq:moment_1}), we get
\[
\sup_{u\in \cU_{\delta}} \Big| Q_1^{\ddag}(v_u) - n \Psi_1(t_u) \Big| \le C n \sqrt{\zeta}\ ,
\]
with probability tending to one, which combined with \eqref{eq:upper_numerator} and using the bounds we obtained for the last two terms there, implies
\beq\label{eq:control_uniform_Q1ddag}
 \frac{\sum_i \big|u^\top (X_i -\bar X)\big|}{\|\bSigma^{1/2}u\|}\leq n \Psi_1(t_u)+ Cn \sqrt{\zeta}\ , \quad \quad \forall u\in \cU_{\delta}\ .
\eeq

\medskip

\noindent 
{\bf STEP 2: Control of the second moment.} Uniformly, over $u\in\cU_{\delta}$, we have
\beqn
\frac{\sum_i \big[u^\top (X_i -\bar X)\big]^2}{\|\bSigma^{1/2}u\|^2}&=& n\bar{\eta}(1-\bar{\eta}) t_u^2+ Q^{\circ}_2(v_u)- n(v_u^\top \var{Z})^2+2t_u \sum_i\big[v_u^\top (Z_i -\bar Z)\big](\eta_i-\bar{\eta}) \\
&\geq & n\bar{\eta}(1-\bar{\eta}) t_u^2+ \inf_{v\in \cV}Q^{\circ}_2(v)- 2t_u\sup_{v\in \cV} \big|v^\top\sum_iZ_i(\eta_i-\bar{\eta})\big| \ .
\eeqn
As explained in a previous proof (Section~\ref{sec:malk3_proof}), 
there is a constant $C > 0$ such that
$\inf_{v\in \cV} Q^{\circ}_2(v) \geq  n - C n\sqrt{\zeta}$ with probability going to one. 
Also, conditionally on $\eta$, $\sup_{v\in \cV}\big[v^\top\sum_i Z_i(\eta_i-\bar{\eta})\big]^2\big[n\bar{\eta}(1-\bar{\eta})\big]^{-1}$  is distributed as the supremum of $\binom{p}{s}$  possibly dependent $\chi^2(s)$ random variables. Hence, there is a constant $C > 0$ such that
$$\sup_{v\in \cV} \big|v^\top\sum_iZ_i(\eta_i-\bar{\eta})\big|\leq C \sqrt{\bar{\eta}(1-\bar{\eta})}n\sqrt{\zeta} \ ,$$ 
with probability going to one.  
Gathering these bounds with \eqref{t_u} and \eqref{eta_cheb}, with probability going to one, we get
\beq
\frac{\sum_i \big[u^\top (X_i -\bar X)\big]^2}{\|\bSigma^{1/2}u\|^2}\geq n+  n \eta(1-\eta) t_u^2 -Cn \sqrt{\zeta} \ , \quad \quad \forall u\in \cU_{\delta}\ . \label{eq:lower_denominator}
\eeq 
Using that fact that $\zeta=o(1)$ and $t_{\star}=o(1)$, we show similarly that
\beq\label{eq:upper_denominator}
\sup_{u \in \cU_{\delta}} \frac{\sum_i \big[u^\top (X_i -\bar X)\big]^2}{\|\bSigma^{1/2}u\|^2}\leq n+o_P(n)\ .
\eeq

\medskip

\noindent 
{\bf STEP 3: Control of the statistic \eqref{eq:malk3_estimation}}.
Gathering \eqref{eq:control_uniform_Q1ddag}, \eqref{eq:lower_denominator} and \eqref{eq:upper_denominator} with \eqref{t_u}, we obtain
\beqn
\lefteqn{\left[\frac{\sum_i \big|u^\top (X_i - \bar X)\big|}{\left(\sum_i \big[u^\top (X_i - \bar X)\big]^2\right)^{1/2}}-\sqrt{\frac{2}{\pi}}\right]_+ \left(\frac{1}{n}\sum_i \big[u^\top (X_i - \bar X)\big]^2\right)^{2}}&&\\ &\leq& \left[\frac{\Psi_1(t_u)+ C\sqrt{\zeta}}{1+\eta(1-\eta) t_u^2 - C'\sqrt{\zeta} }- \sqrt{\frac{2}{\pi}} \right]_+  \|\bSigma^{1/2}u\|^4(1+o_P(1))\\ 
&\leq& \sqrt{\frac{2}{\pi}}\frac{(u^{\top}\mudif)^4}{24}\nu(1-\nu)(6\nu-6\nu^2-1)+ O_P[t_u^2(u^{\top}\mudif)^4]+ C\|\bSigma^{1/2}u\|^4 \sqrt{\zeta}\\
&\leq & \sqrt{\frac{2}{\pi}}(1-\delta)^4 \frac{\|\mudif\|^4}{24}\nu(1-\nu)(6\nu-6\nu^2-1) + o_P(u^{\top}\mudif)^4 + C (\lambda_s^{\max}(\bSigma))^2 \sqrt{\zeta} \\
&\leq & \sqrt{\frac{2}{\pi}}(1-\delta)^4 \frac{\|\mudif\|^4}{24}\nu(1-\nu)(6\nu-6\nu^2-1)(1+o_P(1))\ , 
\eeqn
where we used the Taylor development \eqref{eq:taylor_moment_1}, and in the last line, we used the fact that the sparse eigenvalues of $\bSigma$ are assumed to be bounded, \eqref{t_u} and  $t_{\star}^4 \gg \zeta$.

Conversely, for $u=\hat u = \mudif/\|\mudif\|$, one can show in exactly the same way that the statistic in \eqref{eq:malk3_estimation} is larger than $\sqrt{2/\pi}\frac{\|\mudif\|^4}{24}\nu(1-\nu)(6\nu-6\nu^2-1)(1+o_P(1))$. Thus, for any fixed $\delta\in (0,1)$, $\P[\hat{u}\in \cU_{\delta}]=o(1)$. This concludes the proof.

\subsection{Proof of \prpref{malk4}}

The proof is similar to that of \prpref{malk3}.
We still assume that $\E[X] = 0$ without loss of generality (since the statistic is translation invariant) and work with the standardized observations \eqref{standard}.  We also use the same notation, except that we redefine the statistic $V$ to be
\[
V 
= \max_{\|u\|_0\leq s} \frac{\sqrt{n} \sum_i \left[u^\top (X_i- \bar X)\right]^3}{\left(\sum_i \big[u^\top (X_i-\bar X)\big]^2\right)^{3/2}} 
= \sqrt{n} \max_{v \in \cV} \frac{Q_3(v)}{Q_2(v)^{3/2}}\ ,
\]
where $Q_2$ is defined in \eqref{Q2Q4}, $\cV$ in \eqref{cV}, and 
$Q_3(v) := \sum_i \big[v^\top (X_{\ddag i}-\bar {X_\ddag})\big]^3$

\subsubsection{Under $H_0$}
Suppose we are under the null, so that $X_{\ddag i} = Z_i$.  
Noting that 
\[
Q_3(v) = Q^\circ_3(v) -3 Q_2^\circ(v) (v^\top \bar Z) + 2 n (v^\top \bar Z)^3\ , 
\]
where $Q_2^\circ(v) := \sum_i (v^\top Z_i)^2$ and $Q_3^\circ(v) := \sum_i (v^\top Z_i)^3$, we upper bound $V$ by
\begin{eqnarray}
n^{-1/2} V 
&\le& \frac{\max_{v\in\cV} Q_3(v)}{\min_{v\in\cV} Q_2(v)^{3/2}} \notag \\
&\le&  \frac{\max_{v\in\cV} Q^\circ_3(v) + 3\max_{v\in\cV} Q^\circ_2(v)|v^\top \bar Z| + 2n \max_{v\in\cV} |v^\top \bar Z|^3}{\min_{v\in\cV} Q_2(v)^{3/2}} \ . \label{eq:upper_V}
\end{eqnarray}

The denominator has already been considered in the previous proof \eqref{eq:upper_bound_Q2}, so that we concentrate on the numerator. 
We rely on a chaining argument combined with the following deviation inequalities.
\begin{lem}\label{lem:deviation_polynom3}
For any $x>0$, and any unit vectors $v$, $w$, any $a\in\mathbb{R}^n$,  we have
\beq \label{eq:deviation_moment_3}
\P\Big[|\sum_{i=1}^n a_i (v^{\top}Z_i)^3|\geq C_1\|a\| (\sqrt{x}+1) + C_2 \|a\|_{\infty}\sqrt{\log(n)}x+ C_3\|a\|_{\infty} x^{3/2}\Big] \le 2e^{-x}\ ,
\eeq
\beq \label{eq:deviation_moment_3-b}
\P\big[|Q^\circ_3(v)|\geq C_1\sqrt{n}(\sqrt{x}+1) + C_2 x^{3/2}\big] \leq 2e^{-x}\ ,
\eeq
and
\beq \label{eq:deviation_difference_3}
\P\big[|Q^\circ_3(v)-Q^\circ_3(w)|\geq\|v - w\| \big(C_1\sqrt{n}(\sqrt{x}+1) + C_2 x^{3/2}\big) \big]\leq 6e^{-x}\ ,
\eeq
where $C_1, C_2, C_3$ are positive universal constants.
\end{lem}
The proof is postponed to Section~\ref{sec:proofs_deviation}.

Fix some $x>0$. 
For any integer $j$, set $\epsilon_j=2^{-j}$, and let $N_j$ denote the $\eps_j$-covering number of $\cV$.  
Note that $N_j \le \binom{p}s (1 + 2^{j+1})^s$ by \eqref{V_net}.
Let $\cV_j \subset \cV$ be an $\eps_j$-net for $\cV$ of cardinality $N_j$.
Let $\Pi_j : \cV \mapsto \cV_j$ be such that $\|\Pi_j v - v\| \le \eps_j$ for all $v \in \cV$.  Since $v \mapsto Q_3^\circ(v)$ is almost surely continuous, we have the following decomposition:
\[Q_3^\circ(v)= Q_3^\circ(\Pi_0 v)+ \sum_{j=1}^{\infty } \big[ Q_3^\circ(\Pi_{j+1}v)- Q_3^\circ(\Pi_{j}v) \big]\ , \]
from which we deduce 
\[\sup_{v\in\cV}|Q_3^\circ(v)|\leq \sup_{v\in \cV} |Q_3^\circ(\Pi_0 v)|+ \sum_{j=1}^{\infty } \sup_{v\in \cV}|Q_3^\circ(\Pi_{j+1}v)- Q_3^\circ(\Pi_{j}v)|\ ,
\]
We simultaneously control the deviations of all these suprema. 

Combining \eqref{eq:deviation_moment_3-b} together with an union bound and the fact that $\log N_0 \le s \log(3 e p/s)$, we obtain
\[\sup_{v\in \cV} |Q_3^\circ(\Pi_0 v)|\leq C\left[\sqrt{ns\log\left(\frac{ep}{s}\right)}+ \left[s\log\left(\frac{ep}{s}\right)\right]^{3/2}+\sqrt{nx}+x^{3/2}\right] \]
with probability larger than $1-e^{-x}$. 

For any integer $j \ge 0$, the range of $v \mapsto (\Pi_j v, \Pi_{j+1}v)$ is a set with cardinality at most $N_j N_{j+1} \le N_{j+1}^2$.
Moreover, by the triangle inequality, $\|\Pi_{j}v-\Pi_{j+1}v\|\leq \|\Pi_{j}v-v\| + \|\Pi_{j+1}v -v\| \le 3\epsilon_{j+1}$, for any $v\in \cV$. Hence, by \eqref{eq:deviation_difference_3}, we get 
\beqn
\frac1{3\epsilon_{j+1}} \sup_{v\in \cV}|Q_3^\circ(\Pi_{j+1}v)- Q_3^\circ(\Pi_{j}v)|
&\leq&C\left[\sqrt{ns\log\left(\frac{ep}{s}\right)}+ \left[s\log\left(\frac{ep}{s}\right)\right]^{3/2}\right]\\ 
&+& C\left[ \sqrt{ns\log\left(1+ \frac{2}{\epsilon_{j+1}}\right)}+ \left(s\log\left(1+ \frac{2}{\epsilon_{j+1}}\right)\right)^{3/2}\right]\\
&+ & C\left[\sqrt{nx}+x^{3/2}+\sqrt{nj}+ j^{3/2}\right]
\eeqn
with probability larger than $1- 6e^{-j}e^{-x}$. 
Gathering all these deviation inequalities leads to 
\beq\label{eq:upper_Q3}
\sup_{v\in\cV}|Q_3^\circ(v)|\leq  C\left[\sqrt{ns\log\left(\frac{ep}{s}\right)}+ \left[s\log\left(\frac{ep}{s}\right)\right]^{3/2}+ \sqrt{nx}+x^{3/2}\right]\ ,
\eeq 
with probability larger than $1-Ce^{-x}$.  In fact, $\sup_{v\in\cV}|Q_3^\circ(v)|$ is the leading term in \eqref{eq:upper_V} since, for any $x>0$,  $\P_{0}\left[\max_{v\in \cV}|v^{\top}\bar{Z}|>C_1 \sqrt{s\log(ep/s)/n} +\sqrt{2 x/n}\right]\leq e^{-x}$ and \[\P_{0}\left[\max_{v\in \cV}Q_2^\circ(v)\geq n\left(1+\sqrt{s/n}+ \sqrt{2s\log(ep/s)/n}+ \sqrt{2x/n}\right)^2\right]\leq e^{-x}\] by Lemma \ref{lem:concentration_vp_wishart}.

In conclusion, under the null hypothesis, for any $0<x<n$
\beq\label{eq:control_V_H0}
V \leq \frac{C_1\left[\sqrt{\frac{s\log\left(\frac{ep}{s}\right)}{n}}+ \frac{\left[s\log\left(\frac{ep}{s}\right)\right]^{3/2}}{n}\right]+  C_2\left[\frac{\sqrt{x}}{n}+\frac{x^{3/2}}{n}\right]}{\left[1- \frac{1}{n}- 5\sqrt{\frac{s\log\left(\frac{ep}{s}\right)}{n}}- 2\sqrt{\frac{2x}{n}}\right]^{3/2}_+}
\eeq 
with probability larger than $1-Ce^{-x}$.

\subsubsection{Under $H_1$}
Under the alternative, let $v = \bSigma^{1/2}\mudif/\|\bSigma^{1/2} \mudif\|$, as before.  
Then $v^\top X_{\ddag i} = \omega_i t + z_i$, where $t := \|\mudif\|^2/\|\bSigma^{1/2} \mudif\|$, $\omega_i:=\nu-\eta_i$ and $z_i := v^\top Z_i, i=1,\dots,n$ are iid standard normal. For any integer $k$, define $S_k:= \sum_{i}(\omega_it+z_i)^k$. Then, $Q_2(v)=S_2-S_1^2/n$ and $Q_3(v)=S_3+3S_2S_1/n+2S_1^3/n^2$. By Chebyshev's inequality, we have $Q_2(v)= n (1+ \nu(1-\nu)t^2)(1+O_P(1/\sqrt{n})$.  Similarly, $Q_3(v) =n \nu (1-\nu)(1-2\nu)t^3+ O_P(\sqrt{n}(1\vee t^3))$.  
If $t=o(1)$, then $V\sim \nu(1-\nu)(1-2\nu)t^3$ in probability. Comparing $t^3$ with \eqref{eq:control_V_H0} and using Condition \eqref{eq:condition_V} on $t$, we conclude that the test is asymptotically powerful. When $t\to r\in(0,\infty]$, then $V$ converges in probability towards a positive constant and the test is asymptotically powerful. As usual, we handle the case where $t$ does not converge by extracting subsequences.

\subsection{Proof of \prpref{malk5}}

The proof is then similar to that of \prpref{malk3}.
We still assume that $\E[X] = 0$ without loss of generality (since the statistic is translation invariant) and work with the standardized observations \eqref{standard}.  We also use the same notation, except that we redefine the statistic $V$ to be
\beqn
V &=& \max_{\|u\|_0\leq s} \frac{\sum_i \big[u^\top (X_i-\bar{X})\big]^2 \sign[u^\top (X_i-\bar{X})]}{\sum_i \big[u^\top (X_i - \bar X)\big]^2} \\
&=& \max_{v\in\cV} \frac{Q_2^{\rm sign}(v)}{Q_2(v)} \ ,
\eeqn
where $Q_2$ was defined in \eqref{Q2Q4}, $\cV$ in \eqref{cV}, and $Q_2^{\rm sign}(v) := \sum_i [v^\top (X_{\ddag i}-\bar X_\ddag)]^2 \sign[v^\top (X_{\ddag i}-\bar X_\ddag)]$.   

\subsubsection{Under $H_0$}
Under the null, we have $X_{\ddag i} = Z_i$, and bound $V$ from above by 
\[
V \le \frac{\max_{v\in\cV} Q_2^{\rm sign}(v)}{\min_{v\in\cV} Q_2(v)} \ .
\]
We already controlled the denominator in \eqref{eq:upper_bound_Q2}.
In particular, $\min_{v\in\cV} Q_2(v) \ge n/2$ with probability tending to 1.
We proceed with bounding the numerator. First, define 
\beqn
Q_2^{\rm sign,\circ}(v) := \sum_i (v^\top Z_i)^2 \sign(v^\top Z_i)
\eeqn
and observe that
\beqn
|Q_2^{\rm sign,\circ}(v)- Q_2^{\rm sign}(v)| 
&\leq& \sum_i (v^\top Z_i)^2\big| \sign(v^\top Z_i)- \sign(v^\top (Z_i-\bar{Z}))\big| \\ 
& &+ \ \sum_i \big|(v^\top Z_i)^2- (v^\top (Z_i-\bar{Z}))^2\big| \\
&\leq & 2n (v^{\top} \bar{Z})^2 + |v^{\top} \bar{Z}| \sum_{i} \big|2 v^{\top} Z_i - v^{\top} \bar Z\big| \\
&\leq & 3 n (v^{\top} \bar{Z})^2 + 2 |v^{\top} \bar{Z}| \sum_{i} |v^{\top} Z_i|\ ,
\eeqn
where in the second line we used the fact that $\sign(a) - \sign(a + b) = 0$ when $|a| > |b|$.
Thus,  the following decomposition holds
\beq
\max_{v\in\cV} Q_2^{\rm sign}(v)\leq \max_{v\in\cV} Q_2^{\rm sign,\circ}(v)+ 3n \max_{v\in\cV} (v^{\top} \bar{Z})^2+ 2 \Big(\max_{v\in\cV}|v^{\top} \bar{Z}|\Big) \Big(\max_{v\in\cV} Q_1^\circ(v)\Big)\ .
\eeq
By \eqref{barZ}, $n \max_{v\in\cV} (v^{\top} \bar{Z})^2 \leq 9 s\log(ep/s)$ with probability going to one. Furthermore, by \eqref{eq:concentration_A1} and \eqref{eq:upper_expect}, $\max_{v\in\cV} Q_1^\circ(v) \leq n$ with probability going to one. It remains to control $\max_{v\in\cV} Q_2^{\rm sign,\circ}(v)$. We have the following deviation inequalities.

\begin{lem}\label{lem:control_Q2}
For any $t>0$ and for any normed vectors $v$, $w$ such that $v^{\top}w\geq 0$, 
\beq
\P\left[Q_2^{\rm sign,\circ}(v)\geq \sqrt{8n t}+ 2 t \right] \le  e^{-t}\label{eq:deviation_Q2sgn} \ ,
\eeq
and
\beq
\P\left[Q_2^{\rm sign,\circ}(v)- Q_2^{\rm sign,\circ}(w) \geq 6\|v-w\| \sqrt{n t}+ 2\|v-w\| t \right] \le e^{-t}\ . \label{eq:deviation_Q2sgn-diff}
\eeq
\end{lem}
The proof is postponed to Section~\ref{sec:proofs_deviation}.

In order to control $\max_{v\in\cV} Q_2^{\rm sign,\circ}(v)$, we combine the same chaining argument developed in the proof of \prpref{malk4} with the deviation inequalities of Lemma \ref{lem:control_Q2}. 
This leads to  
\beqn
\max_{v\in \cV} Q_2^{\rm sign,\circ}(v)\leq C_1 \sqrt{ns\log\left(\frac{ep}{s}\right)}+  C_2s\log\left(\frac{ep}{s}\right)+ C_3(\sqrt{nx}+x)\ , 
\eeqn 
with probability at least $1-e^{-x}$, valid for any $x>0$. 

Gathering all these bounds and using $\log(ep/s)=o(n)$, we conclude that
\beq \label{V-null-malk5}
V \le C \sqrt{\zeta},
\eeq
with probability going to one under the null.

\subsubsection{Under $H_1$}
Under the alternative, assume without loss of generality that $\nu < 1/2$.
With $v = \bSigma^{1/2} \mudif/\|\bSigma^{1/2} \mudif\|$, as before, we have $v^{\top}X_{\ddag i} = \xi_i t + z_i$, where $t = \|\mudif\|^2/\|\bSigma^{1/2} \mudif\|$, $\xi_i=\nu-\eta_i$, and $z_i \sim \cN(0,1)$.
We have $Q_2(v) = n + n \nu (1-\nu) t^2 + O_P(\sqrt{n + n t^2})$.  
And then, proceeding as we did earlier, 
\[
Q_2^{\rm sign}(v) \ge  Q_2^{\rm sign,\circ}(v) -  3n \big(v^{\top} \bar{X_\ddag} \big)^2 - 2|v^{\top} \bar{X_\ddag}| \sum_{i}|v^{\top} X_{\ddag i}|\ .
\]
Using Chebyshev's inequality, we have
\beqn
|v^{\top} \bar{X_\ddag}|&\leq&  |\bar{\xi}|t +|\bar{Z}|= O_{P}\left(\sqrt{\frac{\nu(1-\nu)}{n}t}+\frac{1}{\sqrt{n}}\right)\\
\sum_{i}|v^{\top} X_{\ddag i}|&=&  n \, O_{P}\left(1+ \nu(1-\nu) t\right)\ . 
\eeqn
And with some tedious, but elementary calculations, we find that
\beqn
\frac1n \E \big[Q_2^{\rm sign,\circ}(v)\big] 
&=& (1-\nu) \big[ (\nu^2 t^2+1)(1-2\Phi(\nu t)) - 2 \nu t \phi(\nu t) \big] \\
&& \quad + \ \nu \big[ ((1-\nu)^2 t^2+1)(1-2\Phi(-(1-\nu)t)) + 2 (1-\nu) t \phi((1-\nu)t) \big] \\
&=:& g(\nu, t) \ ,
\eeqn
with $g(\nu, t)$ being strictly increasing in $t > 0$.  This is because the first derivative with respect to $t$ is equal to
\[
4 \nu (1-\nu) \big[ t \big(\tfrac12 - \nu\Phi(\nu t) - (1-\nu)\Phi(-(1-\nu)t)\big) + \phi((1-\nu)t) - \phi(\nu t) \big] \ ,
\]
and the second to
\[
4 \nu (1-\nu) \big[ \tfrac12 - \nu\Phi(\nu t) - (1-\nu)\Phi(-(1-\nu)t) \big] \ .
\]
The first derivative is equal to 0 at $t=0$ and the second derivative is bounded from below by 
\[
4 \nu (1-\nu) \big[ \tfrac12 - \nu\Phi(\nu t) - (1-\nu)\Phi(-\nu t) \big] = 4 \nu (1-\nu)(1-2\nu) \big[ -\tfrac12 +\Phi(\nu t)\big] > 0 \ , 
\]
when $t > 0$.  So $\E \big[Q_2^{\rm sign,\circ}(v)\big]$ is indeed increasing in $t > 0$.
Moreover, a Taylor development at $t = 0$ gives
\[
\frac1n \E \big[Q_2^{\rm sign,\circ}(v)\big] = \frac13 \sqrt{\frac2\pi} \nu (1-\nu)(1-2\nu) t^3 + O(t^4)\ .
\]
To this we add the fact that
\[
\frac1n \Var\big[Q_2^{\rm sign,\circ}(v)\big] \le \frac1n (1 + \nu (1-\nu) t^2)\ .
\]
As in the previous proofs, it suffices to prove that the test is asymptotically powerful when $t$ converges to some limit in $\bar \bbR$. (Indeed, if the test is not powerful for some sequence $t$, then one can extract a subsequence such that $t$ is converging and the risk of the test is bounded away from zero.) 
First, we focus  on the case where $t = o(1)$, which is more subtle.  In that case $Q_2(v) \le 2 n$ with probability tending to one, and by Chebyshev's inequality, 
\[
Q_2^{\rm sign}(v) \ge \frac{n}3 \sqrt{\frac2\pi} \nu (1-\nu)(1-2\nu) t^3 + O(n t^4) + O_P(\sqrt{n}).
\]
From this we get that 
\[
V \ge \frac{Q_2^{\rm sign}(v)}{Q_2(v)} \ge \frac{\frac{n}3 \sqrt{\frac2\pi} \nu (1-\nu)(1-2\nu) t^3 + O(n t^4) + O_P(\sqrt{n})}{2n} \sim \frac16 \sqrt{\frac2\pi} \nu (1-\nu)(1-2\nu) t^3,
\]
when $n^{-1/6} \ll t \ll 1$.  With the control of $V$ under the null in \eqref{V-null-malk5}, and our working assumption \eqref{malk5_condition}, we conclude.

Now, if $t\to  l\in (0,\infty]$, then by Chebyshev's inequality again, 
\[
\frac{Q_2^{\rm sign}(v)}{Q_2(v)} \ge (1 - o_P(1)) \frac{\E[Q_2^{\rm sign}(v)]}{\E[Q_2(v)]} \sim \frac{g(\nu, t)}{1 + \nu(1-\nu) t^2} \to 
\begin{cases}
\frac{g(\nu, l)}{1 + \nu(1-\nu) l^2} & \text{if } l < \infty \ ; \\
1 - 2\nu & \text{if } l = \infty \ .
\end{cases} 
\] 
In both cases, the limit on the right-hand side is strictly positive.

\subsection{Proof of \prpref{upper_diagonal_sparse}}

The proof is similar to that of \prpref{upper_test_sparse2}.  We use the same notation.
Let $V$ denote the statistic defined in \eqref{diagonal-stat}.  
Since $\bSigma$ is diagonal, we have 
\beq \label{V-diag-def2}
V = \max_{\|v\|_0 \le s} \frac{v^\top \hat\bSigma_\ddag v}{v^\top \diag(\hat\bSigma_\ddag) v}\ ,
\eeq
by a simple change of variables $v = \bSigma^{1/2} u$.
We will also use the fact that $\|\bSigma^{1/2} u\|_0 = \|u\|_0$ for any vector $u$, since again $\bSigma$ is diagonal.

\subsubsection{Under $H_0$}
We first upper-bound $V$ under the null, starting from 
\[
V \le \frac{\max_{\|v\|_0 \le s} v^\top \hat\bSigma_\ddag v}{\min_{\|v\|_0 \le s} v^\top \diag(\hat\bSigma_\ddag) v} = 
\frac{\lambda_s^{\rm \max}(\hat\bSigma_\ddag)}{\displaystyle \min_{j \in [p]} \hat\sigma_{jj}/\sigma_{jj}} \ .
\]
We already controlled $\lambda_s^{\rm \max}(\hat\bSigma_\ddag) = \hat\lambda_{s, \bSigma}^{\rm max}$ in \eqref{eq:upper_lambda_s}, where we found that 
\[
\lambda_s^{\rm \max}(\hat\bSigma_\ddag) \le 1 + 15 (\sqrt{\zeta} \vee \zeta)\ ,
\]
with probability tending to one.
Moreover, using the fact that $\hat\sigma_{jj}/\sigma_{jj} \sim \frac1n \chi^2_{n-1}$, for any $t > 0$, we have 
\[
\P_0\left(\min_{j \in [p]} \frac{\hat\sigma_{jj}}{\sigma_{jj}} \le 1 + 2 \sqrt{\frac{t}n} + 2 \frac{t}n\right) \le p e^{-t}\ ,
\] 
using \lemref{birge_chi2} and the union bound, so that
\[
\min_{j \in [p]} \frac{\hat\sigma_{jj}}{\sigma_{jj}} \ge 1 - 3 \sqrt{\frac{\log p}n} - 3 \frac{\log p}n\ ,
\]
with probability tending to one.
Since $\log p = o(n)$ and $\frac1n \log p = O(\zeta)$, we conclude that there is a constant $C$ such that
\beq \label{V-null-diag}
V \le 1 + C \big(\sqrt{\zeta} \vee \zeta\big). 
\eeq
with probability tending to one under the null.

\subsubsection{Under $H_1$}
Under the alternative, we choose $v = \mudif_\ddag/\|\mudif_\ddag\|$ in \eqref{V-diag-def2} as we did in \secref{proof-selection-known}.

For the numerator, we saw in \eqref{u-star} that 
\[
v^\top \hat\bSigma_\ddag v
\ge 1 + (1-o_P(1)) \nu (1-\nu) \|\mudif_\ddag\|^2\ .
\]

For the denominator, we obtain
\[
v^\top \diag(\hat\bSigma_\ddag) v = \sum_{j=1}^p \hat \sigma_{\ddag jj} v_j^2\ ,
\]
with
\[
\hat \sigma_{\ddag jj} = \frac1n \sum_{i=1}^n \big[ Z_{ij} -\bar Z_j - (\eta_i -\bar\eta) \mudif_{\ddag j}\big]^2\ .
\]
Conditional on $\eta_1, \dots, \eta_n$, we have that $\hat \sigma_{\ddag 11}, \dots, \hat \sigma_{\ddag pp}$ are independent with $n\hat \sigma_{\ddag jj}$ having the chi-squared distribution with $n-1$ degrees of freedom and non-centrality parameter $B_j := \mudif_{\ddag j}^2 \sum_i (\eta_i - \bar\eta)^2 = \mudif_{\ddag j}^2 n \bar\eta(1-\bar\eta)$, so that $v^\top\diag(\hat\bSigma_\ddag) v$ has (conditional) mean
\[
\sum_j \frac1n \big(n-1 + B_j\big) v_j^2 = 1 - \frac1{n} + \bar\eta (1- \bar\eta) \frac1{\|\mudif_\ddag\|^2} \sum_j \mudif_{\ddag j}^4\ ,
\]
and (conditional) variance
\[
\sum_j \frac1{n^2} \big(2(n-1) + 4 B_j\big) v_j^4 = \frac{2(n-1)}{n^2} \frac1{\|\mudif_\ddag\|^4} \sum_j \mudif_{\ddag j}^4 + \frac1{\|\mudif_\ddag\|^4} \frac4n \bar\eta (1- \bar\eta) \sum_j \mudif_{\ddag j}^6\ .
\]
By Chebyshev's inequality, and the fact that $\bar\eta = \nu + O_P(1/\sqrt{n})$, we get that 
\begin{eqnarray}
v^\top \diag(\hat\bSigma_\ddag) v 
&\le& 1 + \nu (1-\nu) \frac1{\|\mudif_\ddag\|^2} \sum_j \mudif_{\ddag j}^4 \notag \\
&+& \ O_P\left(\frac1{\sqrt{n}}\right) \frac1{\|\mudif_\ddag\|^2} \left[\sum_j \mudif_{\ddag j}^4 + \sqrt{\sum_j \mudif_{\ddag j}^4} + \sqrt{\sum_j \mudif_{\ddag j}^6}\right] \label{diag-term}.
\end{eqnarray}
We have $\sum_j \mudif_{\ddag j}^k \le \|\mudif_\ddag\|^{k}$ and $\big(\sum_j \mudif_{\ddag j}^k\big)^{1/k} \le \|\mudif_\ddag\|$, for any $k \ge 2$, 
so that the remainder term in \eqref{diag-term} is of order $O_P\big(\frac1{\sqrt{n}}\big) (1 + \|\mudif_\ddag\|^2)$.

Gathering all of the above, we conclude that
\beqn
\frac{v^\top \hat\bSigma_\ddag v}{v^\top \diag(\hat\bSigma_\ddag) v} 
&\ge& 1 + \nu (1-\nu)\frac{ \Big(\|\mudif_\ddag\|^2 - \frac1{\|\mudif_\ddag\|^2} \sum_j \mudif_{\ddag j}^4\Big)}{1+ \|\mudif_\ddag\|^2 } 
 + \ o_P\left(\frac{\|\mudif_\ddag\|^2}{1+ \|\mudif_\ddag\|^2 }\right) + O_P\left(\frac1{\sqrt{n}}\right) \ .
\eeqn
Note that when \eqref{upper_diagonal_sparse} holds, and $p \to \infty$, $(1+\|\mudif_\ddag\|^2)/\sqrt{n} = o(\|\mudif_\ddag\|^2)$.
We then notice that
\beqn
\|\mudif_\ddag\|^2 - \frac1{\|\mudif_\ddag\|^2} \sum_j \mudif_{\ddag j}^4 
&=& \frac{\sum_{j \ne k} \mudif_{\ddag j}^2\mudif_{\ddag k}^2}{\sum_j \mudif_{\ddag j}^2} \\ 
&=& \sum_{k} \mudif_{\ddag k}^2 (1 - \kappa_k^2)  \\
&\ge& \|\mudif_\ddag\|^2 (1 - \kappa^2) \ ,
\eeqn
where $\kappa_k = |\mudif_{\ddag k}|/\|\mudif_\ddag\|$ and $\kappa = \max_k \kappa_k$ by definition.  
Hence, when \eqref{upper_diagonal_sparse} holds, we have 
\[\frac{v^\top \hat\bSigma_\ddag v}{v^\top \diag(\hat\bSigma_\ddag) v} \ge 1 + \nu(1-\nu)(1-\kappa^2) \frac{\|\mudif_\ddag\|^2}{1+\|\mudif_\ddag\|^2}  + \ o_P\left(\frac{\|\mudif_\ddag\|^2}{1+ \|\mudif_\ddag\|^2 }\right) + O_P\left(\frac1{\sqrt{n}}\right) \ ,\]
with $\kappa$ assumed to be fixed.  

Comparing this bound with the control of $V$ under the null in \eqref{V-null-diag}, we conclude.

The proof for on the consistency of variable selection is parallel to the one detailed in \secref{proof-selection-known} and details are omitted.

\subsection{Proof of \prpref{coordinate_unknown}}

The proof of the detection part is a straightforward adaptation of that of Propositions \ref{prp:malk3} and \ref{prp:malk5}, and is omitted. 
We therefore turn to the variable selection part, and focus on proving that $\hat J_1$ is consistent for $J := {\rm supp}(\mudif)$ under the stated conditions.  The same arguments apply to proving the consistency of $\hat{J}_2$ and details are therefore omitted.

We first observe that the number of false positives goes to zero in probability.  Indeed, 
\[
\E |\hat{J}_1\setminus J| = \sum_{j \notin J} \P\big[j \in \hat J_1\big] \le p \P_0\big[T_{1,j} > q_1^{-1}(\alpha/p)\big] = \alpha = o(1)\ .
\]
And, by Markov's inequality, this implies that that $|\hat{J}_1\setminus J| \to 0$ in probability.

It remains to prove that $|J \setminus \hat{J}_1| = o_{P}(|J|)$. 
First, take $j \notin J$ to bound the quantile function $q_1^{-1}(.)$.  Using the arguments in \secref{proof-malk3-h0} for the case $s=1$, or directly using concentration bounds for the numerator and denominator defining $T_1^{(j)}$, we can see that 
$$\P\Big[T_1^{(j)} > \frac{n \sqrt{2/\pi} + \sqrt{n} t}{\sqrt{1 - t/\sqrt{n}}}\Big] \le C e^{-t^2/C} \ ,$$
for some universal constant $C > 0$.
Hence, 
\[
b \ge C e^{-t^2/C} \Rightarrow q_1^{-1}(b) \le \frac{n \sqrt{2/\pi} + \sqrt{n} t}{\sqrt{1 - t/\sqrt{n}}}\ .
\]
In particular, when $\alpha \ge p^{-a}$ for some $a > 0$ fixed, 
\[q_1^{-1}(\alpha/p) \le q_1^{-1}(p^{-a-1}) \le n \sqrt{2/\pi} + C n \sqrt{\frac{\log(p)}{n}}\ , \]
for a different constant $C>0$.

Since the effective dynamic range of $\mudif$ and the $s$-sparse Riesz constant of $\bSigma$ are both bounded, Condition \eqref{upper_test_sparse1-select_malk3-1} implies that
$$\min_{j \in J} \frac{(\mudif_j)^2}{ \sigma_{jj}} \gg \left[\frac{\log(p)}{n}\right]^{1/4}\ .$$ 
Then exactly as in \secref{proof-malk3-h1}, in the simplest case where $s=1$, we find that, for any $j \in J$, $\frac1n T_1^{(j)}- \sqrt{2/\pi} \gg \sqrt{\frac{\log(p)}{n}}$, where the inequality follows from assumption on $\nu$.  (Recall that $\nu$ is fixed.)  This implies that $\P[j\notin \hat{J}_1] = o(1)$ uniformly in $j \in J$.  
Then
\[
\E |J \setminus \hat{J}_1| = \sum_{j \in J} \P\big[j \notin \hat J_1\big] = o(|J|)\ .
\]
And we conclude by Markov's inequality.

\subsection{Proof of deviation inequalities}\label{sec:proofs_deviation}

\subsubsection{Proof of Lemma \ref{lem:deviation_polynom4}}
The first bound is a consequence of Chebyshev's inequality. Thus, we focus on \eqref{eq:deviation_difference_4}. 
\beqn
 Q^\circ_4(v)-Q^\circ_4(w)= \frac{1}{2}\left[\sum_{i=1}^n ( (v-w)^{\top}Z_i)^{3} ( (v+w)^{\top} Z_i)+ ( (v-w)^{\top}Z_i) ( (v+w)^{\top}Z_i)^3 \right]\ .
\eeqn 
Since $v$ and $w$ are unit vectors, $ (v-w)^{\top}Z_i$ and $( v+w)^{\top}Z_i$ are independent. Define $A_i$, $i=1,\ldots n$ and $B_i$, $i=1,\ldots,n$ $2n$ independent standard normal variables. Then, $2(Q^\circ_4(v)-Q^\circ_4(w))$ follows the same distribution as $\|v-w\|^3\|v+w\| \sum_{i=1}^n A^3_i B_i + \|v-w\|\|v+w\|^3 \sum_{i=1}^n A_i B_i^3$. Consequently, the proofs boil down to controlling random variables of the form $T:=\sum_{i=1}^n A_iB_i^3$. Conditionally to $A_i$, $T$ is a weighted Gaussian chaos of degree 3. Applying Lemma \ref{lem:deviation_polynom3}, we obtain
\[\P\big[|T|\geq C_1\|A\|(\sqrt{x}+1) + C_2 \|A\|_{\infty}\sqrt{\log(n)}x+ C_3\|A\|_{\infty} x^{3/2}\big]\leq 2e^{-x}\ ,\]
for any $x>0$. Since $\|A\|^2$ follows a $\chi^2$ distribution and $\|A\|_{\infty}$ is a supremum of $n$ Gaussian variables, we use Lemma \ref{lem:birge_chi2} to obtain
\[\P\Big[\|A\|^2\geq n + 2\sqrt{nx}+2x\Big]\leq e^{-x}\quad \quad  \P\Big[\|A\|_{\infty}\geq \sqrt{2(\log(2n)+x)}\Big]\leq e^{-x}\ .\]
Reorganizing all the terms leads to 
\[\P\big[|T|\geq C_1\sqrt{n(x\vee 1)}+ C_2 x^{2}\big]\leq 4e^{-x}\ .\]
Then, \eqref{eq:deviation_difference_4} follows from the inequality $\|v+w\|\leq 2$ (since $v$ and $w$ are unit vectors).

\subsubsection{Proof of Lemma \ref{lem:process_subgaussian}}
Fix $v, w \in \cV$.  For any $\lambda > 0$,
\[
\E\left[e^{\lambda(Q^\circ_1(v)-Q^\circ_1(w))}\right] = \left(\E\left[e^{\lambda(|v^\top Z| - |w^\top Z|)}\right]\right)^n\ ,
\]
where $Z$ is standard normal.  Decompose $w$ into $w = a v + \sqrt{1-a^2} w_0$ where $a \in (-1,1)$ and $w_0 \perp v$, and define $X = v^\top Z$ and $Y = w_0^\top Z$.  Notice that $X$ and $Y$ are iid standard normal variables.  We then apply \lemref{deviation_lambda} below to get
\beqn
\E\left[e^{\lambda(Q^\circ_1(v)-Q^\circ_1(w))}\right]\leq \exp\left[2n \lambda^2 \|v-w\|^2\right]\ , 
\eeqn 
and conclude the proof of \lemref{process_subgaussian} by a simple application of  Chernoff's bound.

\begin{lem}\label{lem:deviation_lambda}
Consider $X$ and $Y$ two independent standard normal variables. For any $a \in (-1,1)$ and $\lambda>0$, we have
\beqn
\E\left[e^{\lambda (|X| - |aX+\sqrt{1-a^2}Y|)}\right]\leq e^{4\lambda^2(1-a)}\ . 
\eeqn 
\end{lem}

\begin{proof}
Let $b = \sqrt{1-a^2}$.  Using a power series expansion, we have 
\beqn
\E\left[e^{\lambda (|X| - |aX+bY|)}\right] 
&=&  1 +\sum_{k=2}^{\infty} \frac{\lambda^k}{k!}\E\left[ (|X| - |aX+bY|)^k\right]\ ,
\eeqn 
since $|X|$ and $|aX+bY|$ have the same expectation. 
Define $\delta = \sqrt{(1-a)^2+b^2} = \sqrt{2 -2a}$.  For the term of order $k$, we get
\beqn
\E\left[ (|X| - |aX+bY|)^k\right]&\leq& \E\left[|(1-a)X+bY|^k\right] = \delta^{k}\E[|X|^k]\\
& &= \delta^{k} \frac{2^{k/2}\Gamma(\frac{k+1}{2})}{\sqrt{\pi}}\\
&\leq& \delta^{k} \left[\left(\frac{k-1}{\sqrt{2}}\E[|X|^{k-1}]\right)\wedge \E[|X|^{k+1}]\right] \ . 
\eeqn 
As a consequence, 
\beqn
\frac{\lambda^{2k+1}\delta^{2k+1}\E\left[ |X|^{2k+1}\right]}{(2k+1)!}\leq \lambda^{2k}\delta^{2k}\frac{\E[X^{2k}]}{2k!}+ \lambda^{2k+2}\delta^{2k+2}(2k+2)\frac{\E[X^{2k+2}]}{(2k+2)!}\ .
\eeqn 
Coming back to the exponential moment, we obtain 
\beqn
\E\left[e^{\lambda (|X| - |aX+bY|)}\right]&\leq & 1+ \sum_{k=2}^{\infty} \frac{\lambda^k}{k!}\delta^{k}\E[|X|^k]
\\
&\leq & 1 + \sum_{k=1}^{\infty }\frac{\lambda^{2k}}{(2k)!}\delta^{2k}\E[X^{2k}]+ \sum_{k=1}^{\infty }\frac{\lambda^{2k+1}}{(2k+1)!}\delta^{2k+1}\E[|X|^{2k+1}]\\ &\leq & 1+
\sum_{k=1}^{\infty} \frac{\lambda^{2k}}{(2k)!}\delta^{2k}(2k+2)\E\left[ X^{2k}\right]\\
&\leq & 1 +\sum_{k=1}^{\infty} (2\lambda\delta)^{2k}\frac{\E\left[ X^{2k}\right]}{(2k)!}= \E\left[e^{2\lambda\delta X}\right]\\
&\leq & \exp\left[ 2\lambda^2 \delta^2 \right]\ .
\eeqn 
\end{proof}

\subsubsection{Proof of Lemma \ref{lem:process_subgaussian2}}
Denote $\Upsilon_{2s}$ the  $2s$-sparse Riesz constant of $\bSigma$. 
We prove the subgaussian deviation bounds using the Laplace transform. 
For $\lambda>0$, we have
\[\E_\eta\Big[\exp\left\{\lambda (R_{u_1,\eta} - R_{u_2,\eta} ) \right\} \Big] =\prod_{i=1}^n\E_{\eta_i} \left[e^{\lambda T_i}\right]\ ,\]
where 
\beqn
T_i&:=&  |v_{u_1}^\top Z_i- (\eta_i-\nu)t_{v_{u_2}}| - |{v_{u_2}}^\top Z_i-(\eta_i-\nu)t_{u_2}| - \E_{\eta_i}\left[|v_{u_1}^\top Z_i- (\eta_i-\nu)t_{u_1}|\right]\\&& +  \E_{\eta_i}\left[|v_{u_2}^\top Z_i- (\eta_i-\nu)t_{u_2}|\right]\ . 
\eeqn
Using the Taylor expansion of the exponential function, we get
$\E_{\eta_i}\left[e^{\lambda T_i}\right] 
=  1 +\sum_{k=2}^{\infty} \frac{\lambda^k}{k!}\E_{\eta_i}\left[ T_i^k\right]\ .$
Note that  \[\E_{\eta_i}\left[|v_{u_1}^\top Z_i- (\eta_i-\nu)t_{u_1}|\right]-   \E_{\eta_i}\left[|v_{u_2}^\top Z_i- (\eta_i-\nu)t_{u_2}|\right]|\leq  |\eta_i-\nu||t_{u_2}-t_{u_1}|\ .\]
Since $u_1$ and $u_2$ are in $\cU_{\delta}$, $\|v_{u_1}-v_{u_2}\|\leq 2\Upsilon^{1/2}_{2s}\|u_1-u_2\|$. Similarly, $|t_{u_1}-t_{u_2}|\leq 2\Upsilon_{2s}\|u_1-u_2\|t_{\star}$. Combining these bounds, we obtain
\beqn
\E_{\eta_i}\left[ T_i^k\right]&\leq& \E_{\eta_i}\left[ |T_i|^k\right]\\
& \leq&\E_{\eta_i}\left[\left(\big|(v_{u_1}-v_{u_2})^{\top}Z_i\big| + 2|\eta_i-\nu||t_{u_1}-t_{u_2}| \right)^k \right]\\
&\leq & 2^{2k-1}|\eta_i-\nu|^k|t_{u_1}-t_{u_2}|^k+ 2^{k-1}\E\left[\big|(v_{u_1}-v_{u_2})^{\top}Z_i\big|^k\right]\\
&\leq & 2^{3k-1}\Upsilon^{k}_{2s} \|u_1-u_2\|^{k}t_{\star}^k+ 2^{2k-1}\Upsilon^{k/2}_{2s}\E\left[\big|(u_1-u_2)^{\top}Z_i\big|^k\right]\ .
\eeqn
We have already bounded $\E\big[\big|(u_1-u_2)^{\top}Z_i\big|^k\big]$ in the proof of Lemma \ref{lem:deviation_lambda}. Arguing as in this last proof, we upper bound each term of the Taylor expansion to get
\beqn
\E_{\eta_i}\left[e^{\lambda T_i}\right]&\leq& -1 + \exp\left[C\lambda^2\Upsilon^2_{2s}\|u_1-u_2\|^2t_{\star}^2\right]+ \exp\left[C\lambda^2\Upsilon_{2s}\|u_1-u_2\|^2\right]\\
&\leq & \exp\left[C\lambda^2 \|u_1-u_2\|^2\right]\ ,
\eeqn
since we assume that $t_{\star}=o(1)$ and $\Upsilon_{2s}=O(1)$.
We conclude that

\[
\E_\eta\Big[\exp\left\{\lambda (R_{u_1,\eta} - R_{u_2,\eta} ) \right\} \Big] 
\le \exp\left[C n \lambda^2 \|u_1-u_2\|^2 \right]\ .
\]
Since this is true uniformly over $\eta$, it is true unconditionally, and an application of Chernoff's bound yields the desired result.

\subsubsection{Proof of Lemma \ref{lem:deviation_polynom3}}

{\em Proof of  \eqref{eq:deviation_moment_3}.}
We rely on the concentration bounds developed by \cite{MR2123200} for Rademacher chaoses. 
Here $n$ is fixed and we define a sequence of iid Rademacher random variables $(Y_{i,k} : i =1, \dots, n; k \ge 1)$. Then, as $N \to \infty$,
\[
T := \sum_{i=1}^n a_i \Big(\sum_{j=1}^N \frac{Y_{i,k}}{N^{1/2}}\Big)^3 
\]
converges in distribution towards $T^* := \sum_{i=1}^n a_i( v^{\top}Z_i)^3$, and any moment of $T$ converges to that of $T^*$.
Developing $T$, we get $T = T_1 + T_2 + T_3$, where 
\beqn
T_1 &:=& \frac6{N^{3/2}} \sum_{i=1}^n a_i\ \sum_{1 \le j_1 < j_2 < j_3 \le N}Y_{i,j_1}Y_{i,j_2}Y_{i,j_3} \ , \\ 
T_2 &:=& \frac3{N^{3/2}} \sum_{i=1}^n a_i\ \sum_{1 \le j \ne k \le N} Y_{i,j} Y_{i,k}^2 = \frac{3(N-1)}{N^{3/2}} \sum_{i=1}^na_i \sum_{1 \le j \le N} Y_{i,j} \ , \\
T_3 &:=& \frac1{N^{3/2}} \sum_{i=1}^n a_i\sum_{1 \le j \le N} Y_{i,j}^3 = \frac1{N^{3/2}} \sum_{i=1}^n a_i\sum_{1 \le j \le N} Y_{i,j} \ .
\eeqn
Recall that $n$ is fixed here, while $N \to \infty$.  
Hence, by Chebyshev's inequality, $T_3 = o_P(1)$, that is, converges to zero in probability as $N \to \infty$.
Consequently, for any fixed $x_1,x_2>0$, 
\beq\label{eq:devi3}
\P\left[|T^*|\geq x_1+ x_2\right]\leq \limsup \P\left[|T_1|\geq x_1\right]+ \limsup  \P\left[|T_2|\geq x_2\right]\ ,
\eeq
where the limit superior is w.r.t.~$N\to\infty$. 
Observe that $T_2$ converges in distribution towards $\cN(0, 9\|a\|^2)$.
Hence, $\lim\sup  \P\left[|T_2|\geq x_2\right]\leq 2e^{-x_2^2/(18\|a\|^2)}$. We focus on the deviations of $T_1$, which is a Rademacher chaos of order $3$. First, by Cauchy-Schwarz inequality, $\E(|T_1|)\leq \E^{1/2}(T_1^2)\leq  C\|a\|$.  For any positive integers $k$ and $\ell$, let $\bbB_{k\times\ell}$ denote the unit ball in $\mathbb{R}^{k \times \ell}$ for the Euclidean metric. 
By \citep[Cor.~4]{MR2123200}, for any $t> 0$, 
\beq\label{eq:devi3-b}
\P\big[|T_1|\geq \E(|T_1|) + t \big]\leq \exp\left[-C \bigwedge_{\ell=1}^3 \left(\frac{t}{\E[W_\ell]}\right)^{2/\ell}\right]\ ,
\eeq
where 
\beqn
W_1&= &\sup_{\alpha\in \bbB_{n \times N}}3\Big|\sum_{i=1}^n a_i\sum_{j_1} \sum_{j_2 \ne j_1} \sum_{j_3\notin \{j_1,j_2\}}\frac{Y_{i,j_1}Y_{i,j_2}\alpha_{i,j_3}}{N^{3/2}}\Big| \ , \\
W_2&= &\sup_{\alpha^{(1)},\alpha^{(2)} \in \bbB_{n \times N}}3\Big|\sum_{i=1}^n a_i\sum_{j_1} \sum_{j_2 \ne j_1} \sum_{j_3 \notin \{j_1,j_2\}} \frac{Y_{i,j_1}\alpha^{(1)}_{i,j_2}\alpha^{(2)}_{i,j_3}}{N^{3/2}}\Big| \ , \\
W_3&=& \sup_{\alpha^{(1)},\alpha^{(2)},\alpha^{(3)} \in \bbB_{n \times N}}\Big|\sum_{i=1}^n  a_i \sum_{j_1} \sum_{j_2 \ne j_1} \sum_{j_3\notin \{j_1,j_2\}} \frac{\alpha^{(1)}_{i,j_1}\alpha^{(2)}_{i,j_2}\alpha^{(3)}_{i,j_3} }{N^{3/2}}\Big| \ .
\eeqn
We now bound the expectation of these three random variables. 
For $W_1$, we have 
\[
W_1 \le 3 V_1 + 3 U_1 \ ,
\]
where
\[
V_1 := \sup_{\alpha\in \bbB_{n \times N}} \Big|\sum_{i=1}^n a_i \sum_{j_1}\sum_{j_2\neq j_1}\frac{Y_{i,j_1}Y_{i,j_2}}{N^{3/2}}\sum_{j_3=1}^N\alpha_{i,j_3}\Big | \ ,
\]
and 
\[
U_1 := \sup_{\alpha\in \bbB_{n \times N}}\Big|\sum_{i=1}^n a_i \sum_{j_1}\sum_{j_2\neq j_1}\frac{Y_{i,j_1}Y_{i,j_2}(\alpha_{i,j_1}+\alpha_{i,j_2}) }{N^{3/2}}\Big| \ .
\]
Applying the triangle and Cauchy-Schwarz inequalities, we have 
\[U_1 \le 2 \sum_{i=1}^n |a_i| H_i^{1/2}, \quad \text{where} \quad H_i := \sum_{j_1} \frac{Y_{i,j_1}^2}{N^{3}}  \Big(\sum_{j_2\neq j_1}Y_{i,j_2}\Big)^2 = N^{-3} \sum_{j_1} \Big(\sum_{j_2\neq j_1}Y_{i,j_2}\Big)^2\ .\]
Note that $0 \le H_i \le 1$ and $\E(H_i) = (N-1)/N^2 = o(1)$, so that $\E(U_1^m) = o(1)$ for any fixed $m \ge 1$.  
\beqn
\eeqn
For $V_1$, letting $\tilde\alpha_i = \frac1{\sqrt{N}} \sum_j \alpha_{i,j}$, and realizing that $\alpha \in \bbB_{n \times N}$ implies that $\tilde\alpha \in \bbB_n := \bbB_{n \times 1}$ by the Cauchy-Schwarz inequality, we have
\beqn
V_1
&\leq & \sup_{\tilde\alpha\in \bbB_n} \Big|\sum_{i=1}^n a_i\tilde\alpha_i\sum_{j_1}\sum_{j_2\neq j_1}\frac{Y_{i,j_1}Y_{i,j_2}}{N}\Big| 
=  \left[\sum_{i=1}^n a_i^2 \Big(\sum_{j_1}\sum_{j_2\neq j_1}\frac{Y_{i,j_1}Y_{i,j_2}}{N}\Big)^2 \right]^{1/2} 
\eeqn
where we applied Cauchy-Schwarz inequality. The last term has  second moment 
 bounded by $C \|a\|^2$, and therefore first moment bounded by $C \|a\|$ by the Cauchy-Schwarz inequality.
We conclude that $\lim\sup\E[W_1]\leq C\|a\|$.  

We proceed similarly for $W_2$, starting from
\[
W_2 \le  3V_2 + 3 U_2\ ,
\]
where
\[
V_2 := \sup_{\alpha^{(1)},\alpha^{(2)} \in \bbB_{n \times N}}  \Big|\sum_{i=1}^n a_i \sum_{ j_1,j_2,j_3} \frac{Y_{i,j_1}\alpha^{(1)}_{i,j_2}\alpha^{(2)}_{i,j_3}}{N^{3/2}}\Big| \ ,
\]
and
\[
U_2 := \sup_{\alpha^{(1)},\alpha^{(2)} \in \bbB_{n \times N}}2\Big|\sum_{i=1}^n a_i\sum_{j_1} \sum_{ j_2 \neq j_1} \frac{Y_{i,j_1}\alpha^{(1)}_{i,j_1}\alpha^{(2)}_{i,j_2}}{N^{3/2}}\Big|+ \Big|\sum_{i=1}^n a_i\sum_{j_1}\sum_{j_2}  \frac{Y_{i,j_1}\alpha^{(1)}_{i,j_2}\alpha^{(2)}_{i,j_2}}{N^{3/2}}\Big|  \ .
\]
By the triangle inequality, and Cauchy-Schwarz inequality multiple times, 
\beqn
\Big|\sum_{i=1}^n a_i\sum_{j_1} \sum_{ j_2 \neq j_1} \frac{Y_{i,j_1}\alpha^{(1)}_{i,j_1}\alpha^{(2)}_{i,j_2}}{N^{3/2}}\Big| 
&\le& N^{-3/2} \sum_{i=1}^n |a_i| \Big[\sum_{j_1} Y_{i,j_1}^2 {\alpha^{(1)}_{i,j_1}}^2\Big]^{1/2} \Big[\sum_{j_1} \Big(\sum_{j_2 \ne j_1} \alpha^{(2)}_{i,j_2}\Big)^2\Big]^{1/2} \\
&\le& N^{-1} \sum_{i=1}^n |a_i|=o(1)\ ,
\eeqn
using the fact that $\alpha^{(1)},\alpha^{(2)} \in \bbB_{n \times N}$. Similarly,
\[\Big|\sum_{i=1}^n a_i\sum_{j_1}\sum_{j_2}  \frac{Y_{i,j_1}\alpha^{(1)}_{i,j_2}\alpha^{(2)}_{i,j_2}}{N^{3/2}}\Big|\leq \Big|\sum_{i=1}^n a_i\sum_{j_2}  \frac{|\alpha^{(1)}_{i,j_2}\alpha^{(2)}_{i,j_2}|}{N^{1/2}}\Big|\leq N^{-1/2}\sum_{i=1}^n |a_i|=o(1)\]
For $V_2$, as for $V_1$ (and using similar notation), we have  
\beqn
V_2
&\leq & \sup_{\tilde\alpha^{(1)},\tilde\alpha^{(2)} \in \mathbb{B}_n} \Big|\sum_{i=1}^n  a_i \tilde\alpha^{(1)}_i \tilde\alpha^{(2)}_i H_i\Big|  \\
&&= \sup_{\tilde\alpha^{(1)} \in \mathbb{B}_n} \Big[\sum_{i=1}^n a_i^2 H_i^2 (\tilde\alpha_i^{(1)})^2 \Big]^{1/2} =  \bigvee_{i=1}^{n} \big|a_i H_i\big|\ ,
\eeqn
where $H_i := \frac1{\sqrt{N}} \sum_j Y_{i,j}$, and we used the Cauchy-Schwarz inequality multiple times.
By the convergence of moments in the central limit theorem \citep[Th.~1]{bahr}, $W_2$ is upper bounded by a variable converging in moment to $\|a\|_{\infty}$ times  the supremum of $n$ independent standard normal distributions. We obtain $\lim\sup \E[W_2]\leq C\|a\|_{\infty}\sqrt{\log(n)}$. 

We work on $W_3$ in a similarly way.  Applying the Cauchy-Schwarz inequality multiple times, and reasoning as we did before, we have
\beqn
W_3&\leq &\sup_{\alpha^{(1)},\alpha^{(2)},\alpha^{(3)}\in \bbB_{n \times N}}\Big|\sum_{i=1}^n a_i \sum_{j_1,j_2,j_3} \frac{\alpha^{(1)}_{i,j_1}\alpha^{(2)}_{i,j_2}\alpha^{(3)}_{i,j_3} }{N^{3/2}}\Big|+o(1)\\
&\leq& \sup_{\tilde\alpha^{(1)},\tilde\alpha^{(2)},\tilde\alpha^{(3)}\in \mathbb{B}_n}\Big|\sum_{i=1}^n a_i \tilde\alpha^{(1)}_{i}\tilde\alpha^{(2)}_{i}\tilde\alpha^{(3)}_{i} \Big|+o(1)\\ 
&\leq& \sup_{\alpha^{(1)},\alpha^{(2)} \in \mathbb{B}_n}\Big[\sum_{i=1}^n  \left(a_i\alpha^{(1)}_{i}\alpha^{(2)}_{i}\right)^2\Big]^{1/2}+o(1)\leq \|a\|_{\infty}+o(1) \ .\\
\eeqn
We conclude the proof of \eqref{eq:deviation_moment_3} by combining \eqref{eq:devi3} and \eqref{eq:devi3-b} with the above bound. 

\bigskip\noindent
{\em Proof of  \eqref{eq:deviation_moment_3-b}.}
This simply follows from the observation that
$\sqrt{\log(n)}x\leq \sqrt{nx}+ x^{3/2}$ is valid for all $x \ge 0$ and all $n \ge 1$.

\bigskip\noindent
{\em Proof of  \eqref{eq:deviation_difference_3}.}
Fix two any unit vectors $v$ and $w$. 
We have
\beq\label{eq:decomposition_difference_3}
Q_3^\circ(v)- Q_3^\circ(w)=  \frac{1}{4} U_1 +\frac{3}{4} U_2 \ ,
\eeq
where
\[
U_1 := \sum_{i=1}^n [(v-w)^\top Z_i]^3 \ , \quad  U_2 := \sum_{i=1}^n [(v-w)^\top Z_i][(v+w)^\top Z_i]^2 \ .
\]
As a consequence, $U_1/\|v-w\|^3$ follows the same distribution as $Q_3^\circ(v)$ and we can control its deviations using \eqref{eq:deviation_moment_3}. Since $(v-w)^\top Z_i$ is independent of $(v+w)^\top Z_i$, $U_2/[\|v-w\|\|v+w\|^2]$ follows the same distribution as 
\beqn
U:= \sum_{i=1}^n A_i B_i^2\ ,
\eeqn
where $A_1, \dots, A_n, B_1, \dots, B_n$ are iid standard normal. Observe that $U$ is a quadratic function with respect to $B$. 
Relying on a straightforward extension of \citep[Lem.~1]{MR1805785} that provides a deviation bound for non-necessarily positive quadratic forms of Gaussian random variables, we have 
\[\P\left[U \geq \sum_{i=1}^n A_i+ 2 \Big(x \sum_{i}^nA_i^2\Big)^{1/2}+ 2 x \|A\|_{\infty} \right]\leq e^{-x}\]
for any $x>0$. Classical deviation inequalities for Gaussian distributions, $\chi^2$ distributions and suprema of Gaussian vectors lead to 
\[
\P\Big[\sum_{i=1}^n A_i\geq \sqrt{2nx}\Big] \leq e^{-x}\ , \quad \P\Big[\sum_{i}^nA_i^2\geq n+2\sqrt{nx}+ 2x\Big]\leq e^{-x}
\]
and
\[
\P\Big[\|A\|_{\infty}\geq \sqrt{2(\log(2n)+x)}\Big] \leq e^{-x}\ , \quad \forall x > 0.
\]
Gathering these four deviation inequalities, we obtain
\[
 \P\left[U \geq C_1\sqrt{nx}+ C_2\sqrt{\log(n)}x \right]\leq 4 e^{-x}\ , \quad \forall x > 0.
\]
Since $v$ and $w$ are unit vector, $\|v+w\|\leq 2$. Coming back to \eqref{eq:decomposition_difference_3} and combining the previous deviation inequality with \eqref{eq:deviation_moment_3}, we obtain
\[
\P\left[Q_3^\circ(v)- Q_3^\circ(w)\geq\|v - w\| \left(C_1\sqrt{n}(\sqrt{x}+1)+ C_2 \sqrt{\log(n)} x + C_3 x^{3/2}\right)\right] \leq 6e^{-x} \ ,\]
for any $x>0$.

\subsubsection{Proof of Lemma \ref{lem:control_Q2}}

Again, we apply Laplace method. Let $X$ be a standard normal variable.  For any $\lambda\in (-1/2,1/2)$,  $$\E[\exp(\lambda X^2\sign(X))]=  \frac12 (1-2\lambda)^{-1/2}+ \frac12(1+2\lambda)^{-1/2} \ .$$ 
Computing the Taylor expansion of this  expression leads to 
\beqn
\log\left(\E[\exp(\lambda X^2\sign(X))]\right)&=& \log\left(1+ \sum_{k=1}^{\infty} \frac{\lambda^{2k} (4k-1)!!}{(2k)!} \right)\\
&\leq & \sum_{k=1}^{\infty} \frac{\lambda^{2k} (4k-1)!!}{(2k)!}\\
&\leq & \sum_{k=1}^{\infty} 2^{2k-1}\lambda^{2k}\leq \frac{2\lambda^{2}}{1-2\lambda}\ ,
\eeqn
where we compare the power series in the last line. Thus, for any $\lambda\in (0,1/2)$, 
\beqn
\log(\E[\exp[\lambda Q_2^{\rm sign,\circ}(v)]])\leq \frac{2n \lambda^{2}}{1-2\lambda}\ . 
\eeqn
We now refer the reader to \citep{1998:birge}, where it is proved that such a bound implies that, for any $t>0$, 
\[\P\left[Q_2^{\rm sign,\circ}(v)\geq \sqrt{8n t}+ 2 t \right]\leq e^{-t}\ . \]

Consider $X$ and $Y$ two independent  standard normal variables. Let $a\in [0,1)$ and $b=\sqrt{1-a^2}$. 
We compute the generating function of $Z:=X^2\sign(X)-(aX+bY)^2\sign(aX+bY)$. 
Using a  power series expansion, we get
\beqn
\E\left[e^{\lambda Z}\right] 
&=&  1 +\sum_{k=1}^{\infty} \frac{\lambda^{2k}}{(2k)!}\E\left[ (X^2\sign(X) - (aX+bY)^2\sign(aX+bY))^{2k}\right]\ ,
\eeqn 
by symmetry about 0. Decompose the $2k$-th moment into a sum of two terms
\beqn
\lefteqn{\E\left[ (X^2\sign(X) - (aX+bY)^2\sign(aX+bY))^{2k}\right]}\\ &\leq &\E\left[ (X^2 - (aX+bY)^2)^{2k}\right]+ \E\left[(X^2 + (aX+bY)^{2})^{2k}\IND{\sign(X)\neq \sign(aX+bY)}\right]\\
&& =: A_1+ A_2\ . 
\eeqn
Since $X^2-(aX+bY)^2= [(1-a)X-bY][(1+a)X+bY]$ is the product of two independent normal variables with zero mean and variance $(1-a)^2+b^2=2(1-a)$ and $2(1+a)$, respectively, we obtain
\beqn
A_1 = (1-a^2)^{k}2^{2k}[(2k-1)!!]^2\ . 
\eeqn 
Turning to $A_2$, when $\sign(X)\neq \sign(aX+bY)$ we have 
\[X^2 + (aX+bY)^{2} = (1-a^2) X^2 + b^2 Y^2 + 2 aX (aX + bY) \le (1-a^2) X^2 + b^2 Y^2 = (1-a^2)(X^2 +Y^2)\ . \]
Hence, we obtain
\beqn
A_2&\leq& \E\left[((1-a^2)(X^2+  Y^2))^{2k}\right] = (1-a^2)^{2k} (4 k)!! \leq (1-a^2)^{k} 2^{2k} (2k)! \ ,
\eeqn
using the fact that $X^2+  Y^2 \sim \chi^2_2$ and the moments of this distribution. 
Coming back to the moment generating function, we get 
\beqn
\log\left(\E\left[e^{\lambda Z}\right]\right) &\leq & \log\left[1+ \sum_{k=1}^{\infty}\left(\sqrt{1-a^2}\lambda\right)^{2k} \frac{2^{2k}[(2k-1)!!]^2+ 2^{2k}(2k)!}{(2k)!}\right]\\
&\leq & \log\left[1+ \sum_{k=1}^{\infty}2\left(2\sqrt{1-a^2}\lambda\right)^{2k}\right] 
\\
&\leq &\sum_{k=1}^{\infty}2\left(2\sqrt{1-a^2}\lambda\right)^{2k} \\
&\leq &  \frac{8(1-a^2)\lambda^2}{1-2\sqrt{1-a^2}\lambda}\ ,
\eeqn
if $0<\lambda<(2\sqrt{1-a^2})^{-1}$.
Applying this bound to $Q_2^{\rm sign,\circ}(v)- Q_2^{\rm sign,\circ}(w)$, with $a=v^{\top} w$ assumed to be nonnegative, yields 
\beqn
\log\left(\E\left[e^{\lambda (Q_2^{\rm sign,\circ}(v)- Q_2^{\rm sign,\circ}(w))}\right]\right) &\leq &\frac{8n(1-a^2)\lambda^2}{1-2\sqrt{1-a^2}\lambda}\\
&\leq & \frac{8n\|v-w\|^2 \lambda^2 }{1-2\|v-w\|\lambda}\ , 
\eeqn 
using the fact that $\|v\| = \|w\| = 1$.
We use \citep{1998:birge} again, where this bound is shown to entail that, for any $t>0$, 
\[\P\left[Q_2^{\rm sign,\circ}(v)- Q_2^{\rm sign,\circ}(w)\geq 6\|v-w\| \sqrt{n t}+ 2\|v-w\| t \right]\leq e^{-t}\ . \]

\subsection*{Acknowledgements}
We thank two anonymous referees for suggestions that greatly improved the presentation.
This work was partially supported by the US Office of Naval Research (N00014-13-1-0257) and the French Agence Nationale de la Recherche (ANR 2011 BS01 010 01 projet Calibration).

\bibliographystyle{chicago}
\bibliography{ref}

\begin{thebibliography}{}

\bibitem[\protect\citeauthoryear{Akaike}{Akaike}{1974}]{MR0423716}
Akaike, H. (1974).
\newblock A new look at the statistical model identification.
\newblock {\em IEEE Trans. Automatic Control\/}~{\em AC-19}, 716--723.
\newblock System identification and time-series analysis.

\bibitem[\protect\citeauthoryear{Aldous}{Aldous}{1985}]{aldous85}
Aldous, D.~J. (1985).
\newblock {\em \textit{Exchangeability and related topics}, \'Ecole d'\'et\'e
  de probabilit\'es de Saint Flour XIII}, Volume~{\bf 1117} of {\em Lecture
  Notes in Mathematics}.
\newblock Berlin: Springer-Verlag.

\bibitem[\protect\citeauthoryear{Amini and Wainwright}{Amini and
  Wainwright}{2009}]{amini2009high}
Amini, A.~A. and M.~J. Wainwright (2009).
\newblock High-dimensional analysis of semidefinite relaxations for sparse
  principal components.
\newblock {\em The Annals of Statistics\/}~{\em 37\/}(5B), 2877--2921.

\bibitem[\protect\citeauthoryear{Azizyan, Singh, and Wasserman}{Azizyan
  et~al.}{2013}]{azizyan13}
Azizyan, M., A.~Singh, and L.~Wasserman (2013).
\newblock Minimax theory for high-dimensional gaussian mixtures with sparse
  mean separation.
\newblock {\em Neural Information Processing Systems (NIPS)\/}.

\bibitem[\protect\citeauthoryear{Azizyan, Singh, and Wasserman}{Azizyan
  et~al.}{2014}]{azizyan2014efficient}
Azizyan, M., A.~Singh, and L.~Wasserman (2014).
\newblock Efficient sparse clustering of high-dimensional non-spherical
  gaussian mixtures.
\newblock {\em arXiv preprint arXiv:1406.2206\/}.

\bibitem[\protect\citeauthoryear{Belkin and Sinha}{Belkin and
  Sinha}{2010}]{belkin2010polynomial}
Belkin, M. and K.~Sinha (2010).
\newblock Polynomial learning of distribution families.
\newblock In {\em Foundations of Computer Science (FOCS), 2010 51st Annual IEEE
  Symposium on}, pp.\  103--112. IEEE.

\bibitem[\protect\citeauthoryear{Berthet and Rigollet}{Berthet and
  Rigollet}{2013a}]{berthet2}
Berthet, Q. and P.~Rigollet (2013a).
\newblock Complexity theoretic lower bounds for sparse principal component
  detection.
\newblock In {\em Conference on Learning Theory (COLT)}, pp.\  1046--1066.

\bibitem[\protect\citeauthoryear{Berthet and Rigollet}{Berthet and
  Rigollet}{2013b}]{berthet}
Berthet, Q. and P.~Rigollet (2013b).
\newblock Optimal detection of sparse principal components in high dimension.
\newblock {\em The Annals of Statistics\/}~{\em 41\/}(4), 1780--1815.

\bibitem[\protect\citeauthoryear{Bickel and Levina}{Bickel and
  Levina}{2004}]{MR2108040}
Bickel, P.~J. and E.~Levina (2004).
\newblock Some theory of {F}isher's linear discriminant function, `naive
  {B}ayes', and some alternatives when there are many more variables than
  observations.
\newblock {\em Bernoulli\/}~{\em 10\/}(6), 989--1010.

\bibitem[\protect\citeauthoryear{Birg{\'e}}{Birg{\'e}}{2001}]{MR1836557}
Birg{\'e}, L. (2001).
\newblock An alternative point of view on {L}epski's method.
\newblock In {\em State of the art in probability and statistics ({L}eiden,
  1999)}, Volume~36 of {\em IMS Lecture Notes Monogr. Ser.}, pp.\  113--133.
  Beachwood, OH: Inst. Math. Statist.

\bibitem[\protect\citeauthoryear{Birg{\'e} and Massart}{Birg{\'e} and
  Massart}{1998}]{1998:birge}
Birg{\'e}, L. and P.~Massart (1998).
\newblock Minimum contrast estimators on sieves: exponential bounds and rates
  of convergence.
\newblock {\em Bernoulli\/}~{\em 4\/}(3), 329--375.

\bibitem[\protect\citeauthoryear{Birnbaum, Johnstone, Nadler, and
  Paul}{Birnbaum et~al.}{2013}]{birnbaum2013minimax}
Birnbaum, A., I.~M. Johnstone, B.~Nadler, and D.~Paul (2013).
\newblock Minimax bounds for sparse pca with noisy high-dimensional data.
\newblock {\em The Annals of Statistics\/}~{\em 41\/}(3), 1055--1084.

\bibitem[\protect\citeauthoryear{Boucheron, Bousquet, Lugosi, and
  Massart}{Boucheron et~al.}{2005}]{MR2123200}
Boucheron, S., O.~Bousquet, G.~Lugosi, and P.~Massart (2005).
\newblock Moment inequalities for functions of independent random variables.
\newblock {\em Ann. Probab.\/}~{\em 33\/}(2), 514--560.

\bibitem[\protect\citeauthoryear{Brubaker and Vempala}{Brubaker and
  Vempala}{2008}]{brubaker2008isotropic}
Brubaker, S.~C. and S.~S. Vempala (2008).
\newblock Isotropic pca and affine-invariant clustering.
\newblock In {\em Building Bridges}, pp.\  241--281. Springer.

\bibitem[\protect\citeauthoryear{Cai, Ma, and Wu}{Cai
  et~al.}{2013a}]{cai2013optimal}
Cai, T., Z.~Ma, and Y.~Wu (2013a).
\newblock Optimal estimation and rank detection for sparse spiked covariance
  matrices.
\newblock {\em Probability Theory and Related Fields\/}, 1--35.

\bibitem[\protect\citeauthoryear{Cai, Ma, and Wu}{Cai
  et~al.}{2013b}]{cai2013sparse}
Cai, T.~T., Z.~Ma, and Y.~Wu (2013b).
\newblock Sparse {PCA}: Optimal rates and adaptive estimation.
\newblock {\em The Annals of Statistics\/}~{\em 41\/}(6), 3074--3110.

\bibitem[\protect\citeauthoryear{Cand\`es and Tao}{Cand\`es and
  Tao}{2005}]{CT07}
Cand\`es, E. and T.~Tao (2005).
\newblock {The Dantzig selector: Statistical estimation when p is much larger
  than n}.
\newblock {\em Annals of Statistics\/}~{\em 35\/}(6), 2313--2351.

\bibitem[\protect\citeauthoryear{Chan and Hall}{Chan and Hall}{2010}]{chan10}
Chan, Y. and P.~Hall (2010).
\newblock Using evidence of mixed populations to select variables for
  clustering very high-dimensional data.
\newblock {\em J. Amer. Statist. Assoc.\/}~{\em 105\/}(490), 798--809.

\bibitem[\protect\citeauthoryear{Chang}{Chang}{1983}]{chang83}
Chang, W.-C. (1983).
\newblock On using principal components before separating a mixture of two
  multivariate normal distributions.
\newblock {\em J. Roy. Statist. Soc. Ser. C\/}~{\em 32}, 267---275.

\bibitem[\protect\citeauthoryear{Chaudhuri, Dasgupta, and Vattani}{Chaudhuri
  et~al.}{2009}]{chaudhuri2009learning}
Chaudhuri, K., S.~Dasgupta, and A.~Vattani (2009).
\newblock Learning mixtures of gaussians using the k-means algorithm.
\newblock {\em arXiv preprint arXiv:0912.0086\/}.

\bibitem[\protect\citeauthoryear{Chen, Donoho, and Saunders}{Chen
  et~al.}{1998}]{Chen_Donoho_Saunders98}
Chen, S.~S., D.~L. Donoho, and M.~A. Saunders (1998).
\newblock Atomic decomposition by basis pursuit.
\newblock {\em SIAM J. Sci. Comput.\/}~{\em 20\/}(1), 33--61 (electronic).

\bibitem[\protect\citeauthoryear{{d'Aspremont}, {El Ghaoui}, Jordan, and
  Lanckriet}{{d'Aspremont} et~al.}{2007}]{aspremont}
{d'Aspremont}, A., L.~{El Ghaoui}, M.~I. Jordan, and G.~R.~G. Lanckriet (2007).
\newblock A direct formulation for sparse pca using semidefinite programming.
\newblock {\em SIAM Review\/}~{\em 49\/}(3), 434--448.

\bibitem[\protect\citeauthoryear{Davidson and Szarek}{Davidson and
  Szarek}{2001}]{MR1863696}
Davidson, K.~R. and S.~J. Szarek (2001).
\newblock Local operator theory, random matrices and {B}anach spaces.
\newblock In {\em Handbook of the geometry of {B}anach spaces, {V}ol. {I}},
  pp.\  317--366. Amsterdam: North-Holland.

\bibitem[\protect\citeauthoryear{Donoho and Jin}{Donoho and
  Jin}{2009}]{MR2546396}
Donoho, D. and J.~Jin (2009).
\newblock Feature selection by higher criticism thresholding achieves the
  optimal phase diagram.
\newblock {\em Philos. Trans. R. Soc. Lond. Ser. A Math. Phys. Eng.
  Sci.\/}~{\em 367\/}(1906), 4449--4470.
\newblock With electronic supplementary materials available online.

\bibitem[\protect\citeauthoryear{Fan and Peng}{Fan and Peng}{2004}]{MR2065194}
Fan, J. and H.~Peng (2004).
\newblock Nonconcave penalized likelihood with a diverging number of
  parameters.
\newblock {\em Ann. Statist.\/}~{\em 32\/}(3), 928--961.

\bibitem[\protect\citeauthoryear{Friedman and Meulman}{Friedman and
  Meulman}{2004}]{friedman04}
Friedman, J. and J.~Meulman (2004).
\newblock Clustering objects on subsets of attributes.
\newblock {\em J. Roy. Statist. Soc. Ser. B\/}~{\em 66}, 815---849.

\bibitem[\protect\citeauthoryear{{Hardt} and {Price}}{{Hardt} and
  {Price}}{2014}]{2014arXiv1404.4997H}
{Hardt}, M. and E.~{Price} (2014, April).
\newblock {Tight bounds for learning a mixture of two gaussians}.
\newblock {\em ArXiv e-prints\/}.

\bibitem[\protect\citeauthoryear{Hastie, Tibshirani, and Friedman}{Hastie
  et~al.}{2009}]{ESL}
Hastie, T., R.~Tibshirani, and J.~Friedman (2009).
\newblock {\em {The elements of statistical learning}}.
\newblock New York: Springer.

\bibitem[\protect\citeauthoryear{Hsu and Kakade}{Hsu and
  Kakade}{2013}]{hsu2013learning}
Hsu, D. and S.~M. Kakade (2013).
\newblock Learning mixtures of spherical gaussians: moment methods and spectral
  decompositions.
\newblock In {\em Proceedings of the 4th conference on Innovations in
  Theoretical Computer Science}, pp.\  11--20. ACM.

\bibitem[\protect\citeauthoryear{Ingster, Pouet, and Tsybakov}{Ingster
  et~al.}{2009}]{MR2546395}
Ingster, Y.~I., C.~Pouet, and A.~B. Tsybakov (2009).
\newblock Classification of sparse high-dimensional vectors.
\newblock {\em Philos. Trans. R. Soc. Lond. Ser. A Math. Phys. Eng.
  Sci.\/}~{\em 367\/}(1906), 4427--4448.

\bibitem[\protect\citeauthoryear{Ji and Jin}{Ji and Jin}{2012}]{ji2010ups}
Ji, P. and J.~Jin (2012).
\newblock U{PS} delivers optimal phase diagram in high-dimensional variable
  selection.
\newblock {\em Ann. Statist.\/}~{\em 40\/}(1), 73--103.

\bibitem[\protect\citeauthoryear{Jin}{Jin}{2009}]{MR2520682}
Jin, J. (2009).
\newblock Impossibility of successful classification when useful features are
  rare and weak.
\newblock {\em Proc. Natl. Acad. Sci. USA\/}~{\em 106\/}(22), 8859--8864.

\bibitem[\protect\citeauthoryear{Jin, Ke, and Wang}{Jin
  et~al.}{2015}]{jin2015phase}
Jin, J., Z.~T. Ke, and W.~Wang (2015).
\newblock Phase transitions for high dimensional clustering and related
  problems.
\newblock {\em arXiv preprint arXiv:1502.06952\/}.

\bibitem[\protect\citeauthoryear{Jin and Wang}{Jin and
  Wang}{2014}]{jin2014important}
Jin, J. and W.~Wang (2014).
\newblock Important feature pca for high dimensional clustering.
\newblock {\em arXiv preprint arXiv:1407.5241\/}.

\bibitem[\protect\citeauthoryear{Johnstone and Lu}{Johnstone and
  Lu}{2009}]{johnstone2009consistency}
Johnstone, I.~M. and A.~Y. Lu (2009).
\newblock On consistency and sparsity for principal components analysis in high
  dimensions.
\newblock {\em Journal of the American Statistical Association\/}~{\em
  104\/}(486), 682--693.

\bibitem[\protect\citeauthoryear{Kalai, Moitra, and Valiant}{Kalai
  et~al.}{2012}]{kalai2012disentangling}
Kalai, A.~T., A.~Moitra, and G.~Valiant (2012).
\newblock Disentangling gaussians.
\newblock {\em Communications of the ACM\/}~{\em 55\/}(2), 113--120.

\bibitem[\protect\citeauthoryear{Laurent and Massart}{Laurent and
  Massart}{2000}]{MR1805785}
Laurent, B. and P.~Massart (2000).
\newblock Adaptive estimation of a quadratic functional by model selection.
\newblock {\em Ann. Statist.\/}~{\em 28\/}(5), 1302--1338.

\bibitem[\protect\citeauthoryear{Ledoux}{Ledoux}{1996}]{MR1600888}
Ledoux, M. (1996).
\newblock Isoperimetry and {G}aussian analysis.
\newblock In {\em Lectures on probability theory and statistics
  ({S}aint-{F}lour, 1994)}, Volume 1648 of {\em Lecture Notes in Math.}, pp.\
  165--294. Berlin: Springer.

\bibitem[\protect\citeauthoryear{Malkovich and Afifi}{Malkovich and
  Afifi}{1973}]{malkovich1973tests}
Malkovich, J.~F. and A.~Afifi (1973).
\newblock On tests for multivariate normality.
\newblock {\em Journal of the American Statistical Association\/}~{\em
  68\/}(341), 176--179.

\bibitem[\protect\citeauthoryear{Mallat and Zhang}{Mallat and
  Zhang}{1993}]{Mallat_Zhang93}
Mallat, S. and Z.~Zhang (1993).
\newblock Matching pursuit with time-frequency dictionaries.
\newblock {\em IEEE Trans. Image Process.\/}~{\em 41}, 3397--3415.

\bibitem[\protect\citeauthoryear{Mallows}{Mallows}{1973}]{mallows1973some}
Mallows, C. (1973).
\newblock Some comments on cp.
\newblock {\em Technometrics\/}~{\em 15}, 661--675.

\bibitem[\protect\citeauthoryear{Mardia}{Mardia}{1970}]{MR0397994}
Mardia, K.~V. (1970).
\newblock Measures of multivariate skewness and kurtosis with applications.
\newblock {\em Biometrika\/}~{\em 57\/}(3), 519--530.

\bibitem[\protect\citeauthoryear{Massart}{Massart}{2007}]{MR2319879}
Massart, P. (2007).
\newblock {\em Concentration inequalities and model selection}, Volume 1896 of
  {\em Lecture Notes in Mathematics}.
\newblock Berlin: Springer.
\newblock Lectures from the 33rd Summer School on Probability Theory held in
  Saint-Flour, July 6--23, 2003, With a foreword by Jean Picard.

\bibitem[\protect\citeauthoryear{Maugis and Michel}{Maugis and
  Michel}{2011}]{MR2870505}
Maugis, C. and B.~Michel (2011).
\newblock A non asymptotic penalized criterion for {G}aussian mixture model
  selection.
\newblock {\em ESAIM Probab. Stat.\/}~{\em 15}, 41--68.

\bibitem[\protect\citeauthoryear{Pan and Shen}{Pan and Shen}{2007}]{pan07}
Pan, W. and X.~Shen (2007).
\newblock Penalized model-based clustering with application to variable
  selection.
\newblock {\em J. Mach. Learn. Res.\/}~{\em 8}, 1145---1164.

\bibitem[\protect\citeauthoryear{Raftery and Dean}{Raftery and
  Dean}{2006}]{raftery06}
Raftery, A. and N.~Dean (2006).
\newblock Variable selection for model-based clustering.
\newblock {\em J. Amer. Statist. Assoc.\/}~{\em 101}, 168--178.

\bibitem[\protect\citeauthoryear{Schwarz}{Schwarz}{1978}]{schwarz78}
Schwarz, G. (1978).
\newblock Estimating the dimension of a model.
\newblock {\em Ann. Statist.\/}~{\em 6\/}(2), 461--464.

\bibitem[\protect\citeauthoryear{Srivastava}{Srivastava}{1984}]{MR777837}
Srivastava, M.~S. (1984).
\newblock A measure of skewness and kurtosis and a graphical method for
  assessing multivariate normality.
\newblock {\em Statist. Probab. Lett.\/}~{\em 2\/}(5), 263--267.

\bibitem[\protect\citeauthoryear{Tibshirani}{Tibshirani}{1996}]{Ti1996}
Tibshirani, R. (1996).
\newblock Regression shrinkage and selection via the lasso.
\newblock {\em J. Roy. Statist. Soc. Ser. B\/}~{\em 58\/}(1), 267--288.

\bibitem[\protect\citeauthoryear{Tropp}{Tropp}{2004}]{Tropp04}
Tropp, J.~A. (2004).
\newblock Greed is good: algorithmic results for sparse approximation.
\newblock {\em IEEE Trans. Info. Theory\/}~{\em 50\/}(10), 2231--2242.

\bibitem[\protect\citeauthoryear{Tsybakov}{Tsybakov}{2009}]{MR2724359}
Tsybakov, A.~B. (2009).
\newblock {\em Introduction to nonparametric estimation}.
\newblock Springer Series in Statistics. New York: Springer.
\newblock Revised and extended from the 2004 French original, Translated by
  Vladimir Zaiats.

\bibitem[\protect\citeauthoryear{van~der Vaart and Wellner}{van~der Vaart and
  Wellner}{1996}]{MR1385671}
van~der Vaart, A.~W. and J.~A. Wellner (1996).
\newblock {\em Weak convergence and empirical processes}.
\newblock Springer Series in Statistics. New York: Springer-Verlag.
\newblock With applications to statistics.

\bibitem[\protect\citeauthoryear{Vershynin}{Vershynin}{2010}]{vershynin}
Vershynin, R. (2010).
\newblock Introduction to the non-asymptotic analysis of random matrices.
\newblock Available from \url{http://arxiv.org/abs/1011.3027}.

\bibitem[\protect\citeauthoryear{Verzelen}{Verzelen}{2012}]{verzelen}
Verzelen, N. (2012).
\newblock Minimax risks for sparse regressions: Ultra-high dimensional
  phenomenons.
\newblock {\em Electron. J. Stat.\/}~{\em 6}, 38--90.

\bibitem[\protect\citeauthoryear{von Bahr}{von Bahr}{1965}]{bahr}
von Bahr, B. (1965).
\newblock On the convergence of moments in the central limit theorem.
\newblock {\em Ann. Math. Statist.\/}~{\em 36}, 808--818.

\bibitem[\protect\citeauthoryear{Vu and Lei}{Vu and Lei}{2012}]{vu2012minimax}
Vu, V.~Q. and J.~Lei (2012).
\newblock Minimax rates of estimation for sparse pca in high dimensions.
\newblock In {\em International Conference on Artificial Intelligence and
  Statistics}, pp.\  1278--1286.

\bibitem[\protect\citeauthoryear{Vu and Lei}{Vu and Lei}{2013}]{vu2013minimax}
Vu, V.~Q. and J.~Lei (2013).
\newblock Minimax sparse principal subspace estimation in high dimensions.
\newblock {\em The Annals of Statistics\/}~{\em 41\/}(6), 2905--2947.

\bibitem[\protect\citeauthoryear{Wang and Zhu}{Wang and Zhu}{2008}]{wang08}
Wang, S. and J.~Zhu (2008).
\newblock Variable selection for model-based high- dimensional clustering and
  its application to microarray data.
\newblock {\em Biometrics\/}~{\em 64}, 440---448.

\bibitem[\protect\citeauthoryear{Witten and Tibshirani}{Witten and
  Tibshirani}{2010}]{witten2010framework}
Witten, D.~M. and R.~Tibshirani (2010).
\newblock A framework for feature selection in clustering.
\newblock {\em Journal of the American Statistical Association\/}~{\em
  105\/}(490), 713--726.

\bibitem[\protect\citeauthoryear{Xie, Pan, and Shen}{Xie et~al.}{2008}]{xie08}
Xie, B., W.~Pan, and X.~Shen (2008).
\newblock Penalized model-based clusteringwith cluster-specific diagonal
  covariance matrices and grouped variables.
\newblock {\em Electron. J. Stat.\/}~{\em 2}, 168---212.

\bibitem[\protect\citeauthoryear{Zhu and Hastie}{Zhu and Hastie}{2004}]{zhu04}
Zhu, J. and T.~Hastie (2004).
\newblock Classification of gene microarrays by penalized logistic regression.
\newblock {\em Biostatistics\/}~{\em 5\/}(2), 427--443.

\end{thebibliography}


\end{document}